\documentclass{amsart}
\allowdisplaybreaks[4]
\usepackage{graphicx}
\usepackage{color}

\newcommand{\R}{\mathbf{R}}

\newcommand{{\ba}}{\bf a}
\newcommand{\ve}{\varepsilon}
\newcommand{\la}{\lambda}

\newcommand{\ga}{\gamma}
\newcommand{\Ga}{\Gamma}
\newcommand{\pa}{\partial}
\newcommand{\ra}{\rightarrow}

\newcommand{\del}{\delta}

\newcommand{\al}{\alpha}
\newcommand{\e}{\mathrm{e}}

\newcommand{\be}{\begin{equation}}
\newcommand{\ee}{\end{equation}}

\newtheorem{lem}{Lemma}{\bf}{\it}
{\it}{\rm}
\newtheorem{rem}{Remark}{\it}{\rm}
{\it}{\rm}

\newtheorem{theorem}{Theorem}
\newtheorem{proposition}{Proposition}

\numberwithin{theorem}{section}
\numberwithin{lem}{section}
\numberwithin{equation}{section}
\numberwithin{proposition}{section}
\numberwithin{corollary}{section}
\title[LSW Model]{On Global Asymptotic Stability for the LSW Model with subcritical initial data}

\author{Joseph G. Conlon and Michael Dabkowski }

\address{(Joseph G. Conlon): University of Michigan\\ Department of Mathematics\\ Ann Arbor,
  MI 48109-1109}
\email{conlon@umich.edu}

\address{(Michael Dabkowski): University of Michigan-Dearborn \\ Department of Mathematics and Statistics\\ Dearborn,
  MI 48128}
\email{mgdabkow@umich.edu}

\keywords{nonlinear pde, coarsening}
\subjclass{35F05,  82C70, 82C26}

\begin{document}

\maketitle

\begin{abstract}
The main result of the paper is a global asymptotic stability result for solutions to the 
Lifschitz-Slyozov-Wagner (LSW) system of equations. This extends  some local asymptotic stability results of Niethammer-Vel\'{a}zquez (2006). The method of proof is along similar lines to the one used in a previous paper  of the authors. This previous paper proves global asymptotic stability for a class of infinite dimensional dynamical systems for which no Lyapounov function is (apparently) available.
\end{abstract}

\section{Introduction.}
In this paper we shall be principally concerned with studying the large time behavior of solutions to the  Lifschitz-Slyozov-Wagner (LSW) equations \cite{ls,w}. The LSW equations occur in a variety of contexts \cite{pego, penrose} as a mean field approximation for the evolution of particle clusters of various volumes. Clusters of volume $x>0$ have density $c(x,t)\ge 0$ at time $t>0$. The density evolves according to a linear law, subject to the linear mass conservation constraint as follows:
\begin{eqnarray}
\frac{\pa c(x,t)}{\pa t}&=& \frac{\pa}{\pa x}\Big( \Big( 1- \big( x L^{-1}(t)\big)^{\alpha} \Big) c(x,t)\Big),  \quad x>0, \label{U1}\\
\int_0^\infty x c(x,t) dx&=& 1. \label{V1} 
\end{eqnarray}
We shall require the parameter $\al$ in (\ref{U1}) to lie in the interval $0<\al\le 1$, whence the system (\ref{U1}), (\ref{V1}) includes the standard LSW model corresponding to $\al=1/3$, and the much simpler Carr-Penrose model \cite{cp} corresponding to $\al=1$. 

One wishes then to solve (\ref{U1}) for $t>0$ and initial condition $c(x,0) = c_0(x) \ge 0, \ x > 0$, subject to the constraint (\ref{V1}).  The parameter $L(t) > 0$ in (\ref{U1}) is determined by the constraint (\ref{V1}) and is therefore given by the formula,
\be \label{W1}
L(t)^{\al} = \int^\infty_0 \ x^{\al} c(x,t)dx \Big/ \int^\infty_0 c(x,t) dx.
\ee
 Existence and uniqueness of solutions to (\ref{U1}), (\ref{V1}) with given initial data $c_0(\cdot)$ satisfying the constraint has been proven in 
\cite{laurencot}
for integrable functions $c_0(\cdot)$, and in \cite{np2} for initial data such that $c_0(x)dx$ is an arbitrary Borel probability measure with compact support. In \cite{np3}   the methods of \cite{np2} are further developed to prove existence and uniqueness for initial data such that $c_0(x)dx$ is a Borel probability measure with finite first moment.

The system (\ref{U1}), (\ref{V1}) can be interpreted as an evolution equation for the probability density function (pdf) of random variables. Thus let us assume that the initial data $c_0(x)\ge 0, \ x\ge0,$ for (\ref{U1}), (\ref{V1}) satisfies $\int_0^\infty c_0(x) \ dx<\infty$. The conservation law (\ref{V1}) implies that the mean $\langle X_0\rangle$ of $X_0$ is finite, and this is the only absolute requirement on the variable $X_0$. If for $t>0$ the variable $X_t$ has pdf $c(\cdot,t)/\int_0^\infty c(x,t) \ dx$ then (\ref{U1}) with $L(t)=\langle X_t^\al\rangle$ is an evolution equation for the pdf of $X_t$. We can also see that
\be \label{X1}
\frac{d}{dt}\langle X_t\rangle \ = \  c(0,t)\Big/\left[\int_0^\infty c(x,t) \ dx\right]^2 \ ,
\ee
whence the function $t\ra \langle X_t\rangle$ is increasing.  In \cite{c2} it was shown that for a wide range of initial data $c_0(\cdot)$, there exists positive constants $C_1,C_2$, depending only on the initial data,  such that
\be \label{Y1}
C_1T \ \le \langle X_T\rangle \ \le \ C_2T \quad {\rm for  \ \ }T\ge 1 \ .
\ee
We shall show here that for initial data variables $X_0$ which are {\it subcritical} \cite{cn}, the limit $\lim_{T\ra\infty}   \langle X_T\rangle/T$ exists, and may be computed in terms of the corresponding self-similar solution to (\ref{U1}), (\ref{V1}). 

Let $X$ be a non-negative random variable. With $X$ we may associate a {\it beta function} \cite{c2} with domain $[0,\|X\|_\infty)$ defined by
\be \label{Z1}
\beta_X(x) \ = \ \frac{d}{dx}E[X-x \ | \ X>x]+1 \ = \ \frac{c_X(x)h_X(x)}{w_X(x)^2} \ ,
\ee
 where the functions $c_X(\cdot),w_X(\cdot),h_X(\cdot)$ are given by 
 \be \label{AA1}
 c_X(\cdot) \ = \ {\rm pdf \ of \ } X, \quad w_X(x)=P(X>x), \quad h_X(x)=E[X-x; \ X>x] \ .
 \ee
 The variable $X$ is said to be subcritical if $\lim_{x\ra\|X\|_\infty}\beta_X(x)=\beta_0$ exists and $0<\beta_0<1$. It is easy  to see from (\ref{Z1}) that if $X$ is subcritical then  $\|X\|_\infty<\infty$.  Furthermore, $w_X(x)\simeq (\|X\|_\infty-x)^p$ as $x\ra\|X\|_\infty$, where $\beta_0=p/(p+1)$.  
 
 For each $\beta$ in the interval $0<\beta\le 1$ the system (\ref{U1}), (\ref{V1}) has a unique self-similar solution. The corresponding random variable  $\mathcal{X}_\beta$ , normalized to have mean $1$, has the property  $\lim_{x\ra\|\mathcal{X}_\beta\|_\infty}\beta_{\mathcal{X}_\beta}(x)=\beta$. The solution to (\ref{U1}), (\ref{V1}) with initial data variable $X_0$ a constant times $\mathcal{X}_\beta$  is then given from (\ref{X1}), (\ref{Z1}) by
 \be \label{AB1}
X_t \ = \  \langle X_t\rangle \mathcal{X}_\beta \ , \quad \frac{d}{dt}  \langle X_t\rangle \ = \ \beta_{\mathcal{X}_\beta}(0) \ .
 \ee
 Our main result is that (\ref{AB1}) holds approximately at large time if the initial data for (\ref{U1}), (\ref{V1}) is subcritical.
 \begin{theorem}
 Assume the initial data  for (\ref{U1}), (\ref{V1}) has compact support, and that the corresponding random variable $X_0$ has continuous beta function $\beta_{X_0}:[0,\|X_0\|_\infty)\ra\R$  which satisfies $\lim_{x\ra\|X_0\|_\infty}\beta_{X_0}(x)=\beta$  with $0<\beta<1$. Let $X_t$ be the random variable corresponding to the solution $c(\cdot,t)$ of (\ref{U1}), (\ref{V1}) at time $t>0$.  Then
 \be \label{AC1}
 \frac{X_t}{\langle X_t\rangle} \xrightarrow{D} \mathcal{X}_\beta \ \ {\rm as \ \ } t\ra\infty \ , \quad 
 \lim_{t\ra\infty} \frac{d}{dt}  \langle X_t\rangle \ = \ \beta_{\mathcal{X}_\beta}(0) \ ,
 \ee
 where $\xrightarrow{D}$ denotes convergence in distribution. 
 
  If the function $x\ra \beta_{X_0}(x)$ is H\"{o}lder continuous at $x=\|X_0\|_\infty$,  then there exist constants $C,\nu>0$ such that 
 \be \label{AD1}
\left|\frac{d}{dt}\langle X_t\rangle - \beta_{\mathcal{X}_\beta}(0)\right| \ \le \ \frac{C}{1+t^\nu} \ , \quad t\ge 0 \ .
\ee
The cdf of $X_t/\langle X_t\rangle$ has a corresponding rate of convergence to the cdf of $\mathcal{X}_\beta$. In particular, for any $\del$ with $0<\del<1$, one has
 \begin{multline} \label{AE1}
 P\left(\mathcal{X}_\beta>x[1+C_\del/(1+ t^\nu)]\right) \ \le \  P\left(\frac{X_t}{\langle X_t\rangle}>x\right) \\
  \le \ P\left(\mathcal{X}_\beta>\frac{x}{1+C_\del/(1+ t^\nu)}\right) \ , \quad
 {\rm for \ } 0\le x\le (1-\del)\|\mathcal{X}_\beta\|_\infty \ , \ t\ge 0 \ ,
 \end{multline}
 where $C_\del>0$ is a constant depending on $\del$. 
 \end{theorem}
 Previous results imply that (\ref{AC1}) holds in certain cases. In particular, the results of \cite{cp} imply that (\ref{AC1}) holds when $\al=1$. For $\al=1/3$ Theorem 4.1 of \cite{nv1} implies that (\ref{AC1}) holds if $\beta<1$ is sufficiently small. Theorem 3.2 of \cite{nv1} implies that (\ref{AC1}) holds for any $\beta<1$, provided the initial data random variable $X_0$ for (\ref{U1}), (\ref{V1}) satisfies  the additional  condition
 $\|\beta_{X_0}(\|\mathcal{X}_\beta\|^{-1}_\infty\|X_0\|_\infty\cdot)-\beta_{\mathcal{X}_\beta}(\cdot)\|_\infty<\del(\beta)$, where $\del(\beta)>0$ depends on $\beta$ (see Remark 7 in $\S8$).

 The first step in proving Theorem 1.1 is to observe following \cite{np1} that if the initial data 
 $c_0(\cdot)$ for (\ref{U1}), (\ref{V1}) has compact support then $c(\cdot,t)$  has compact support for all $t>0$. Then by means of a transformation (see $\S8$) one can normalize the support for all time to the interval $[0,1]$. In that case the LSW equations may be written as
\begin{equation} \label{A1}
\frac{\pa w(x,t)}{\pa t}+ \left[\phi(x)-\kappa(t)\psi(x)\right] \frac{\pa w(x,t)}{\pa x}=w(x,t), \quad 0\le x<1, \ t\ge 0,
\end{equation}
with the mass conservation law
\begin{equation} \label{B1}
 \int_{0}^1  w(x,t)  \ dx=1 \  ,  \quad t\ge 0.
\end{equation} 
The functions $\phi(\cdot), \ \psi(\cdot)$ are given by the formulae
\be \label{C1}
\phi(x) \ = \ x^\al-x, \quad \psi(x) \ = \ 1-x^\al \ , \quad 0\le x\le 1 \ .
\ee
 We shall study solutions to (\ref{A1}), (\ref{B1})  for more general $\phi(\cdot) \ \psi(\cdot)$ than (\ref{C1}) with $0<\al\le1$. In particular, we require $\phi(\cdot) \ \psi(\cdot)$ to satisfy the conditions
\begin{eqnarray}
\quad &\phi(x) \  {\rm is \  concave \  and \  satisfies \ }& \quad \phi(0)=\phi(1)=0, \ -1< \phi'(1)<0.  \label{D1}\\
\quad &\psi(x)  \  {\rm is \  convex \  and \  satisfies \ }&\quad \psi(1)=0,  \ \  \psi'(1)<0, \ \psi''(1)-\phi''(1)>0.  \label{E1}
\end{eqnarray}

The fundamental quantity of study in the generalized LSW model (\ref{A1}), (\ref{B1}) is the (renormalized) cluster density $c(x,t)\ge 0$ at volume $x\in[0,1]$ and time $t>0$.  The function $w(\cdot,t)$ in (\ref{A1}), (\ref{B1}) is the integral of $c(\cdot,t)$ and is therefore non-negative decreasing with $\lim_{x\ra 1}w(x,t)=0$.  We define $w(\cdot,t)$ in terms of $c(\cdot,t)$, and in addition the function $h(\cdot,t)$ by
\be \label{F1}
w(x,t) \ = \ \int_x^1c(x',t) \ dx' \ , \quad h(x,t) \ = \ \int_x^1 w(x',t) \ dx' \ , \quad 0\le x<1 \ .
\ee
The conservation law (\ref{B1}) is evidently the same as $h(0,t)=1$. Conditions on the initial data for (\ref{A1}), (\ref{B1}) are given in terms of the {\it beta function} for $c(\cdot,t)$ defined by
\be \label{G1}
\beta(x,t) \ = \ \frac{c(x,t)h(x,t)}{w(x,t)^2} \ , \quad 0\le x\le 1 \ .
\ee
We assume that $\beta(\cdot,0)$ has the property
\be \label{H1}
\lim_{x\ra 1}\beta(x,0) \ = \ \beta_0>0 \ .
\ee
It is easy to see that $\beta_0\le 1$, and we say that the initial data for (\ref{A1}), (\ref{B1}) is {\it subcritical} if $\beta_0<1$.  The main result of \cite{cn} is a {\it weak} asymptotic stability result for solutions to (\ref{A1}), (\ref{B1}) with initial data satisfying (\ref{H1}). Further assumptions on the functions $\phi(\cdot), \ \psi(\cdot)$ beyond (\ref{D1}), (\ref{E1}) are required to prove this.  In the case of subcritical initial data these are as follows:
 \be \label{I1}
 \phi(\cdot), \ \psi(\cdot)  \ {\rm are \ } C^3 \ {\rm on \ } (0,1]  \ {\rm \ and \ } \phi'''(x)\ge0, \  \psi'''(x)\le 0  \   {\rm for \ } 0<x\le 1.
 \ee
 Evidently (\ref{I1}) holds for the functions (\ref{C1}) with $0<\al\le1$. 
\begin{theorem} Let $w(x,t), \ x,t\ge 0$, be the solution to (\ref{A1}), (\ref{B1}) with coefficients satisfying (\ref{D1}), (\ref{E1}) and assume that the initial data $w(\cdot,0)$ has beta function $\beta(\cdot,0)$ satisfying (\ref{H1}) with $0<\beta_0<1$.  Then there is a positive constant $C_1$ depending only on the initial data such that $\kappa(t)\ge C_1$ for all $t\ge  0$. If in addition (\ref{I1}) holds, then there is a positive constant $C_2$  depending only on the initial data such that $\kappa(t)\le C_2$ for all $t\ge  0$ and
\be \label{J1}
\lim_{T\ra \infty}\frac{1}{T}\int_0^T \kappa(t) \ dt \ = \ [1/\beta_0-\phi'(1)-1]/|\psi'(1)| \ .
\ee
\end{theorem} 

In \cite{cn} we were also able to prove a {\it strong} asymptotic stability result for solutions to (\ref{A1}), (\ref{B1}) in the case when $\phi(\cdot), \ \psi(\cdot)$ are quadratic:
\begin{theorem} Assume that the functions $\phi(\cdot), \ \psi(\cdot)$ are quadratic, and that  the initial data $w(\cdot,0)$ for (\ref{A1}), (\ref{B1})  has beta function $\beta(\cdot,0)$ satisfying
(\ref{H1}).    Then setting $\kappa=[1/\beta_0-\phi'(1)-1]/|\psi'(1)|$, one has for $\beta_0<1$, 
\be \label{K1}
\lim_{t\ra\infty} \kappa(t ) =   \kappa, \quad \lim_{t\ra\infty} \|\beta(\cdot,t)-\beta_\kappa(\cdot)\|_\infty=0,
\ee
where $\beta_\kappa(\cdot)$ is the beta function of the time independent solution $w_\kappa(\cdot)$ of (\ref{A1}). 
\end{theorem}
Evidently Theorem 1.3 applies to the Carr-Penrose model where $\al=1$.
Our main goal here will be to show that the conclusions of Theorem 1.3  also hold for a class of non-quadratic functions $\phi(\cdot), \ \psi(\cdot)$, which includes the functions (\ref{C1}) with $0<\al<1$.  Theorem 1.1 is then an easy consequence.  

We accomplish this by establishing asymptotic stability results for a class of PDE similar to the one studied in \cite{cd}.
In particular, we consider for some $\ve_0>0$ the evolution PDE
 \be \label{L1}
 \frac{\pa \xi(y,t)}{\pa t} -h(y) -h(y)\frac{\pa\xi(y,t)}{\pa y} +\rho(\xi(\cdot,t))\left[\xi(y,t)-y\frac{\pa \xi(y,t)}{\pa y}\right] \ = \ 0, \quad y>\ve_0, \ t>0,
  \ee
  with given non-negative  initial data $\xi(y,0), \ y> \ve_0$.  We assume as in \cite{cd} that the function $h:[\ve_0,\infty)\ra\R$ has the properties:
  \be \label{M1}
  h(\cdot) \ {\rm is \ continuous  \ positive \ and \ } \lim_{y\ra\infty}h(y)=h_\infty>0 \ .
  \ee
 For an equilibrium of (\ref{L1}) to exist corresponding to $\rho=1/p>0$ we need the equilibrium function $\xi_p(y), \ y>\ve_0,$ to satisfy
  \be \label{N1}
  -h(y) -h(y)\frac{d\xi_p(y)}{d y} +\frac{1}{p}\left[\xi_p( y)-y\frac{d \xi_p(y)}{d y}\right] \ = \ 0, \quad  \lim_{y\ra\infty}\xi_p(y)=ph_\infty \ .
  \ee 
We can easily solve  (\ref{N1}) by using the integrating factor
\be \label{O1}
\tilde{h}(y) \ = \ K\exp\left[\int_{\ve_0}^y\frac{dy'}{ph(y')+y'}\right] \  \quad {\rm for  \ some  \ constant  \ } K  ,
\ee 
where $K$ is chosen so that  $\lim_{y\ra\infty}\tilde{h}(y)/y=1$.  Then we have that
\be \label{P1}
\xi_p(y) \ = \ \tilde{h}(y)\int_y^\infty \frac{ph(y')}{\left(ph(y')+y'\right)\tilde{h}(y')} \ dy' \  .
\ee
It follows from (\ref{M1}), (\ref{P1}) that $\lim_{y\ra\infty}\xi_p(y)=ph_\infty$, whence (\ref{N1}) implies that $\lim_{y\ra\infty}yd\xi_p(y)/dy=0$. Furthermore, if the function $h(\cdot)$ is decreasing then $ph_\infty\le \xi_p(y)\le ph(y)$ for all $y> \ve_0$. 

Following \cite{cd}, we impose conditions on  the functional $\rho(\cdot)$  so that  the equilibrium $\xi_p$ is a global attractor for (\ref{L1}).  To do this we assume there is a positive functional $I(\cdot)$ on continuous positive functions $\zeta:[\ve_0,\infty)\ra\R$ with the property that
\be \label{Q1}
\frac{1}{p}\frac{d}{dt} I(\xi(\cdot,t)) \ = \ \left[\rho(\xi(\cdot,t))-\frac{1}{p}\right]I(\xi(\cdot,t)) \  \quad {\rm for \ solutions \ } \xi(\cdot,t) \ {\rm of \ (\ref{L1})}  \ . 
\ee
From (\ref{Q1}) it follows that if we can show that $\log I(\xi(\cdot,t))$ remains bounded as $t\ra\infty$ then $\rho(\xi(\cdot,t))$ converges as $t\ra\infty$ to $1/p$ in the averaged sense
\be \label{R1}
\lim_{T\ra\infty}\frac{1}{T}\int_0^T \rho(\xi(\cdot,t)) \ dt \ = \ \frac{1}{p} \ .
\ee
Let $[\cdot,\cdot]$ denote the Euclidean inner product on $L^2([\ve_0,\infty))$ and $dI(\zeta(\cdot)):[\ve_0,\infty)\ra\R$ the gradient of the functional $I(\cdot)$ at $\zeta(\cdot)$. Then from(\ref{L1}), (\ref{Q1}) we have that
\be \label{S1}
p\left[\rho(\zeta(\cdot))-\frac{1}{p}\right]I(\zeta(\cdot)) \ = \ [dI(\zeta(\cdot)), h+hD\zeta]+\rho(\zeta(\cdot))[dI(\zeta(\cdot)), yD\zeta-\zeta] \ .
\ee
We conclude from (\ref{S1}) that
\be \label{T1}
\rho(\zeta(\cdot)) \ = \ \frac{I(\zeta(\cdot))+[dI(\zeta(\cdot)), h+hD\zeta]}{pI(\zeta(\cdot))+[dI(\zeta(\cdot)), \zeta-yD\zeta]}  \ .
\ee

In $\S2$ we shall use a transformation introduced in \cite{cn} to relate (\ref{A1}), (\ref{B1}) to (\ref{L1}), (\ref{T1}). The limit (\ref{J1}) of Theorem 1.2 then follows from (\ref{R1}) by showing the boundedness of $\log I(\xi(\cdot,t))$ as $t\ra\infty$.

\vspace{.1in}

\section{Derivation of the  PDE (\ref{L1}) from the LSW Model}
We shall use the transformation introduced in $\S6$ of \cite{cn} to obtain (\ref{L1}) from (\ref{A1}). 
We first recall that the solution $w(x,t)$ of (\ref{A1}) is given by $w(x,t)=e^tw(F(x,t),0)$ where  $F(x,t)$ is the solution to the initial value problem
\begin{eqnarray} \label{A2}
\frac{\pa F(x,t)}{\pa t}+ \left[\phi(x)-\kappa(t)\psi(x)\right] \frac{\pa F(x,t)}{\pa x}&=&0, \quad 0\le x<1, \ t\ge 0,
\\
F(x,0) &=& x, \quad 0\le x<1. \nonumber
\end{eqnarray}
For $t\ge 0$ let $u(t)$ be the function
\be \label{B2}
u(t)= \exp\left[\int_0^t \{\phi'(1)-\psi'(1)\kappa(s)\} \  ds \right] \ ,
\ee
and $f(x), \ 0\le x< 1$, be the function defined by
\be \label{C2}
\frac{d}{dx}\log f(x) \ = \ -\frac{\psi'(1)}{\psi(x)} \ ,  \ 0\le x<1, \qquad \lim_{x\ra 1} (1-x)f(x) \ = \ 1.
\ee
Evidently  $f:[0,1)\ra\R$ is a strictly increasing function satisfying $f(0)>0$ and  $\lim_{x\ra 1} f(x)=\infty$. We define now the domains $\mathcal{D}=\{ (x,u)\in\R^2: 0<x<1, \ u>0\}$ and  $\hat{\mathcal{D}}=\{ (z,u)\in\R^2: z>f(0)u, \ u>0\}$. Then the transformation $(z,u)=(f(x)u, u)$ maps $\mathcal{D}$ to $\hat{\mathcal{D}}$. Let $g:\hat{\mathcal{D}}\ra\R$ be the function
\be \label{D2}
g(z,u) \ = \  uf(x)\left[\phi'(1)-\frac{\psi'(1)\phi(x)}{\psi(x)}\right] \  , \quad z=f(x)u, \ 0<x<1, \  u>0 \ ,
\ee
and $\hat{F}(z,t)$  the solution to the initial value problem
\begin{eqnarray} \label{E2}
\frac{\pa \hat{F}(z,t)}{\pa t}+ g(z,u(t))\frac{\pa \hat{F}(z,t)}{\pa z}&=&0, \quad z>f(0)u(t), \ t\ge 0,
\\
\hat{F}(z,0) &=& z, \quad z>f(0). \nonumber
\end{eqnarray}
The solutions $F(x,t)$ of (\ref{A2}) and $\hat{F}(z,t)$ of (\ref{E2}) are related by the identity
\be \label{F2}
f(F(x,t)) \ = \ \hat{F}( f(x)u(t),t) \ , \quad 0<x<1,  \  t>0. 
\ee
It follows from (\ref{D1}), (\ref{E1}) (\ref{C2}), (\ref{D2}) that the function $g(\cdot,u)$ is negative and
\be \label{G2}
-\lim_{z\ra\infty}\frac{g(z,u)}{u} \ = \  \frac{\psi''(1)\phi'(1)-\psi'(1)\phi''(1)}{2\psi'(1)} \ = \ \al_0>0 \ .
\ee
In the case of $\phi(\cdot), \ \psi(\cdot)$  being quadratic functions then $-g(\cdot,u)/u\equiv\al_0$. 

We write
\be \label{H2}
\hat{F}(z,t) \ = \ z+\al_0u(t)\xi\left(\frac{z}{\al_0 u(t)} ,t\right) \ , \quad z>f(0)u(t) \ .
\ee
Then setting $y=z/\al_0u(t)=f(x)/\al_0$ we see from (\ref{E2}) that $\xi(y,t)$ satisfies the PDE (\ref{L1}) in the domain $\{(y,t): \ y>f(0)/\al_0, \ t>0\}$, where the functions $h(\cdot), \rho(\cdot)$ are given by the formulae
\be \label{I2}
h(y) \ = \ \frac{f(x)}{\al_0}\left[\frac{\psi'(1)\phi(x)}{\psi(x)}-\phi'(1)\right] \ , \quad \rho(\xi(\cdot,t)) \ = \ \frac{1}{u(t)}\frac{du(t)}{dt} \ .
\ee
Evidently for the function $h(\cdot)$ of (\ref{I2}) one has $h_\infty=1$ in (\ref{M1}). 
Next for any $p>0$ we consider (\ref{A1}), (\ref{B1}) with initial data $w(\cdot,0)$ given by
\be \label{J2}
w(x,0) \ = \ \frac{K_p}{f(x)^p} \ , \quad {\rm where \ }  K_p\int_0^1\frac{dx}{f(x)^p} \ = \ 1 \ .
\ee
It is easy to see that the $\beta_0$ of (\ref{H1}) corresponding to (\ref{J2}) is given by $\beta_0=p/(1+p)$.  The conservation law for solutions to (\ref{A1}), (\ref{B1}) becomes then
\be \label{K2}
1 \ = \ e^t\int_0^1w(F(x,t),0) \ dx \ = \ \
e^tK_p\int_0^1\frac{dx}{f(F(x,t))^p} \ 
= \ e^tK_p\int_0^1\frac{dx}{\hat{F}(f(x)u(t),t)^p}  \ ,
\ee
where we have used (\ref{F2}).  Let $I(\cdot)$ be the functional defined by
\be \label{L2}
I(\zeta(\cdot)) \ = \ \frac{K_p}{\al_0^p}\int_{\ve_0}^\infty \frac{a(y)}{[b(y)+\zeta(y)]^p} \ dy \ ,
\ee
where $a(\cdot), b(\cdot), \ \ve_0$ are given by the formulae
\be \label{M2}
a(y) \ = \ \frac{dx}{dy} \ = \ -\frac{\psi(x)}{y\psi'(1)} \ , \quad b(y) \ = \ y \ , \quad \ve_0 \ = \ f(0)/\al_0 \ .
\ee
Observe that $\lim_{y\ra\infty} y^2a(y)=1/\al_0$, whence the integral in (\ref{L2}) is finite for all non-negative $\zeta:(0,\infty)\ra\R$.  
Evidently (\ref{K2}) is the same as 
\be \label{N2}
\frac{e^t}{u(t)^p}I(\xi(\cdot,t)) \ = \ 1 \ .
\ee
Differentiating (\ref{N2}), we see using the formula for $\rho(\cdot)$ in (\ref{I2}) that (\ref{Q1}) holds. We conclude that the function $\xi(\cdot,t), \ t>0,$ defined in (\ref{H2}) is  a solution to the evolution equation (\ref{L1}) with $h(\cdot)$ given by (\ref{I2}), $\rho(\cdot)$  by (\ref{T1}), and $I(\cdot)$ by (\ref{L2}), (\ref{M2}). The initial data for (\ref{L1}) is $\xi(\cdot,0)\equiv 0$.   

We have a similar development when the initial data for (\ref{A1}), (\ref{B1}) satisfies (\ref{H1}) with $\beta_0=p/(1+p)$. Let the function $\tilde{w}_0:(0,1]\ra\R$ be defined by
\be \label{O2}
\tilde{w}_0\left(\frac{f(0)}{f(x)}\right) \ = \ w(x,0) \ , \quad 0\le x<1 \ ,
\ee
and the functional $I(\cdot)$ by
\be \label{P2}
I(\zeta(\cdot),\eta) \ = \ \int_{\ve_0}^\infty \eta^{-p}\tilde{w}_0\left(\frac{f(0)\eta}{\al_0[y+\zeta(y)]}\right) a(y) \ dy \ ,
\ee
where $a(\cdot)$ is given by (\ref{M2}). The conservation law (\ref{B1}) is then expressed in terms of the functional $I(\cdot)$ by 
\be \label{Q2}
\frac{e^t}{u(t)^p}I(\xi(\cdot,t),\eta(t)) \ = \ 1 \ , \quad {\rm where \ } \eta(t)=u(t)^{-1} \ .
\ee
Differentiating (\ref{Q2}), we have from (\ref{I2}) that
\be \label{R2}
\frac{1}{p}\frac{d}{dt} I(\xi(\cdot,t),\eta(t))) \ = \ \left[\rho(\xi(\cdot,t),\eta(t))-\frac{1}{p}\right]I(\xi(\cdot,t))  \ . 
\ee
We can obtain from (\ref{R2}) a formula for $\rho(\cdot)$ similar to (\ref{T1}).  Thus we have from (\ref{L1}), (\ref{I2}), (\ref{R2}) that
\begin{multline} \label{S2}
p\left[\rho(\zeta(\cdot),\eta)-\frac{1}{p}\right]I(\zeta(\cdot),\eta) \ = \\
 [d_\zeta I(\zeta(\cdot),\eta), h+hD\zeta]+\rho(\zeta(\cdot),\eta)[d_\zeta I(\zeta(\cdot),\eta), yD\zeta-\zeta] -\eta\rho(\zeta(\cdot),\eta)\pa I(\zeta(\cdot),\eta)/\pa \eta \ ,
\end{multline}
where $d_\zeta I(\zeta(\cdot),\eta)$ denotes gradient with respect to the function $\zeta(\cdot)$. 
We conclude from (\ref{S2}) that $\rho(\cdot)$ is given by the formula
\be \label{T2}
\rho(\zeta(\cdot),\eta) \ = \ \frac{I(\zeta(\cdot),\eta)+[d_\zeta I(\zeta(\cdot),\eta), h+hD\zeta]}{pI(\zeta(\cdot),\eta)+[d_\zeta I(\zeta(\cdot),\eta), \zeta-yD\zeta]+\eta\pa I(\zeta(\cdot),\eta)/\pa \eta}  \ .
\ee
In the case $h(\cdot)\equiv1$ and $\zeta(\cdot)$ a constant function, (\ref{I2}), (\ref{T2}) with $I(\cdot)=G_1(\cdot)$ yield equation (3.23) of \cite{cn}. Similarly (\ref{L1}), (\ref{T2}) yield equation (3.22) of \cite{cn}. 
Note also from (\ref{J1}), (\ref{B2}) that $\eta(t)$ converges to zero  as $t\ra\infty$ at exponential rate $1/p$. Letting $\eta\ra0$ in (\ref{P2}), (\ref{T2}) we obtain the situation of the previous paragraph. 

We shall obtain some boundedness and regularity properties of the function $I(\cdot)$ given by (\ref{P2}), assuming only that $\beta(\cdot,0)$ is continuous,  satisfies  (\ref{H1}) and in addition the inequality
\be \label{U2}
0 \ < \ \inf\beta(\cdot,0) \ \le \ \sup\beta(\cdot,0) \ <1  \ .
\ee
Observe that
\be \label{V2}
0 \ < \ \frac{f(0)\eta}{\al_0[y+\zeta(y)]} \ \le \ 1 \quad {\rm if \ } 0<\eta\le 1, \ \zeta(y)\ge0, \ y\ge \ve_0 \ .
\ee
Hence $I(\zeta(\cdot),\eta)$ is  well defined by (\ref{P2}) provided $\zeta(\cdot)$ is a non-negative function and $0<\eta\le 1$.  In order to use our assumptions (\ref{H1}), (\ref{U2}) on the beta function,  we make a change of variable  in the integral  (\ref{P2}). Let $\Ga:\R^+\times (0,1]\times[0,1)\ra[0,1)$ be the function defined by
\be \label{W2}
f(\Ga(\zeta,\eta,x)) \ =  \ \frac{f(x)+\al_0\zeta}{\eta} \ , \quad \zeta\ge0, \ 0<\eta\le 1, \ 0\le x<1 \  .
\ee
From (\ref{C2}), (\ref{W2}) we see that 
\be \label{X2}
\lim_{\eta\ra 0}\eta^{-1}[1-\Ga(\zeta,\eta,x)] \ = \ \frac{1}{f(x)+\al_0\zeta} \ .
\ee
 Setting $y=f(x)/\al_0$ in (\ref{P2})  we have that
\be \label{Y2}
I(\zeta(\cdot),\eta) \ = \ \eta^{-p}\int_0^1 w(\Ga(\zeta(f(x)/\al_0),\eta,x),0) \ dx \ .
\ee

We recall from \cite{c2} how $\beta(\cdot,0)$ is related to the function $w(\cdot,0)$. 
Let $X$ be the random variable with cdf determined by $w(\cdot,0)$, so
  \be \label{Z2}
  P(X>x) \ = \ \frac{w(x,0)}{w(0,0)} \ , \quad 0<x<1 \ .
  \ee
The beta function is related to the cdf of $X$ by the formula
 \be \label{AA2}
 \beta(x,0) \ = \ 1+\frac{d}{dx} E[X-x \ | \ X>x] \ , \quad 0\le x<1 \ .
 \ee 
The function $h(\cdot,0)$ defined by (\ref{F1}) can be expressed as the expectation value
\be \label{AB2}
h(x,0) \ = \ w(0,0)E[X-x; \ X>x] \ .
\ee
Hence on integration of (\ref{AA2}) we obtain the identity
\be \label{AC2}
E[X-x \ | \ X>x] \ = \ \frac{h(x,0)}{w(x,0)} \ = \  \int_x^1 [1-\beta(x',0)] \ dx' \ .
\ee
Integrating (\ref{AC2}) yields the formula
\be \label{AD2}
h(x,0) \ = \ h(0,0)\exp\left[-\int_0^x dx'\Big/\int_{x'}^1[1-\beta(x'',0)] \ dx'' \right] \  .
\ee
\begin{lem}
Assume $\beta(\cdot,0)$ is continuous and satisfies (\ref{H1}), (\ref{U2}). Then for any $M,\ga>0$ there exist positive constants $ C_\ga$, depending only on $\ga$,  and $c_{M,\ga}$, depending only on $M,\ga$, such that
\be \label{AE2}
c_{M,\ga}\eta^{\ga} \ \le \  I(\zeta(\cdot),\eta) \ \le \ C_\ga \eta^{-\ga} \quad {\rm for \ } 
0\le\zeta(\cdot)\le M, \ 0<\eta\le 1\  .
\ee
If $\beta(x,0), \ 0\le x<1,$ is H\"{o}lder continuous at $x=1$ then the inequality (\ref{AE2}) also holds with $\ga=0$.
\end{lem}
\begin{proof}
It follows from (\ref{B1}) with $t=0$ that for $0<\del\le \eta\le 1$ the inequality (\ref{AE2}) holds  with constants depending only on $\del,M$.  Therefore   we may restrict ourselves to the situation where $0<\eta<\del$ and $\del>0$ can be arbitrarily small.  Next observe from (\ref{H1}) that for any $\ga, \ 0<\ga<1$, there exists $\ve_\ga>0$ such that
\begin{multline} \label{AF2}
 \frac{1}{p+1+\ga}[1-z] \ \ \le \ \int_z^1 [1-\beta(x',0)] \ dx'  \\
  \le \ \frac{1}{p+1-\ga}[1-z]  \quad {\rm for \ }  \ 1-\ve_\ga\le z<1.
\end{multline}
Hence from (\ref{W2}), (\ref{AD2}), (\ref{AF2}) there exists for any $0<\ga\le 1$ constants $\del_\ga>0$ and $c_\ga,C_\ga>0$ such  that
\begin{multline} \label{AG2}
c_\ga[1-\Ga(\zeta,\eta,x)]^{p+1+\ga} \ \le \ h(\Ga(\zeta,\eta,x),0) \\
 \le \ C_\ga[1-\Ga(\zeta,\eta,x)]^{p+1-\ga} \quad {\rm for \ } \zeta\ge 0, \ 0<\eta\le \del_\ga, \ 0\le x<1 \ .
\end{multline}
The inequality (\ref{AE2}) follows now from (\ref{X2}), (\ref{Y2}), (\ref{AC2}), (\ref{AG2}). 

If $\beta(x,0), \ 0\le x<1,$ is H\"{o}lder continuous at $x=1$ with exponent $q, \ 0<q\le 1,$ then there exist constants $C_1,C_2>0$ such that
\begin{multline} \label{AH2}
 \frac{1}{p+1}[1-z]\frac{1}{1+C_1(1-z)^q} \ \le \ \int_z^1 [1-\beta(x',0)] \ dx'  \\
  \le \ \frac{1}{p+1}[1-z]\left\{1+C_2(1-z)^q\right\}  \quad {\rm for \ }  \ 0\le z<1.
\end{multline}
One easily sees  that the inequality (\ref{AG2}) with $\ga=0$ follows from (\ref{AH2}), whence (\ref{AE2}) with $\ga=0$ also holds. 
\end{proof}
We obtain an expression for $d_\zeta I(\zeta(\cdot),\eta)$  by differentiating (\ref{Y2}). To do this we use the formula
\be \label{AI2}
-\frac{\pa w(x,0)/\pa x}{w(x,0)} \ = \ \beta(x,0)\Big/ \int_x^1 [1-\beta(x',0)] \ dx'  \ .
\ee
We have also from (\ref{C2}), (\ref{W2}) that
\be \label{AJ2}
\frac{\pa \Ga(\zeta,\eta,x)}{\pa \zeta} \ = \ -\frac{\al_0\psi(\Ga(\zeta,\eta,x))}{\eta\psi'(1)f(\Ga(\zeta,\eta,x))} \ .
\ee
It follows from (\ref{X2}), (\ref{AJ2}) that
\be \label{AK2}
\frac{\pa \Ga(\zeta,\eta,x)}{\pa \zeta} \ \simeq \ \al_0\eta(1-x)^2 \quad {\rm as \ } \eta\ra 0 \ .
\ee
  Differentiating (\ref{Y2})  we conclude from (\ref{W2}), (\ref{AI2}), (\ref{AJ2}) that
 \begin{multline} \label{AL2}
  d_\zeta I(\zeta(\cdot),\eta;f(x)/\al_0) \ = \\
   -\eta^{-p}\frac{\al_0g(\Ga(\zeta(f(x)/\al_0),\eta,x))}{f(x)+\al_0\zeta(f(x)/\al_0)}w(\Ga(\zeta(f(x)/\al_0),\eta,x),0) a(f(x)/\al_0) \ , \quad 0\le x<1 \ ,
  \end{multline}
  where the function $g(\cdot)$ is given by the formula
  \be \label{AM2}
  g(x) \ = \ -\psi(x)\beta(x,0)\Big/\psi'(1)\int_x^1 [1-\beta(x',0)] \ dx'  \ .
  \ee
  Observe from (\ref{U2}), (\ref{AM2})  that
\be \label{AN2}
\inf g(\cdot)>0 \quad {\rm and \ } \lim_{x\ra1}g(x) \ = \ p \ .
\ee
We see from (\ref{AL2}), (\ref{AM2}) that $d_\zeta I(\zeta(\cdot),\eta)$ is a negative function for all non-negative $\zeta(\cdot)$ and $0<\eta\le 1$. 

We can obtain similar  formulas for $\pa I(\zeta(\cdot),\eta)/\pa \eta$ and for the denominator in the expression (\ref{T2}) for $\rho(\zeta(\cdot),\eta)$. First observe from (\ref{C2}), (\ref{W2}) that
\be \label{AO2}
\frac{\pa \Ga(\zeta,\eta,x)}{\pa \eta} \ = \ \frac{\psi(\Ga(\zeta,\eta,x))}{\eta\psi'(1)} \ .
\ee
Now differentiating (\ref{Y2}) with respect to $\eta$  we have from (\ref{AI2}), (\ref{AM2}), (\ref{AO2}) that
 \begin{multline} \label{AP2}
  pI(\zeta(\cdot),\eta) +\eta\frac{\pa I(\zeta(\cdot),\eta)}{\pa \eta} \ = \\
  \eta^{-p}\int_0^1 g(\Ga(\zeta(f(x)/\al_0),\eta,x))w(\Ga(\zeta(f(x)/\al_0),\eta,x),0) \ dx \  .
  \end{multline}
Note that the RHS of (\ref{AP2}) is positive. From (\ref{AL2}), (\ref{AP2}) we conclude that the 
denominator in the expression (\ref{T2}) is given by the formula
\begin{multline} \label{AQ2}
pI(\zeta(\cdot),\eta)+\eta\frac{\pa I(\zeta(\cdot),\eta)}{\pa \eta}+[d_\zeta I(\zeta(\cdot),\eta), \zeta-yD\zeta]   \ = \ \\
 \eta^{-p}\int_0^1 \frac{g(\Ga(\zeta(f(x)/\al_0),\eta,x))}{f(x)+\al_0\zeta(f(x)/\al_0)}f(x)\left[1+D\zeta(f(x)/\al_0)\right] \\
 \times w(\Ga(\zeta(f(x)/\al_0),\eta,x),0) \ dx \
  = \  - [d_\zeta I(\zeta(\cdot),\eta), y(1+D\zeta)] \ .
\end{multline}
\begin{lem}
Assume $\beta(\cdot,0)$ is continuous and satisfies (\ref{H1}), (\ref{U2}). Let $\zeta:[\ve_0,\infty)\ra\R$ be a $C^1$ non-negative function such that $\inf D\zeta(\cdot)\ge-1$, and $0<\eta\le 1$. Then the function $\rho(\cdot,\cdot)$ of (\ref{T2}) satisfies the inequality
\be \label{AR2}
\rho(\zeta(\cdot),\eta) \ \ge \ -\sup_{y>\ve_0}\frac{h(y)}{y} \ .
\ee
Suppose  there is a constant $M$ such that $\sup D\zeta(\cdot)\le M$.  Then there exists $\del_M>0$, depending only on $M$, such that (\ref{AR2}) can be improved to
\be \label{AS2}
\rho(\zeta(\cdot),\eta) \ \ge \ \del_M-\sup_{y>\ve_0}\frac{h(y)}{y} \ .
\ee
Furthermore, there exists $\del_M,K_M>0$, depending only on $M$, such that
\be \label{AT2}
\rho(\zeta(\cdot),\eta) \ \ge \ \del_M \quad {\rm provided \ } \inf \zeta(\cdot) \ \ge \ K_M \ .
\ee

Suppose there is a constant $m>0$ such that $\inf[1+ D\zeta(\cdot)]\ge m$.  For any $M>0$ there is a constant $C_{m,M}$, depending only on $m,M$, such that
\be \label{AU2}
\rho(\zeta(\cdot),\eta) \ \le \ C_{m,M}  \quad {\rm provided \ } \sup \zeta(\cdot) \ \le \ M \ .
\ee
\end{lem}
\begin{proof}
The inequality (\ref{AR2}) follows immediately from (\ref{T2}), (\ref{AQ2}) on using the positivity of $I(\zeta(\cdot),\eta)$ and the negativity of the function $d_\zeta I (\zeta(\cdot),\eta):[\ve_0,\infty)\ra\R$. To prove the remaining inequalities we observe from (\ref{Y2}), (\ref{AL2}), (\ref{AN2}) there are constants $C,c>0$ such that
\be \label{AV2}
\int_{\ve_0}^\infty |d_\zeta I(\zeta(\cdot),\eta;y)| \ dy \ \le \  \frac{C}{\inf_{y>\ve_0}[y+\zeta(y)]}I(\zeta(\cdot),\eta)  \ , 
\ee
and also
\be \label{AW2}
c\inf_{y\ge\ve_0}\frac{y}{y+\zeta(y)}I(\zeta(\cdot),\eta) \ \le \ \int_{\ve_0}^\infty |d_\zeta I(\zeta(\cdot),\eta;y)|y \ dy \ \le \  CI(\zeta(\cdot),\eta) \ .
\ee
Then (\ref{AS2}) follows from the upper bound in (\ref{AW2}). 
We see from (\ref{AV2}) that it is possible to choose $K_M$ sufficiently large, depending on $M$, so that
$0\ge[d_\zeta I(\zeta(\cdot),\eta), h+hD\zeta]\ge -I(\zeta(\cdot),\eta)/2$ when $\inf \zeta(\cdot) \ \ge \ K_M$. From the upper bound in (\ref{AW2}) it follows that $0\le- [d_\zeta I(\zeta(\cdot),\eta), y(1+D\zeta)]\le C[1+M]I(\zeta(\cdot),\eta)$.  Evidently we may choose $\del_M=1/2C[1+M]$ in (\ref{AT2}). The inequality (\ref{AU2}) is a consequence of the lower bound in (\ref{AW2}). 
\end{proof}
Next we  estimate some second derivatives of $I(\zeta(\cdot),\eta)$ with respect to $\zeta(\cdot)$.
To carry this out we first observe from (\ref{C2}), (\ref{W2}) that
\be \label{AX2}
\frac{\pa}{\pa x}\Ga(\zeta(f(x)/\al_0),\eta,x) \ = \ 
\frac{f(x)\psi(\Ga(\zeta,\eta,x))}{\eta\psi(x)f(\Ga(\zeta,\eta,x))} \left[1+D\zeta(f(x)/\al_0)\right]  \ .
\ee
 We conclude from (\ref{AL2}), (\ref{AX2}) that
 \begin{multline} \label{AY2}
  d_\zeta I(\zeta(\cdot),\eta;f(x)/\al_0) \left[1+D\zeta(f(x)/\al_0)\right]  \ = \\
-\eta^{-p} \frac{\psi(x)a(f(x)/\al_0)}{\psi'(1)f(x)} \frac{\pa}{\pa x} w(\Ga(\zeta(f(x)/\al_0),\eta,x),0)   \ . 
 \end{multline}
 Suppose now that $\zeta,G:[\ve_0,\infty)\ra\R$ are non-negative $C^1$ functions with bounded first derivatives. We have from (\ref{AY2}) upon integration by parts that
 \begin{multline} \label{AZ2}
 [d_\zeta I(\zeta(\cdot),\eta), G(1+D\zeta)] \ = \
  \eta^{-p} \frac{\psi(0)G(f(0)/\al_0)}{\psi'(1)f(0)}w(\Ga(\zeta(f(0)/\al_0),\eta,0),0) \\
  +\eta^{-p}\int_0^1\left\{\frac{[\psi'(x)+\psi'(1)]G(f(x)/\al_0)}{\psi'(1)f(x)}
  -\frac{G'(f(x)/\al_0)}{\al_0}\right\}w(\Ga(\zeta(f(x)/\al_0),\eta,x),0) \ dx \ .
 \end{multline}  
 \begin{lem}
 Assume $\beta(\cdot,0)$ is continuous and satisfies (\ref{H1}), (\ref{U2}). Let $G:[\ve_0,\infty)\ra\R$ be a $C^1$ non-negative function with bounded first derivative. Define $H_G(\zeta(\cdot),\eta)$ for non-negative $C^1$ functions $\zeta:[\ve_0,\infty)\ra\R^+$ with bounded derivative and $0<\eta\le 1$ by $H_G(\zeta(\cdot),\eta)=-[d_\zeta I(\zeta(\cdot),\eta), G(1+D\zeta)]$.  Given  any $\ga>0$ there is a constant $C_{\ga}$ such that
 \be \label{BA2}
 |H_G(\zeta_1(\cdot),\eta)-H_G(\zeta_2(\cdot),\eta)| \ \le C_{\ga}\eta^{-\ga}\|\zeta_1(\cdot)-\zeta_2(\cdot)\|_\infty \ ,
 \ee
 for all non-negative $\zeta_j(\cdot), \ j=1,2.$, and also
 \be \label{BB2}
 |H_G(\zeta(\cdot),\eta_1)-H_G(\zeta(\cdot),\eta_2)| \ \le C_{\ga}\eta_1^{-(\ga+1)}|\eta_1-\eta_2| \ ,
 \ee
 for all $\zeta(\cdot)\ge0, \ 0<\eta_1<\eta_2\le 1$.
 \end{lem}
 \begin{proof}
  We use the representation (\ref{AZ2}) and the fundamental theorem of calculus to prove (\ref{BA2}).  For the first term on the RHS of (\ref{AZ2}) we have
 \begin{multline} \label{BC2}
 w(\Ga(\zeta_1(f(0)/\al_0),\eta,0),0)-w(\Ga(\zeta_2(f(0)/\al_0),\eta,0),0) \ = \\
- \int_0^1d\la \   \frac{g(X(\la))}{\ve_0+\la\zeta_1(\ve_0)+(1-\la)\zeta_2(\ve_0)} \  w(X(\la),0)\left[\zeta_1(\ve_0)-\zeta_2(\ve_0)\right]  \ ,
 \end{multline}
 where $X(\la)=\Ga(\la\zeta_1(\ve_0)+(1-\la)\zeta_2(\ve_0),\eta,0)$. It follows from (\ref{W2}), (\ref{AC2}), (\ref{AG2}), (\ref{BC2}) there is for any $\ga>0$ a constant $C_{1,\ga}$ such that
 \be \label{BD2}
 \left|w(\Ga(\zeta_1(\ve_0),\eta,0),0)-w(\Ga(\zeta_2(\ve_0),\eta,0),0)\right| \ \le \ 
 C_{1,\ga}\eta^{p-\ga}\left|\zeta_1(\ve_0)-\zeta_2(\ve_0)\right|  \ .
 \ee
 We can make a similar argument to estimate the integral on the RHS of (\ref{AZ2}), whence (\ref{BA2}) follows. We prove (\ref{BB2}) in the same way as (\ref{BA2})  by using (\ref{AO2}). \end{proof}
 \begin{rem}
 Note that we only use the integrability of $\psi'(\cdot)$ in our estimate, so $\psi'(x)$ can diverge as $x\ra 0$, as in the case of the LSW functions (\ref{C1}).  Similarly we only need the function $y\ra G'(y)$ to be integrable close to $y=\ve_0$. Hence the result of Lemma 2.3 applies to the function $G(\cdot)=h(\cdot)$. 
 \end{rem}
 In Lemma 2.2 we obtained bounds on $\rho(\zeta(\cdot),\eta)$ which are uniform for $0<\eta\le 1$. Next we show that $\lim_{\eta\ra 0}\rho(\zeta(\cdot),\eta)$ exists. 
\begin{lem}
Assume $\beta(\cdot,0)$ is continuous and satisfies (\ref{H1}), (\ref{U2}). For $m,M>0$, let $S_{m,M}$ be the set of non-negative $C^1$ functions $\zeta:[\ve_0,\infty)\ra\R$  such that $1+\inf D\zeta(\cdot)\ge m$, and $\|\zeta(\cdot)\|_{1,\infty}\le M$, where $\|\cdot\|_{1,\infty}$ is defined by (\ref{E3}). Then the function $\rho(\cdot,\cdot)$ of (\ref{T2}) satisfies the limit
\be \label{BE2}
\lim_{\eta\ra 0}\sup\{\left|\rho(\zeta(\cdot),\eta)-\rho_p(\zeta(\cdot))\right| \ : \ \zeta(\cdot)\in S_{m,M}\} \ = \ 0 \ ,
\ee
where $\rho_p(\zeta(\cdot))$ is given by the RHS of (\ref{T1}) with $I(\zeta(\cdot))$ as in (\ref{L2}), (\ref{M2}).  If the function $x\ra\beta(x,0)$ is H\"{o}lder continuous at $x=1$ of order $\ga>0$ then there exists a constant $C_{m,M}$, depending only on $m,M$, such that
\be \label{BF2}
\sup\{\left|\rho(\zeta(\cdot),\eta)-\rho_p(\zeta(\cdot))\right| \ : \ \zeta(\cdot)\in S_{m,M}\} \ \le \  C_{m,M} \ \eta^\ga 
\quad {\rm for \  0<\eta\le 1} \ .
\ee
\end{lem}
\begin{proof}
We already saw in Lemma 2.1 that the function $I(\zeta(\cdot),\eta)$ can diverge as $\eta\ra 0$. We can however obtain a finite limit by introducing a suitable normalization. Thus we  show that
\begin{multline} \label{BG2}
\lim_{\eta\ra 0}\sup\Bigg\{\Bigg|\frac{\eta^p}{w(\Ga(\zeta(f(0)/\al_0),\eta,0),0)}I(\zeta(\cdot),\eta) \\
-\left[\ve_0+\zeta(\ve_0)\right]^p \frac{\al_0^pI_p(\zeta(\cdot))}{K_p}\Bigg|  \ 
 : \ 0\le \zeta(\cdot)\le M \Bigg\} \ = \ 0 \ ,
\end{multline}
where $I_p(\zeta(\cdot))$ is the function $I(\zeta(\cdot))$ of (\ref{L2}), (\ref{M2}).  To see (\ref{BG2}) we observe from (\ref{Y2}) that
\begin{multline} \label{BH2}
\frac{\eta^p}{w(\Ga(\zeta(f(0)/\al_0),\eta,0),0)}I(\zeta(\cdot),\eta) \ = \\ 
\int_0^1 \frac{w(\Ga(\zeta(f(x)/\al_0),\eta,x),0)}{w(\Ga(\zeta(f(0)/\al_0),\eta,0),0)} \ dx \ = \ 
\int_0^1 \al(\zeta(\cdot),\eta,x) \ dx \ .
\end{multline}
From (\ref{AC2}), (\ref{AD2}) we have that
\begin{multline} \label{BI2}
 \al(\zeta(\cdot),\eta,x) \ = \
  \exp\left[-\int_{\Ga(\zeta(f(0)/\al_0),\eta,0)}^{\Ga(\zeta(f(x)/\al_0),\eta,x)} dx'\Big/\int_{x'}^1[1-\beta(x'',0)] \ dx'' \right]  \times \\ 
  \frac{\int_{\Ga(\zeta(f(0)/\al_0),\eta,0}^1 [1-\beta(x',0)] \ dx'}{\int_{\Ga(\zeta(f(x)/\al_0),\eta,x)}^1 [1-\beta(x',0)] \ dx'} \ .
\end{multline}
From (\ref{H1}), (\ref{X2}), (\ref{BI2}) we see that
\be \label{BJ2}
\lim_{\eta\ra 0}  \al(\zeta(\cdot),\eta,x) \ = \ \left(\frac{f(0)+\al_0\zeta(f(0)/\al_0)}{f(x)+\al_0\zeta(f(x)/\al_0)}\right)^p \ .
\ee
Note that for any $M>0$ and $\del$ with $0<\del<1$, the limit in (\ref{BJ2}) is uniform for $[\zeta(\cdot),x]$ in the set $\{[\zeta(\cdot),x]: \ 0\le \zeta(\cdot)\le M, \ 0\le x\le 1-\del\}$.  We also have that for any $\ga$ with $0<\ga<p$ there exists a constant $C_{M,\ga}$  such that
\be \label{BK2}
0 \ \le  \al(\zeta(\cdot),\eta,x) \ \le \ C_{M,\ga}(1-x)^{p-\ga} \quad {\rm for \ } 0<\eta\le 1, \ 0\le\zeta(\cdot)\le M \ .
\ee
We conclude from (\ref{L2}), (\ref{M2}), (\ref{BH2}), (\ref{BJ2}), (\ref{BK2}) that
\begin{multline} \label{BL2}
\lim_{\eta\ra 0}\frac{\eta^p}{w(\Ga(\zeta(f(0)/\al_0),\eta,0),0)}I(\zeta(\cdot),\eta) \ = \\ 
\left[f(0)/\al_0+\zeta(f(0)/\al_0)\right]^p\int_0^1\frac{dx}{\left[f(x)/\al_0+\zeta(f(x)/\al_0)\right]^p} \ = \\ 
\left[f(0)/\al_0+\zeta(f(0)/\al_0)\right]^p\int_{\ve_0}^\infty \frac{a(y)}{[b(y)+\zeta(y)]^p} \ dy  \ = \ 
\left[\ve_0+\zeta(\ve_0)\right]^p \frac{\al_0^pI_p(\zeta(\cdot))}{K_p} \ .
\end{multline} 
The limit in (\ref{BL2}) is uniform for $\zeta(\cdot)$ in the set $0\le \zeta(\cdot)\le M$, whence (\ref{BG2}) follows. 

To show convergence of the terms involving $d_\zeta I$ in (\ref{T2}) we observe from (\ref{AL2}) that
 \begin{multline} \label{BM2}
  \frac{\eta^p}{w(\Ga(\zeta(f(0)/\al_0),\eta,0),0)}d_\zeta I(\zeta(\cdot),\eta;f(x)/\al_0) \ = \\
   -\frac{\al_0g(\Ga(\zeta(f(x)/\al_0),\eta,x))}{f(x)+\al_0\zeta(f(x)/\al_0)}\al(\zeta(\cdot),\eta,x)  a(f(x)/\al_0) \ , \quad 0\le x<1 \ .
  \end{multline}
  We conclude from (\ref{AN2}), (\ref{BJ2}), (\ref{BM2}) that
  \begin{multline} \label{BN2}
  \lim_{\eta\ra 0} 
    \frac{\eta^p}{w(\Ga(\zeta(f(0)/\al_0),\eta,0),0)}d_\zeta I(\zeta(\cdot),\eta;f(x)/\al_0) \ =   \\
     -\frac{p\left[f(0)/\al_0+\zeta(f(0)/\al_0)\right]^p}{\left[f(x)/\al_0+\zeta(f(x)/\al_0)\right]^{p+1}}a(f(x)/\al_0) \ = \\
      \left[\ve_0+\zeta(\ve_0)\right]^p \frac{\al_0^pd_\zeta I_p(\zeta(\cdot));f(x)/\al_0)}{K_p} \ .
  \end{multline}
  Let $G:[\ve_0,\infty)\ra\R$ be a $C^1$ function with bounded first derivative. It follows from (\ref{BN2}) that
  \begin{multline} \label{BO2}
  \lim_{\eta\ra 0}  \frac{\eta^p}{w(\Ga(\zeta(f(0)/\al_0),\eta,0),0)}[d_\zeta I(\zeta(\cdot),\eta), G(1+D\zeta)] \\ = \
  \left[\ve_0+\zeta(\ve_0)\right]^p \frac{\al_0^p}{K_p}  [d_\zeta I_p(\zeta(\cdot)), G(1+D\zeta)]  \ .
  \end{multline}
  The limit in (\ref{BO2}) is uniform for non-negative functions $\zeta(\cdot)$ satisfying  $\|\zeta(\cdot)\|_{1,\infty}\le M$.  Now (\ref{BE2}) follows from the lower bound in (\ref{AW2}) and (\ref{BG2}), (\ref{BO2}) upon using the formula (\ref{AQ2}) for the denominator of (\ref{T2}). 
  
  To prove (\ref{BF2}) we observe that if the function $x\ra\beta(x,0)$ is H\"{o}lder continuous at  $x=1$ of order $\ga>0$ then there is a constant $C_1>0$ such that
 \begin{multline} \label{BP2}
  \frac{1}{[1+C_1\eta^\ga]}\left(\frac{f(0)+\al_0\zeta(f(0)/\al_0)}{f(x)+\al_0\zeta(f(x)/\al_0)}\right)^p \    \le \ \al(\zeta(\cdot),\eta,x) \\
   \le \ [1+C_1\eta^\ga]\left(\frac{f(0)+\al_0\zeta(f(0)/\al_0)}{f(x)+\al_0\zeta(f(x)/\al_0)}\right)^p
   \quad {\rm for \ } \zeta(\cdot)\ge 0, \ 0<\eta\le 1, \ 0\le x<1 \ .
  \end{multline}
  By carrying through the argument for the proof of (\ref{BE2}) and using (\ref{BP2})  we obtain (\ref{BF2}). 
 \end{proof}
 We can use the method in the proof of Lemma 2.4 to improve on the result of Lemma 2.1. 
 \begin{lem}
Assume $\beta(\cdot,0)$ is continuous and satisfies (\ref{H1}), (\ref{U2}). Then for any $M,\ga>0$ there exist positive constants  $C_{M,\ga},c_{M,\ga}$, depending only on $M,\ga$, such that
\be \label{BQ2}
c_{M,\ga}\left(\frac{\eta_1}{\eta_2}\right)^{\ga} \ \le \  \frac{I(\zeta_1(\cdot),\eta_1)}{I(\zeta_2(\cdot),\eta_2)} \ \le \ C_{M,\ga} \left(\frac{\eta_2}{\eta_1}\right)^{\ga} \quad {\rm for \ } 
0\le\zeta_1(\cdot), \ \zeta_2(\cdot)\le M, \ 0<\eta_1\le \eta_2\le 1\  .
\ee
\end{lem}
\begin{proof}
From (\ref{BH2}) we have that
\be \label{BR2}
\frac{I(\zeta_1(\cdot),\eta_1)}{I(\zeta_2(\cdot),\eta_2)}  \ = \ 
\left(\frac{\eta_2}{\eta_1}\right)^p\frac{w(\Ga(\zeta_1(f(0)/\al_0),\eta_1,0),0)}{w(\Ga(\zeta_2(f(0)/\al_0),\eta_2,0),0)}\frac{\int_0^1\al(\zeta_1(\cdot),\eta_1,x) \ dx}{\int_0^1\al(\zeta_2(\cdot),\eta_2,x) \ dx} \ .
\ee 
Using the representation analogous to (\ref{BI2}) for $\al(\zeta(\cdot),\eta,x)$, we see that 
\be \label{BS2}
c\left(\frac{\eta_1}{\eta_2}\right)^{p+\ga} \ \le \ \frac{w(\Ga(\zeta_1(f(0)/\al_0),\eta_1,0),0)}{w(\Ga(\zeta_2(f(0)/\al_0),\eta_2,0),0)} \ \le \ C\left(\frac{\eta_1}{\eta_2}\right)^{p-\ga} \ ,
\ee
where $c,C>0$ depend only on $M,\ga>0$. 
\end{proof}

 \vspace{.2in}
 
 \section{Proof of Theorem 1.2}
 We shall prove Theorem 1.2 using the setup developed in $\S2$. In order to carry this out we shall need to assume that the function $u(\cdot)$ of (\ref{B2}) satisfies $\inf u(\cdot)\ge 1$ and also that $\beta(\cdot,0)$ satisfies (\ref{U2}) as well as  (\ref{H1}). The lower bound on $\inf u(\cdot)$ and the requirement that (\ref{U2}) holds are not actually additional constraints beyond (\ref{H1}) on the initial data.  We can see this from the arguments in $\S2$ of \cite{cn}. In fact there exists $T_0>0$ such that $\inf_{t\ge T_0}u(t)\ge 1$ and $\beta(\cdot,T_0)$ satisfies (\ref{U2}).  Hence the formula (\ref{P2}) for $I(\xi(\cdot,t),\eta(t))$ with $\eta(t)=u(t)^{-1}$  is well-defined provided $t\ge T_0$ and $\xi(\cdot,t)$ is a non-negative function. 
 
 Let $\eta:[0,\infty)\ra(0,1]$ be a continuous strictly positive function. We study solutions to (\ref{L1}) with $\rho(\xi(\cdot,t))=\rho(\xi(\cdot,t),\eta(t)), \ t\ge 0,$ where $\rho(\zeta(\cdot),\eta)$ is given by the formula (\ref{T2}) and $I(\zeta(\cdot),\eta)$ by (\ref{P2}). We first consider the linear PDE
 \be \label{A3}
 \frac{\pa \xi(y,t)}{\pa t} -h(y) -h(y)\frac{\pa\xi(y,t)}{\pa y} +\rho(t)\left[\xi(y,t)-y\frac{\pa \xi(y,t)}{\pa y}\right] \ = \ 0, \quad y>\ve_0, \ t>0,
  \ee
 where $\rho:[0,\infty)\ra\R$ is assumed to be a known function.  The initial value problem for (\ref{A3}) can be solved globally in time by the method of characteristics provided
 \be \label{B3}
 \inf \rho(\cdot) \ \ge \ -\frac{h(\ve_0)}{\ve_0} \ .
 \ee
 If (\ref{B3}) holds the characteristic $y(\cdot)$, defined as the solution to the terminal value problem
\be \label{C3}
\frac{dy(s)}{ds} \ = \ -h(y(s))-\rho(s)y(s) \ , \quad 0\le s<t, \ y(t)=y \ ,
\ee
has the property that for $y\ge \ve_0$ then $y(s)\ge \ve_0, \ 0\le s< t$.  The solution to (\ref{A3}) with given initial data $\xi(y,0), \ y\ge \ve_0,$ is then expressed in terms of the characteristic (\ref{C3})  by
\be \label{D3}
\xi(y,t) \ = \ \exp\left[-\int_0^t \rho(s) \ ds\right]\xi(y(0),0)+\int_0^t ds \  h(y(s))\exp\left[-\int_s^t \rho(s') \ ds'\right]  \ . 
\ee
It follows from (\ref{M1}), (\ref{D3}) that if the initial data $\xi(\cdot,0)$ is non-negative, then the function $\xi(\cdot,t)$ is non-negative for all $t>0$.

We prove a local existence and uniqueness theorem for the initial value problem for (\ref{L1}), where the function $t\ra\rho(\xi(\cdot,t))$ in (\ref{L1}) is replaced by the function $t\ra\rho(\xi(\cdot,t),\eta(t))$,  with $\rho(\zeta(\cdot),\eta)$  given by (\ref{T2}). We shall follow the same line of argument as in Lemma 2.1 of \cite{cd}. Thus following (1.9) of \cite{cd}, we define for $m=1,2,\dots,$  norms on $C^m$ functions $\zeta:[\ve_0,\infty)\ra\R$  by
\be \label{E3}
\|\zeta(\cdot)\|_{m,\infty} \ = \ \sup_{\ve_0\le y<\infty}\sum_{k=0}^m y^k\left|\frac{d^k\zeta(y)}{dy^k}\right|  \ .
\ee
\begin{lem}
Let $h:[\ve_0,\infty)\ra\R$ be a  continuous positive function, which is $C^1$ on $(\ve_0,\infty)$, and satisfies $\|h(\cdot)\|_\infty<\infty, \  \int_{\ve_0}^{\ve_0+1}|h'(y)| \ dy+\sup_{y\ge\ve_0+1}y|h'(y)|<\infty, \  \sup_{y\ge\ve_0} [h(y)/y]=h(\ve_0)/\ve_0$. Assume $\beta(\cdot,0)$ is continuous and satisfies (\ref{H1}), (\ref{U2}),  $I(\zeta(\cdot),\eta)$ is given by (\ref{Y2}), and $\eta:[0,\infty)\ra(0,1]$ is a continuous strictly positive function.  Let $\rho(\zeta(\cdot),\eta)$ be defined by (\ref{T2}), and consider the initial value problem for (\ref{A3}) with $\rho(t)=\rho(\xi(\cdot,t),\eta(t))$.  Given $C^1$ non-negative initial data $\xi(y,0), \ y\ge \ve_0$, such that $\|\xi(\cdot,0)\|_{1,\infty}<\infty$ and $\inf D\xi(\cdot,0)\ge -1$, there exists for some $T>0$ a unique non-negative solution $\xi(y,t), \ y\ge\ve_0, \ 0\le t\le T,$ such that $\|\xi(\cdot,t)\|_{1,\infty}<\infty$ and $\inf D\xi(\cdot,t)\ge -1, \ 0\le t\le T$.
\end{lem}
\begin{proof}
Since $\xi(\cdot,0)$ is a non-negative function we have that $1+D\xi(\cdot,0)\ge0$, but is not identically zero. Hence the denominator $[d_\zeta I(\xi(\cdot,0),\eta(0)),y(1+D\xi(\cdot,0))]$ in the formula (\ref{T2}) is positive. Since $\|\xi(\cdot,0)\|_{1,\infty}<\infty$ we have from (\ref{AS2}) of Lemma 2.2  that $\rho_0=\rho(\xi(\cdot,0),\eta(0))> -h(\ve_0)/\ve_0$. For $T>0$ and $0<\ve<[\rho_0+h(\ve_0)/\ve_0]/2$,  let $\mathcal{E}_{\ve,T}$ be the metric space of continuous functions $\rho:[0,T]\ra\R$ such that  $\rho(0)=\rho_0$ and  $\|\rho(\cdot)-\rho_0\|_\infty\le\ve$. If $\rho(\cdot)\in\mathcal{E}_{\ve,T}$ then $\inf\rho(\cdot)\ge -h(\ve_0)/\ve_0+\ve$ and the function $\xi(\cdot,t)$ is well defined by (\ref{D3}).  We define the function $K\rho:[0,T]\ra\R$ by
\be \label{F3}
K\rho(t) \ = \ \rho(\xi(\cdot,t),\eta(t)) \ , \ \ 0\le t\le T, \quad  {\rm where \ } \xi(\cdot,t) \ {\rm is \  given \  by \  (\ref{D3}).}
\ee
Evidently Lemma 2.2 implies that $\inf K\rho(\cdot)\ge -h(\ve_0)/\ve_0$, and fixed points of $K$ correspond to solutions $\xi(\cdot,\cdot)$ of (\ref{A3}) with $\rho(t)=\rho(\xi(\cdot,t),\eta(t))$.   

We first show that $K$ maps $\mathcal{E}_{\ve,T}$ to itself provided $\ve,T>0$ are sufficiently small. To do this we use the representation
\begin{multline} \label{G3}
I(\xi(\cdot,t),\eta(t))-I(\xi(\cdot,0),\eta(0)) \ = \\
 \int_0^1 d\la \ \frac{\pa I(\xi(\cdot,t),\la\eta(t)+(1-\la)\eta(0))}{\pa \eta}[\eta(t)-\eta(0)] + \\
 \int_0^1 d\la \  [d_\zeta I(\la\xi(\cdot,t)+(1-\la)\xi(\cdot,0),\eta(0)),\xi(\cdot,t)-\xi(\cdot,0)]  \ .
\end{multline}
We have from Lemma 2.1, (\ref{Y2}), (\ref{AP2}) there is a constant $C_1$ such that
\be \label{H3}
\left|\eta\frac{\pa I(\zeta(\cdot),\eta)}{\pa\eta}\right| \ \le \ C_1I(\zeta(\cdot),\eta)) \ \le \ C_1C_\ga\eta^{-\ga} \ .
\ee
It follows then from (\ref{AV2}), (\ref{G3}), (\ref{H3}) there is for any $\ga>0$ a constant $C_\ga$ such that
\begin{multline} \label{I3}
\left|I(\xi(\cdot,t),\eta(t))-I(\xi(\cdot,0),\eta(0))\right| \ \le \\
  \frac{C_\ga}{\min[\eta(0),\eta(t)]^{\ga}}\left[\frac{|\eta(t)-\eta(0)|}{\min[\eta(0),\eta(t)]}+\|\xi(\cdot,t)-\xi(\cdot,0)\|_\infty\right] \ .
\end{multline}
Since the function $\eta(\cdot)$ is continuous and $\eta(0)>0$ the first term on the RHS of (\ref{I3}) can be made arbitrarily small for $0\le t\le T$ by choosing $T>0$ sufficiently small.  
We estimate the second term on the RHS of (\ref{I3}) by using (\ref{D3}).  Thus from (\ref{C3}) we have that
\be \label{J3}
\xi(y,0)-\xi(y(0),0) \ = \ -\int_0^t ds \ [h(y(s))+\rho(s)y(s)]D\xi(y(s),0) \ ,
\ee
whence we conclude that for all $\rho(\cdot)$ such that $\|\rho(\cdot)-\rho_0\|_\infty\le\ve$ there is a constant $C_{1,\ve}$ such $|\xi(y,0)-\xi(y(0),0)|\le C_{1,\ve}t\|\xi(\cdot,0)\|_{1,\infty}$.
Now from (\ref{J3}) and the boundedness of $h(\cdot)$ it follows there is a constant $C_{2,\ve}$ such that
\be \label{K3}
\|\xi(\cdot,t)-\xi(\cdot,0)\|_\infty \ \le \  C_{2,\ve}t\left[\ \|\xi(\cdot,0)\|_{1,\infty}+1\right] \ , \quad 0\le t\le T\le 1 \ .
\ee
It follows from (\ref{K3}) that the  second term on the RHS of (\ref{I3}) can be made arbitrarily small for $0\le t\le T$ by choosing $T>0$ sufficiently small. In order to estimate $|K\rho(t)-\rho_0|$ we also need to estimate differences for the function $H_G$ of Lemma 2.3 when $G(y)=h(y)$ and $G(y)=y , y\ge\ve_0$. The inequalities (\ref{BA2}), (\ref{BB2}) hold for both of these functions (see especially the remark following Lemma 2.3).  We have therefore for any $\ga>0$ there exists a constant $C_\ga$ such that
\begin{multline} \label{L3}
\left|H_G(\xi(\cdot,t),\eta(t))-H_G(\xi(\cdot,0),\eta(0))\right| \ \le \\
  \frac{C_\ga}{\min[\eta(0),\eta(t)]^{\ga}}\left[\frac{|\eta(t)-\eta(0)|}{\min[\eta(0),\eta(t)]}+\|\xi(\cdot,t)-\xi(\cdot,0)\|_\infty\right] \ .
\end{multline}
From (\ref{I3}), (\ref{K3}), (\ref{L3}) we conclude for given $\ve>0$, there exists $T>0$ such that if $\rho(\cdot)\in\mathcal{E}_{\ve,T}$  then $K\rho(\cdot)\in\mathcal{E}_{\ve,T}$. 

We may make a similar argument to show that $K:\mathcal{E}_{\ve,T}\ra\mathcal{E}_{\ve,T}$ is a contraction mapping provided $T>0$ is sufficiently small. 
Let $\xi_j(\cdot,t), \ j=1,2, \ 0\le t\le T,$ denote the functions (\ref{D3}) with $\rho(\cdot)=\rho_j(\cdot)\in \mathcal{E}_{\ve,T}, \ j=1,2$. We shall show that
\be \label{M3}
\|\xi_1(\cdot,t)-\xi_2(\cdot,t)\|_\infty \ \le \  C_{3,\ve}t\|\rho_1(\cdot)-\rho_2(\cdot)\|_\infty\left[ \  \|\xi(\cdot,0)\|_{1,\infty}+1\right] \ , \quad 0\le t\le T,
\ee
for a constant $C_{3,\ve}$ depending only on $\rho_0,\ve,$ provided $T\le 1$. To see this first observe from (\ref{C3}) the identity
\be \label{N3}
y(s) \ = \ \exp\left[\int_s^t\rho(s') \ ds'\right]y+\int_s^t \exp\left[\int_s^{s'}\rho(s'') \ ds''\right]h(y(s')) \ ds' \ .
\ee
Since $h(\cdot)$ is non-negative and bounded there exist constants $C_\ve,c_\ve>0$ depending only on $\ve,\rho_0$ such that for $0<T\le 1$,
\be \label{O3}
c_\ve y \ \le y(s) \ \le \ C_\ve y \quad {\rm for \ } y\ge \ve_0, \ 0\le s\le t\le T, \  \rho(\cdot)\in \mathcal{E}_{\ve,T} \ .
\ee
 Next we estimate the difference of the characteristics $y_j(\cdot), \ j=1,2,$ which are the solutions to (\ref{C3}) corresponding to $\rho_j(\cdot), \ j=1,2$ respectively.   Setting $z(s)=y_1(s)-y_2(s), \ s\le t,$ we have from (\ref{C3}) that
\be \label{P3}
\frac{dz(s)}{ds} \ = \ -[g(s)+\rho_1(s)]z(s)-[\rho_1(s)-\rho_2(s)]y_2(s) \ , \quad 0\le s<t, \ z(t)=0 \ ,
\ee
where the function $g(\cdot)$ is given by
\be \label{Q3}
g(s) \ = \ \int_0^1 d\la \  h'(\la y_1(s)+(1-\la)y_2(s)) \  , \quad 0\le s\le t \ .
\ee
Evidently $g(s)=g(s,y)$ depends on the terminal condition $y_1(t)=y_2(t)=y> \ve_0$ as well as on $s$, and it is possible that $g(s,y)$ becomes unbounded as $y\ra\ve_0$.  However, in view of our assumptions on $h'(\cdot)$ and the inequality $\inf\rho_j(\cdot)\ge -h(\ve_0)/\ve_0+\ve,\  j=1,2,$ there is a constant $M_\ve$ independent of $T$  such that
\be \label{Q*3}
\int_0^t|g(s,y)| \ ds \ \le \ M_\ve \quad {\rm for \ } y>\ve_0, \ 0<t\le T \ .
\ee
The solution to (\ref{P3}) has the representation
\be \label{R3}
z(s) \ = \ \int_s^t \exp\left[\int_s^{s'} \{g(s'')+\rho_1(s'')\} \ ds''\right][\rho_1(s')-\rho_2(s')]y_2(s')\ ds' \ .
\ee
We now can estimate the LHS of (\ref{M3}) from (\ref{D3})  by using the fundamental theorem of calculus and (\ref{O3}), (\ref{Q*3}), (\ref{R3}).  We conclude from (\ref{M3}) that $K$ is a contraction if $T>0$ is sufficiently small, by using the inequality (\ref{BA2}) and the inequality
\be \label{S3}
 |I(\zeta_1(\cdot),\eta)-I(\zeta_2(\cdot),\eta)| \ \le C_{\ga}\eta^{-\ga}\|\zeta_1(\cdot)-\zeta_2(\cdot)\|_\infty \ ,
\ee
which follows from (\ref{AV2}). 

It follows from the contraction mapping theorem there exists for sufficiently small  $T>0$ a unique  $\rho(\cdot)\in \mathcal{E}_{\ve,T}$ such that $K\rho(\cdot)=\rho(\cdot)$. Setting this $\rho(\cdot)$ into the expressions (\ref{C3}), (\ref{D3}) we obtain a solution $\xi(y,t), \ y\ge\ve_0, \ 0\le t\le T,$ to the initial value problem. We see from (\ref{D3}), using the boundedness of $h(\cdot)$,  that $\sup_{y\ge\ve_0}\xi(y,t)<\infty$. We can obtain a formula for $D\xi(y,t)$ from the first variation equation for (\ref{C3}). Thus letting $Y(s,y), \ s\le t,$ denote the solution to (\ref{C3}) so as to indicate the dependence on the variable $y$, we  have that $DY(s,y)$ satisfies the linear variation equation
\be \label{T3}
\frac{d}{ds}DY(s,y) \ = \ -[h'(y(s))+\rho(s)]DY(s,y) \  , \quad 0\le s\le t, \ DY(t,y) \ = \ 1 \ .
\ee
Integrating (\ref{T3}) we obtain the formula
\be \label{U3}
DY(s,y) \ = \ \exp\left[\int_s^t \{h'(y(s'))+\rho(s')\}  \ ds' \right] \ .
\ee
Differentiating (\ref{D3}) we have from (\ref{U3}) the expression
\begin{multline} \label{V3}
D\xi(y,t) \ = \ \exp\left[\int_0^t h'(y(s))  \ ds \right]D\xi(y(0),0) \\
+\int_0^t ds \  h'(y(s))\exp\left[\int_s^t h'(y(s')) \ ds'\right]  \\
= \  \exp\left[\int_0^t h'(y(s))  \ ds \right]\left\{1+D\xi(y(0),0)\right\}-1 \ .
\end{multline}
Arguing as for the estimate (\ref{Q*3}), we see that the integral of $h'(\cdot)$ on the RHS of (\ref{V3}) is bounded uniformly in $y>\ve_0, \ 0<t\le T$.  We conclude then from (\ref{O3}) and the first expression on the RHS of (\ref{V3}) that $\sup_{y\ge\ve_0}y|D\xi(y,t)|<\infty$ for $0\le t\le T$.  We have therefore shown  that the solution to the initial value problem satisfies  $\|\xi(\cdot,t)\|_{1,\infty}<\infty$  when $0<t\le T$. We see from the second expression on the RHS of (\ref{V3}) that $1+D\xi(\cdot,t)\ge0$ for $0< t\le T$.  
\end{proof}
The local existence result of Lemma 3.1 can be extended to global existence by using Lemma 2.2.
\begin{proposition}
Let $h:[\ve_0,\infty)\ra\R$ be a  continuous positive decreasing function with $\lim_{y\ra\infty}h(y)=h_\infty>0$, which is $C^1$ on $(\ve_0,\infty)$ and  $\sup_{y\ge\ve_0+1}y|h'(y)|<\infty$.
Assume  that $\beta(\cdot,0), \eta(\cdot), \xi(\cdot,0)$ satisfy the conditions of Lemma 3.1, and in addition that $\inf[1+D\xi(\cdot,0)]\ge m$ for some $m>0$. Then with $I(\zeta(\cdot),\eta)$ given by (\ref{Y2}) and $\rho(\zeta(\cdot),\eta)$ by (\ref{T2}),  there exists for $0<t<\infty$ a unique non-negative solution $\xi(y,t), \ y\ge\ve_0,$ to the initial value problem for (\ref{A3}) with $\rho(t)=\rho(\xi(\cdot,t),\eta(t))$ such that $\|\xi(\cdot,t)\|_{1,\infty}<\infty$ and $\inf D\xi(\cdot,t)\ge -1$.

The solution $\xi(\cdot,t), \ t>0,$ has the property
\be \label{V*3}
\sup_{t\ge 0}\left\{\|\xi(\cdot,t)\|_\infty+\|D\xi(\cdot,t)\|_\infty\right\}<\infty, \quad \inf_{ y\ge \ve_0, \ t>0}[1+D\xi(y,t)] \ > \ 0 \ .
\ee
Furthermore there exist constants 
$\del,M>0$, depending only on $m$ and $\|\xi(\cdot,0)\|_{1,\infty}$, such that
\be \label{Y3}
\del-\frac{h(\ve_0)}{\ve_0} \ \le \ \rho(\xi(\cdot,t),\eta(t)) \ \le \ M  \quad {\rm for \ } t>0 \ .
\ee
\end{proposition}
\begin{proof}
Note that $h(\cdot)$ satisfies the assumptions of Lemma 3.1. To prove existence and uniqueness of the solution $\xi(\cdot,t)$ for all $t>0$ it is sufficient to show for any $T>0$ there exists $A(T)>0$, depending only on $\|\xi(\cdot,0)\|_{1,\infty}, \ m$ and $T$, such that $\|\xi(\cdot,t)\|_{1,\infty}\le A(T)$ for all $0<t<T$. If that is the case then by Lemma 3.1 we may extend the solution of the initial value problem beyond $T$, whence the solution exists globally in time. 

We see from (\ref{V3}) that $\inf[1+D\xi(\cdot,t)]\ge 0$ for $0<t<T$ and there exists $A_1>0$, depending only on $\|D\xi(\cdot,0)\|_\infty$, such that $\sup_{0<t<T}\|D\xi(\cdot,t)\|_\infty\le A_1$.  The lower bound in (\ref{Y3})  follows from this and (\ref{AS2})  of Lemma 2.2. It also follows from the lower bound that for any $T>0$ there exists $A(T)$ such that  $\sup_{y\ge \ve_0}|yD\xi(y,t)|\le A(T)$ if $0<t<T$.  We see this by using (\ref{N3}) and the first expression on the RHS of (\ref{V3}). 

Next observe from (\ref{D3}) that for any positive $\nu<\min\{1,T\}$, there exists a constant $C_\nu>0$, depending only on $\nu$ and  $\|\xi(\cdot,0)\|_\infty \ $, such that
\be \label{W3}
0 \ < \ \sup_{y\ge \ve_0} \xi(y,t) \ \le \ C_\nu\inf_{y\ge \ve_0} \xi(y,t) \ , \quad \nu<t<T \ .
\ee
From (\ref{AT2}) of Lemma 2.2 there exists $\del,K>0$, depending only on   $\|D\xi(\cdot,0)\|_\infty$,  such that for $0<t<T$,  if $\inf\xi(\cdot,t)\ge K$ then $\rho(\xi(\cdot,t),\eta(t))\ge \del$.  We choose now $K$ sufficiently large so  there exists $T_\nu>\nu>0$ such that $\sup_{y\ge\ve_0}\xi(y,t)< C_\nu K$ for $0<t<T_\nu$. Letting $[0,T_\nu)$ be the maximal such interval, we either have that $T_\nu=\infty$ or 
$\sup_{y\ge\ve_0}\xi(y,t)=C_\nu K$. Let us suppose that $T_\nu<\infty$ and $\sup_{y\ge\ve_0}\xi(y,t)>C_\nu K$ for $T_\nu<t<T'_\nu$.  From (\ref{W3}) we have that $\inf_{y\ge\ve_0}\xi(y,t)>K$ for $T_\nu<t<T'_\nu$. It follows now from (\ref{D3}) that
\be \label{X3}
\sup_{y\ge\ve_0}\xi(y,t) \ \le \ e^{-\del (t-T_\nu)}C_\nu K+
\frac{1- e^{-\del (t-T_\nu)}}{\del}\|h(\cdot)\|_\infty \ , \quad T_\nu<t<T'_\nu \ .
\ee
Evidently the RHS of (\ref{X3}) is bounded independent of $T'_\nu>T_\nu$. Iterating this argument we conclude there exists $A_2>0$, depending only on $\|\xi(\cdot,0)\|_{1,\infty} \ $, such that $\sup_{0<t<T}\|\xi(\cdot,t)\|_\infty\le A_2$.   We have established the first inequality of (\ref{V*3}). 

To prove the second inequality of (\ref{V*3}), let $y_\del>\ve_0$ be such that
\be \label{AB3}
y\left[\del-\frac{h(\ve_0)}{\ve_0}\right]+h(y) \ \ge\  \frac{\ve_0\del}{2} \quad {\rm for \ }\ve_0\le y\le y_\del \ .
\ee
From (\ref{C3}), the lower bound in (\ref{Y3}), and (\ref{AB3}) it follows that the characteristic  $y(s), \ 0\le s\le  t,$ with $y(t)=y$ satisfies $y(s)\ge y_\del, \ 0\le s\le t,$ if $y\ge y_\del$.  If $\ve_0\le y<y_\del$ then $y'(s)\le-\ve_0\del/2$ when $y(s)<y_\del$. It follows that
\be \label{AC3}
\int_0^t |h'(y(s))| \ ds \ \le \ \frac{2}{\ve_0\del}\int_{\ve_0}^{y_\del}|h'(y)| \ dy + 
\sup_{y\ge y_\del}[y|h'(y)|]\int_0^t \frac{ds}{y(s)} \ .
\ee
From (\ref{D3}), (\ref{N3}) we have that
\be \label{AD3}
\int_0^t \frac{ds}{y(s)}  \
\le \  \frac{1}{y}\int_0^t ds \  \exp\left[-\int_s^t \rho(s') \ ds'\right] \ \le
\  \frac{1}{yh_\infty}\xi(y,t) \ .
\ee
Now the first inequality of (\ref{V*3}) and (\ref{AC3}), (\ref{AD3}) imply there is a constant $C$ such that
\be \label{AE3}
\int_0^t |h'(y(s))| \ ds \ \le  \ C \quad {\rm for \ all \ } t>0, \ y(t)=y\ge\ve_0 \ .
\ee
We conclude from (\ref{V3}), (\ref{AE3}) that $\inf[1+D\xi(\cdot,t)]\ge e^{-C}m$ for $t>0$, whence the second inequality of (\ref{V*3}) holds. 
The upper bound in (\ref{Y3}) follows from (\ref{V*3}) and 
(\ref{AU2}) of Lemma 2.2. 
\end{proof}
In order to show the solution $\xi(\cdot,t)$ of Proposition 3.1 satisfies $\sup_{t>0}\|\xi(\cdot,t)\|_{1,\infty}<\infty$,   we need a strict positivity assumption on the function $\rho(\cdot)$ at large time.  
\begin{lem}
Let $\rho:[0,\infty)\ra\R$ be a continuous function satisfying $\inf \rho(\cdot)>-h(\ve_0)/\ve_0$ i.e. a strict version of (\ref{B3}), and $\del_0,\tau_0>0$ have the property that
\be \label{Q4}
\int_s^t\rho(s') \ ds' \ \ge \ \del_0 (t-s)\quad {\rm for \ } t\ge \tau_0, \ 0\le s\le t-\tau_0 \ .
\ee
Assume the function $h:[\ve_0,\infty)\ra\R$ is positive continuous decreasing,   $C^1$ on $(\ve_0,\infty)$ and  satisfies the inequality 
\be \label{R4}
\|h(\cdot)\|_\infty + \sup_{y\ge \ve_0+1}y|h'(y)| \ <  \ \infty \ . 
\ee
Then there are constants $C_1,C_2$, independent of $\xi(\cdot,0)$,  such that the solution $\xi(\cdot,t), \ t\ge 0,$ to the initial value problem for (\ref{A3})  satisfies the inequality
\be \label{S4}
\|\xi(\cdot,t)\|_{1,\infty} \ \le \ C_1e^{-\del_0 t}\|\xi(\cdot,0)\|_{1,\infty}+C_2 \ , \quad t\ge 0 \ .
\ee
Assume in addition to (\ref{Q4}) that $h(\cdot)$ is convex, $C^2$ on $(\ve_0,\infty)$ and
\be \label{T4}
\sup_{y\ge \ve_0+1}y^2h''(y) \ < \ \infty \ .
\ee
If $\xi(\cdot,0)$ is $C^2$ on $(\ve_0,\infty)$ then $\xi(\cdot,t)$ is $C^2$ on $(\ve_0,\infty)$, and there are positive constants $\nu_0,C_3,C_4$, independent of $\xi(\cdot,0)$,  such that 
\begin{multline} \label{U4}
\int_{\ve_0}^{\ve_0+\nu_0}|D^2\xi(y,t)|  \ dy +\sup_{y\ge\ve_0+\nu_0}y^2|D^2\xi(y,t)| \ \le \\
 C_3e^{-\del_0 t}\left\{ \|\xi(\cdot,0)\|_{1,\infty}+\int_{\ve_0}^{\ve_0+\nu_0}|D^2\xi(y,0)|  \ dy +\sup_{y\ge\ve_0+\nu_0}y^2|D^2\xi(y,0)| \ \right\}
 +C_4 \ , \quad t\ge 0 \ .
\end{multline}
\end{lem}
\begin{proof}
Multiplying (\ref{V3}) by $y$ we have from (\ref{N3}), upon using the assumption that $h(\cdot)$ is decreasing,  the inequality
\begin{multline} \label{V4}
y|D\xi(y,t)| \ \le \ \exp\left[-\int_0^t \rho(s) \ ds\right]y(0)|D\xi(y(0),0)| \\
+\int_0^t ds \  y(s)|h'(y(s))|\exp\left[-\int_s^t \rho(s') \ ds'\right]  \ , \quad y\ge \ve_0 \ .
\end{multline}
The bound (\ref{S4}) follows now from (\ref{Q4}) by using (\ref{A3}) to estimate $\|\xi(\cdot,t)\|_\infty$ and (\ref{V4}) to estimate $\sup_{y\ge\ve_0}y|D\xi(y,t)|$.  Note that we need to use the  assumption $\inf\rho(\cdot)>-h(\ve_0)/\ve_0$ so as to estimate the second term on the RHS of (\ref{V4}) when $y$ is close to $\ve_0$. 

Differentiating (\ref{V3}) with respect to $y$, we have  from (\ref{U3})  that
\begin{multline} \label{W4}
D^2\xi(y,t) \ = \ \exp\left[\int_0^t \{\rho(s)+2h'(y(s))\} \ ds\right]D^2\xi(y(0),0) \\
+\exp\left[\int_0^th'(y(s)) \ ds\right]\left\{1+D\xi(y(0),0)\right\} \times \\
\int_0^t ds \ h''(y(s))\exp\left[\int_s^t \{\rho(s')+h'(y(s'))\} \ ds'\right]  \ .  
\end{multline}
Multiplying (\ref{W4}) by $y^2$ we have from (\ref{N3}) the inequality
\begin{multline} \label{X4}
y^2|D^2\xi(y,t)| \ \le \ \exp\left[-\int_0^t \rho(s) \ ds\right]y(0)^2|D^2\xi(y(0),0)| \\
+\left\{1+|D\xi(y(0),0)|\right\}\int_0^t ds \ y(s)^2h''(y(s))\exp\left[-\int_s^t \rho(s') \ ds'\right]   \ .
\end{multline}
We choose $\nu_0>0$ so that for some $\del>0$, 
\be \label{Y4}
\inf\rho(\cdot) \ \ge \ - \frac{h(\ve_0+\nu_0)}{\ve_0+\nu_0}+\del \ .
\ee
We see from (\ref{C3}), (\ref{Y4}) that
\begin{multline} \label{Z4}
\frac{dy(s)}{ds}\le  - \min\{\del\ve_0, \ h(\ve_0+\nu_0)\} \ {\rm if \ \ } y(s)\le \ve_0+\nu_0 \ , \\
y(s)\ge \ve_0+\nu_0  \ \ {\rm for \ } s\le t \ \ 
 \ {\rm if \ \ } y(t)=y\ge \ve_0+\nu_0 \ .
\end{multline}
The inequality (\ref{U4}) follows from (\ref{Q4}), (\ref{X4}), (\ref{Z4}) by observing that
\be \label{AA4}
\int_{\{\ve_0<y<y'<\ve_0+1\}} h''(y') \ dy'dy \ < \ \infty \ .
\ee
\end{proof}
\begin{theorem}
In addition to the assumptions of Proposition 3.1, suppose the function $\eta:[0,\infty)\ra(0,1]$ is $C^1$ and satisfies 
\be \label{Z3}
|\eta'(t)|\le C\eta(t),   \quad \lim_{t\ra\infty}\eta(t)=0, 
\ee
for some positive constant $C$. Let $\xi(\cdot,t), \ t>0$, be the solution to the initial value problem for (\ref{A3}) with $\rho(t)=\rho(\xi(\cdot,t),\eta(t)), \ t\ge 0$.  Then one has
\be \label{AA3}
\sup_{t>0}\|\xi(\cdot,t)\|_{1,\infty}<\infty, \quad \lim_{T\ra\infty} \frac{1}{T}\int_0^T\rho(\xi(\cdot,t),\eta(t)) \ dt  \ = \ \frac{1}{p} \ .
\ee
\end{theorem}
\begin{proof}
We observe from (\ref{T2}), \ (\ref{A3}) that
\begin{multline} \label{AF3}
\frac{1}{p}\frac{d}{dt}\log I(\xi(\cdot,t),\eta(t)) \ = \ \rho(t)-\frac{1}{p} \\
+\frac{1}{p}\left[\frac{d\eta(t)}{dt}+\rho(t)\eta(t)\right]\frac{1}{I(\xi(\cdot,t),\eta(t))}\frac{\pa I(\xi(\cdot,t),\eta(t))}{\pa\eta} \ , \quad {\rm where \ } \rho(t) \ =  \ \rho(\xi(\cdot,t),\eta(t)) \ .
\end{multline}
We have now that 
\be \label{AG3}
\lim_{\eta\ra 0} \sup_{\zeta(\cdot)\ge 0}\frac{1}{I(\zeta(\cdot),\eta)}\eta\left|\frac{\pa I(\zeta(\cdot),\eta)}{\pa \eta}\right| \ = \ 0 \ .
\ee
The limit  (\ref{AG3}) follows from (\ref{X2}), (\ref{Y2}), (\ref{AN2}), (\ref{AP2}). The limit in (\ref{AA3}) can now be obtained as a consequence of Lemma 2.1 and (\ref{Y3}), (\ref{Z3}), (\ref{AG3}) by integrating (\ref{AF3}). In fact from  (\ref{Z3}) there is a constant $c>0$ such that $1\ge\eta(t)\ge ce^{-Ct}, \ t\ge 0$. In view of (\ref{V*3}) we may apply Lemma 2.1 to conclude that for any $\nu>0$ the integral of the LHS of (\ref{AF3}) over the interval $[0,T]$ is bounded by $\nu T$ at large $T$. The inequality in (\ref{AA3}) follows from (\ref{S4}) of Lemma 3.2 once we show that (\ref{Q4}) holds.  We prove (\ref{Q4}) in the same way we proved the limit in (\ref{AA3}) by integrating (\ref{AF3}) over the interval $[s,t]$. The inequality follows upon using the differential inequality in (\ref{Z3}) and Lemma 2.5.
\end{proof}
\begin{proof}[Proof of Theorem 1.2]
We first show that the function $u(\cdot)$ defined by (\ref{B2}) satisfies $\lim_{t\ra\infty} u(t)^{-1}=0$. To see this we use the identity from  (2.2) of \cite{cn} that $u(t)^{-1}=\pa F(1,t)/\pa x$. Now the function $x\ra F(x,t), \ 0\le x<1,$ is increasing and convex for all $t>0$.  From (2.2) and (2.44) of \cite{cn} we see there exists a constant $C$ such that
\be \label{AH3}
\frac{\pa F(1,t)}{\pa x} \ \le \  C\frac{\pa F(0,t)}{\pa x} \ ,  \quad t\ge 0 \  .
\ee
We also have from Lemma 2.1 of \cite{cn} that $\lim_{t\ra\infty} F(0,t)=1=F(1,t)$, whence (\ref{AH3}) implies that $\lim_{t\ra\infty}\pa F(1,t)/\pa x=0$.  We may now choose $T_0>0$ such that the function $\eta(t)=u(t)^{-1}$ satisfies $0<\eta(t)\le 1$ for $t\ge T_0$ and $\lim_{t\ra\infty}\eta(t)=0$. Since $\lim_{t\ra\infty}F(0,t)=1$, we may also choose $T_0$ sufficiently large so that the function $x\ra\beta(x,T_0), \ 0\le x<1,$ satisfies (\ref{U2}) and $\lim_{x\ra1}\beta(x,T_0)=p/(p+1)$. 

Next we show that the function $\xi(\cdot,T_0)$ defined by (\ref{H2})  satisfies the assumptions of Theorem 3.1 for the initial data of $\xi(\cdot,t)$. Since $\phi(\cdot), \ \psi(\cdot)$ are assumed to satisfy (\ref{D1}), (\ref{E1}), (\ref{I1}) it follows from Lemma 6.1 of \cite{cn}  that the function $h:[\ve_0,\infty)\ra\R$ defined by (\ref{I2})  is continuous, positive decreasing, $\lim_{y\ra\infty}h(y)=h_\infty=1$ and 
$\sup_{y>\ve_0+1}y|h'(y)|<\infty$. Furthermore, the function $g(\cdot,\cdot)$ in (\ref{D2}) is related to $h(\cdot)$ by $g(z,u)=-\al_0uh(z/\al_0u)$. We then infer from (\ref{E2}) that
\be \label{AI3}
z+\al_0h_\infty\int_0^tu(s) \ ds \  \le \ \hat{F}(z,t) \ \le \ z+\al_0h(\ve_0)\int_0^tu(s) \ ds \ .
\ee
We conclude from (\ref{H2}) and (\ref{AI3}) that
\be \label{AJ3}
\frac{h_\infty}{u(T_0)}\int_0^{T_0}u(s) \ ds  \ \le  \ \inf\xi(\cdot,T_0) \ \le \sup\xi(\cdot,T_0) \ \le \ \frac{h(\ve_0)}{u(T_0)}\int_0^{T_0}u(s) \ ds \  .
\ee
To obtain bounds on $D\xi(\cdot,T_0)$ we differentiate (\ref{H2}) to obtain the identity
\be \label{AK3}
1+D\xi\left(\frac{z}{\al_0u(t)},t\right) \ = \ \frac{\pa \hat{F}(z,t)}{\pa z} \ .
\ee
Differentiating (\ref{F2}) we have that
\be \label{AL3}
\frac{\pa \hat{F}(f(x)u(t),t)}{\pa z} \ = \ \frac{f(F(x,t))\psi(x)}{f(x)\psi(F(x,t))u(t)}\frac{\pa F(x,t)}{\pa x} \ .
\ee
From (2.2) of \cite{cn} we see there are positive constants  $c_0,C_0$ such that
\be \label{AM3}
c_0 \ \le \ \frac{\pa F(x,T_0)}{\pa x} \ \le \ C_0 \quad {\rm for \ } 0\le x\le 1 \ .
\ee
We conclude from (\ref{AL3}), (\ref{AM3}) there are positive constants $m_0,M_0$ such that
\be \label{AN3}
m_0 \ \le \ 1+\inf D\xi(\cdot,T_0) \ \le \ 1+\sup D\xi(\cdot,T_0) \ \le  \ M_0 \ .
\ee
To show that $\sup_{y\ge\ve_0}|yD\xi(y,T_0)|<\infty$ we use the representation  (6.23) from \cite{cn},
\begin{multline} \label{AN*3}
\frac{\pa\hat{F}(z,t)}{\pa z} \ = \ \exp\left[-\int_0^t \frac{\pa g(z(s),u(s))}{\pa z} \ ds \ \right] \\
= \  \exp\left[\int_0^th'(y(s)) \ ds \ \right] \ , \quad {\rm where \ } z \ = \ \al_0u(t)y \ .
\end{multline}
Hence we have from (\ref{AK3}), (\ref{AN*3}) that
\be \label{AO*3}
D\xi(y,t) \ = \ \exp\left[\int_0^th'(y(s)) \ ds \ \right] -1 \ ,
\ee
which is the same formula as (\ref{V3}) in the case $\xi(\cdot,0)\equiv 0$.  Applying Taylor's theorem to (\ref{AO*3}), we conclude from (\ref{N3})  and the inequality $\sup_{y\ge \ve_0+1}y|h'(y)|<\infty$ that
$\sup_{y\ge\ve_0}|yD\xi(y,t)|<\infty$ for all $t\ge 0$. 

To obtain Theorem 1.2 as a consequence of Theorem 3.1 we observe from (\ref{B2}), (\ref{I2}) that
\be \label{AO3}
\rho(\xi(\cdot,t),\eta(t)) \ = \ \phi'(1)-\psi'(1)\kappa(t) \ , \quad t\ge T_0 \ .
\ee 
From (\ref{I2}) we see that $h(\ve_0)/\ve_0=-\phi'(1)$, whence (\ref{Y3}) implies that 
$\kappa(t)\ge\del/|\psi'(1)|>0$ for $t\ge T_0$. The upper bound in (\ref{Y3}) implies that $\kappa(t)\le[M-\phi'(1)]/|\psi'(1)|, \ t\ge T_0$.  
The limit (\ref{J1}) is a consequence of (\ref{AA3}) upon using the relations (\ref{AO3}) and  $\beta_0=p/(p+1)$.  Note for $u(\cdot)$ given by (\ref{B2}) that $\eta(t)=u(t)^{-1}$ satisfies 
\be \label{AP3}
\frac{d\eta(t)}{dt}+\rho(t)\eta(t) \ = \ 0 \ , \quad \rho(t) \ = \ \rho(\xi(\cdot,t),\eta(t)) \ .
\ee
It follows from the already established properties of $\eta(\cdot)$ and (\ref{Y3}), (\ref{AP3}) that (\ref{Z3}) holds. 
\end{proof}
\begin{rem}
Observe that Theorem 3.1 yields a reduction of dimension in dynamics from the original LSW problem,  since the function $\eta(\cdot)$ is only required to satisfy (\ref{Z3}).  The system (\ref{A3}), (\ref{AP3}) with $\rho(t)=\rho(\xi(\cdot,t),\eta(t))$ given by (\ref{T2}) and $I(\zeta(\cdot),\eta)$ by (\ref{Y2})  is equivalent to the LSW dynamics (\ref{A1}), (\ref{B1}). 
\end{rem}

\vspace{.1in} 

\section{Local Asymptotic Stability}
The limit (\ref{AA3}) of Theorem 3.1 is a {\it weak global} asymptotic stability result for solutions to the PDE (\ref{A3}) with 
$\rho(t)=\rho(\xi(\cdot,t),\eta(t)), \ t\ge 0,$ where $\rho(\zeta(\cdot),\eta)$ is defined by (\ref{T2}) and $I(\zeta(\cdot),\eta)$ by (\ref{Y2}). In this section we will prove a {\it strong local} asymptotic result, showing that $\xi(\cdot,t)$ converges as $t\ra\infty$  to the equilibrium solution $\xi_p(\cdot)$ defined in (\ref{P1}).  In order to do this we shall need to impose further assumptions on the function $h(\cdot)$, beyond those required for Theorem 3.1. 

We proceed in parallel to the argument followed in $\S3$ of \cite{cd}.  We first linearize (\ref{A3}) with $\rho(\zeta(\cdot),\eta)$ given by (\ref{T2}), (\ref{Y2}) about the equilibrium $\xi_p(\cdot)$ and study its stability.   To do this we denote by $A,B$ the operators
\be \label{A4}
A\zeta(y) \ = \ \zeta(y)-y\frac{d\zeta(y)}{dy} \ , \quad B\zeta(y) \ = \ \frac{1}{p}\zeta(y)-\left[h(y)+\frac{y}{p}\right]\frac{d\zeta(y)}{dy} \ . 
\ee
Observe now that the functional $\rho(\cdot)$ of (\ref{T1}) satisfies the identity
\be \label{B4}
\rho(\zeta(\cdot))-\frac{1}{p} \ = \ -\frac{\left[dI(\zeta(\cdot)), \ B\{\zeta(\cdot)-\xi_p(\cdot)\}\right]}{pI(\zeta(\cdot))+\left[dI(\zeta(\cdot)),  \ A\zeta(\cdot)\right]} \ .
\ee
Let $\xi(\cdot,t), \ t\ge 0,$ be the solution of  (\ref{A3}) with 
$\rho(t)=\rho(\xi(\cdot,t),\eta(t)), \ t\ge 0,$ where $\rho(\zeta(\cdot),\eta)$ is defined by (\ref{T2}), (\ref{Y2}), which is constructed in Proposition 3.1. We denote by $\ga:[0,\infty)\ra \R$  the function
\be \label{C4}
\ga(t) \ = \ \rho(\xi(\cdot,t),\eta(t))-\rho_p(\xi(\cdot,t))  \ , \quad  t\ge 0\  ,
\ee
where $\rho_p(\zeta(\cdot))$ is given by the RHS of (\ref{T1}) with  $I(\cdot)$ as in (\ref{L2}), (\ref{M2}).
We shall regard $\ga(\cdot)$ as a {\it given} function, for which we can derive some properties using Lemma 2.4 and Theorem 3.1.
We may rewrite (\ref{A3}) with $\rho(t)=\rho(\xi(\cdot,t),\eta(t))$  as
\begin{multline} \label{D4}
\frac{\pa \xi(\cdot,t)}{\pa t}+[B+\ga(t)A]\{\xi(\cdot,t) -\xi_p(\cdot)\} \\
-\frac{\left[dI_p(\xi(\cdot,t)), \ B\{\xi(\cdot,t)-\xi_p(\cdot)\}\right]}{pI_p(\xi(\cdot,t)) 
+\left[dI_p(\xi(\cdot,t),  \ A\xi(\cdot,t)\right]}A\xi(\cdot,t)  
+\ga(t)A\xi_p(\cdot) \ = \ 0 \ ,
\end{multline} 
where $I_p(\cdot)$ is the functional $I(\cdot)$ defined by (\ref{L2}), (\ref{M2}). 
Setting $\tilde{\xi}(\cdot,t)=\xi(\cdot,t)-\xi_p(\cdot)$, we see that the linearization of (\ref{D4}) about $\xi_p(\cdot)$   is given by 
\be \label{E4}
\frac{d\tilde{\xi}(t)}{dt} +[B+\ga(t)A]\tilde{\xi}(t) -[dI_p(\xi_p),B\tilde{\xi}(t)]\frac{A\xi_p}{pI_p(\xi_p)+[dI_p(\xi_p), \ A\xi_p]}+\ga(t)A\xi_p
 \ = \ 0 \ .
\ee

If we set $\ga(\cdot)\equiv 0$ in (\ref{E4}) then the solution satisfies the equation
\be \label{F4}
\tilde{\xi}(t) \ = \ e^{-Bt}\tilde{\xi}(0) 
+\int_0^t ds \ [dI_p(\xi_p),B\tilde{\xi}(s)]\frac{e^{-B(t-s)}A\xi_p}{pI_p(\xi_p)+[dI_p(\xi_p), \ A\xi_p]}
 \ .
\ee
Letting $u(t)=[dI_p(\xi_p),B\tilde{\xi}(t)]$, we see from  (\ref{F4}) that $u(\cdot)$ is the solution to the Volterra integral equation 
\be \label{G4}
u(t)+\int_0^t  K(t-s)u(s) \ ds \ = \ g(t)   \ , \quad t>0,
\ee
where the functions $K,g$ are given by
\begin{multline} \label{H4}
K(t) \   = \  -\frac{[dI_p(\xi_p),e^{-Bt}BA\xi_p]}{pI_p(\xi_p)+[dI_p(\xi_p), \ A\xi_p]}  \ , \ 
g(t) \ = \ [dI_p(\xi_p),e^{-Bt}B\tilde{\xi}(0)] \quad t\ge 0.
\end{multline}
We proceed now as in \cite{cd} to obtain properties of the function $h(\cdot)$, which will imply that the solution to (\ref{G4}) satisfies $\lim_{t\ra\infty}u(t)=0$. 
\begin{lem}
 Assume that  the function $h:[\ve_0,\infty)\ra\R$ is positive, continuous, decreasing, convex, and  $C^2$ on $(\ve_0,\infty)$. Assume further that $h(\cdot)$ satisfies the inequalities 
\be \label{I4}
\sup_{y\ge\ve_0+1} \{y|h'(y)|\}<\infty, \quad  \left[\frac{y}{p}+h(y)\right]yh''(y)+h'(y)[h(y)-yh'(y)]  \ \ge  \ 0, \quad y> \ve_0 \ .
\ee
Then $K(\cdot)$ defined by (\ref{H4}) has the property that the function $t\ra e^{t/p}K(t), \ t\ge0,$ is positive and decreasing.  
\end{lem}
\begin{proof}
We first show that the denominator in the formula (\ref{H4}) for $K(\cdot)$ is positive. To see this we use the identity 
\be \label{J4}
pI_p(\xi_p)+[dI_p(\xi_p), \ A\xi_p] \ = \ -[dI_p(\xi_p), \ y(1+D\xi_p)] \ ,
\ee
which is a particular case of (\ref{AQ2}).  Next we have from (\ref{N1}) that
\be \label{K4}
1+D\xi_p(y) \ = \  \frac{\xi_p(y)+y}{ph(y)+y} \ , \quad y\ge \ve_0 \ .
\ee
Since $ph_\infty\le \xi_p(y)\le ph(y), \ y\ge \ve_0,$ it follows from (\ref{K4}) that there are positive upper and lower bounds on $1+D\xi_p(y)$, which are uniform for $y\ge \ve_0$.  Now from (\ref{J4}) we infer that the denominator is finite and positive. 

Next we show that $K(\cdot)$ is a positive function. It is evident from (\ref{N1}), (\ref{K4}) that $A\xi_p(\cdot)$ is positive and $\|A\xi_p\|_\infty<\infty$. We can also see that the function $BA\xi_p(\cdot)$ is positive. To show this we use the commutation relation
\be \label{L4}
BA-AB \ =  \ \left[h(y)-yh'(y)\right]\frac{d}{dy} \ .
\ee
From (\ref{N1}), (\ref{L4}) we have that
\begin{multline} \label{M4}
BA\xi_p(y) \ = \ Ah(y) +\left[h(y)-yh'(y)\right]\frac{d\xi_p(y)}{dy} \\
= \ \left[h(y)-yh'(y)\right][1+D\xi_p(y)] \ ,
\end{multline}
whence the positivity of $BA\xi_p(\cdot)$ follows from (\ref{K4}), (\ref{M4}). Since $dI_p(\xi_p)$ is a negative function, the positivity of  $BA\xi_p(\cdot)$ implies the positivity of $K(\cdot)$.  Note that the first inequality of (\ref{I4}) is needed in order to guarantee that $K(t)$ is finite for $t\ge 0$. 

To show that the function $t\ra e^{t/p}K(t)$ is decreasing we  need to show that
\begin{multline} \label{N4}
(B-1/p)BA\xi_p(y) \ = \ B\left[\frac{1}{p}A\xi_p(y)-yh'(y)\frac{d\xi_p(y)}{dy}-yh'(y)\right]-\frac{1}{p}BA\xi_p(y) \\
= \ \left[\frac{y}{p}+h(y)\right] \left[yh''(y)+h'(y)\right]\frac{d\xi_p(y)}{dy}-yh'(y)B\frac{d\xi_p(y)}{dy}-B[yh'(y)]  \ .
\end{multline}
is positive. Using the commutator relation
\be \label{O4}
B\frac{d}{dy}-\frac{d}{dy}B \ = \ \left[\frac{1}{p}+h'(y)\right]\frac{d}{dy} \ ,
\ee
we have from (\ref{N4}) that
\begin{multline} \label{P4}
(B-1/p)BA\xi_p(y) \ = \ \left\{\left[\frac{y}{p}+h(y)\right] \left[yh''(y)+h'(y)\right]-yh'(y)\left[\frac{1}{p}+h'(y)\right]\right\} \frac{d\xi_p(y)}{dy} \\
  -yh'(y)^2-B[yh'(y)]  \\
  = \  \left\{\left[\frac{y}{p}+h(y)\right]yh''(y)+h'(y)[h(y)-yh'(y)]\right\} \left[\frac{d\xi_p(y)}{dy}+1\right] \ .
\end{multline}
Now the second inequality of (\ref{I4}) and (\ref{K4}) imply that the expression in (\ref{P4}) is non-negative. 
\end{proof}
Similarly to Proposition 3.1 of \cite{cd} we have the following:
\begin{proposition}
Assume $h(\cdot)$ satisfies the conditions of Lemma 4.1.  Then the linear evolution equation (\ref{F4}) is asymptotically stable in the following sense: Let the initial data $\tilde{\xi}_0:[\ve_0,\infty)\ra\R$ satisfy $\|\tilde{\xi}_0(\cdot)\|_{1,\infty}<\infty$. Then there is for any $q>p$ a constant $C_q$ depending on $q$ such that
\be \label{AB4}
\|\tilde{\xi}(\cdot,t)\|_{1,\infty} \ \le C_qe^{-t/q}\|\tilde{\xi}_0(\cdot)\|_{1,\infty} \quad {\rm when \ } t\ge 0 \ .
\ee
\end{proposition}
\begin{proof}
Observe that for a function $\zeta:[\ve_0,\infty)\ra\R$ we have
\be \label{AC4}
e^{-Bt}\zeta(y) \ = \ e^{-t/p}\zeta(y_p(0)), \ \quad y\ge \ve_0,
\ee
where $y_p(\cdot)$ is the solution of (\ref{C3}) with $\rho(\cdot)\equiv 1/p$.  Evidently we have from (\ref{AC4}) that $\|e^{-Bt}\zeta(\cdot)\|_\infty\le e^{-t/p}\|\zeta(\cdot)\|_\infty$.  From (\ref{U3}) we have that
\be \label{AD4}
\frac{\pa}{\pa y} e^{-Bt}\zeta(y) \ = \ \exp\left[\int_0^th'(y_p(s)) \ ds\right]D\zeta(y_p(0)) \ .
\ee
From (\ref{N3}) we see that $y_p(0)\ge e^{t/p}y, \ y\ge \ve_0$, whence we conclude from (\ref{AD4}) that
\be \label{AE4}
\sup_{y\ge\ve_0}\left| y\frac{\pa}{\pa y} e^{-Bt}\zeta(y) \right| \ \le \ e^{-t/p}\sup_{y\ge \ve_0}\left|yD\zeta(y)\right| \ .
\ee
It follows now from (\ref{AC4}), (\ref{AE4}) that $\|e^{-Bt}\zeta(\cdot)\|_{1,\infty}\le e^{-t/p}\|\zeta(\cdot)\|_{1,\infty}$ for $t\ge 0$.  We conclude that the formula (\ref{H4}) for $g(t)$ is bounded as
\be \label{AF4}
\left|[dI_p(\xi_p),e^{-Bt}B\tilde{\xi}(0)]\right| \ \le \  Ce^{-t/p}\|\tilde{\xi}_0(\cdot)\|_{1,\infty}  \quad {\rm for \ } t \ge 0,
\ee
where $C$ is a constant. 

From Lemma 4.1 and results on Volterra integral equations \cite{grip},  we  see that the solution $u(t), \ t\ge 0,$ of (\ref{G4}) with $K(\cdot),  \ g(\cdot)$ given by (\ref{H4})   has the property that the function $t\ra e^{t/q}u(t)$ is bounded by a constant times $\|\tilde{\xi}_0(\cdot)\|_{1,\infty}$ for any $q>p$. 
 It follows now, upon using (\ref{AC4}) to estimate the RHS of (\ref{F4}), that $\|\tilde{\xi}(\cdot,t)\|_\infty\le C_qe^{-t/q}\|\tilde{\xi}_0(\cdot)\|_\infty$ for some constant $C_q$ depending on $q>p$. 
Differentiating (\ref{F4}) with respect to the space variable,  we have from (\ref{AD4}) that
\begin{multline} \label{AP4}
D\tilde{\xi}(y,t) \ = \ \exp\left[\int_0^th'(y_p(s)) \  ds \right]D\tilde{\xi}_0(y_p(0))+\\
\int_0^t ds \ \exp\left[\int_s^th'(y_p(s')) \ ds'\right][dI_p(\xi_p),B\tilde{\xi}(s)]\frac{DA\xi_p(y_p(s))}{pI_p(\xi_p)+[dI_p(\xi_p), \ A\xi_p]}    \ .
\end{multline}
To estimate the RHS of (\ref{AP4})  we note on differentiating (\ref{K4}) that
\begin{multline} \label{AQ4}
yDA\xi_p(y) \ = \ -y^2D^2\xi_p(y) \\
 = \ \frac{[ph'(y)+1]y^2\xi_p(y)}{[ph(y)+y]^2}-\frac{y^2D\xi_p(y)}{ph(y)+y}- \frac{py^2[h(y)-yh'(y)]}{[ph(y)+y]^2} \ .
\end{multline}
We may bound the RHS of (\ref{AQ4}) using (\ref{R4}) to obtain the inequality
\be \label{AR4}
\int_{\ve_0}^{\ve_0+1}|yDA\xi_p(y)|  \ dy +\sup_{y\ge \ve_0+1} |yDA\xi_p(y)| \ < \  \infty \ .
\ee
Using (\ref{AR4}) to bound the RHS of (\ref{AP4}), we conclude that  
\be \label{AS4}
\sup_{y\ge\ve_0}|yD\tilde{\xi}(y,t)| \ \le \ C_qe^{-t/q}\|\tilde{\xi}_0(\cdot)\|_\infty \quad {\rm for\ } t\ge 0 \ ,
\ee
where $C_q$ is a constant depending only on $q>p$. 
\end{proof}
We generalize the result of Proposition 4.1 to apply to the non-linear PDE (\ref{D4}) by considering (\ref{D4}) as a perturbation of (\ref{E4}) with $\ga(\cdot)\equiv 0$  of the form
\begin{multline} \label{AT4}
\frac{d\tilde{\xi}(t)}{dt} +[B+\{\ga(t)+\del_1(\tilde{\xi}(t))\}A]\tilde{\xi}(t) \\
+\left\{\ga(t)-\frac{[dI(\xi_p),B\tilde{\xi}(t)]+\del_2(\tilde{\xi}(t))}{pI(\xi_p)+[dI(\xi_p), \ A\xi_p]}\right\}A\xi_p
 \ = \ 0 \ ,
\end{multline}
where $\del_1(\cdot), \ \del_2(\cdot)$ are real valued functionals of $C^1$ functions $\tilde{\zeta}:[\ve_0,\infty)\ra\R$. If we take
\be \label{AU4}
\del_1(\tilde{\zeta}(\cdot)) \ = \ -\frac{\left[dI_p(\xi_p+\tilde{\zeta}), \ B\tilde{\zeta}\right]}{pI_p(\xi_p+\tilde{\zeta})+\left[dI_p(\xi_p
+\tilde{\zeta}),  \ A\{\xi_p+\tilde{\zeta}\}\right]} \  ,
\ee 
and
 \begin{multline} \label{AV4}
\del_2(\tilde{\zeta}(\cdot)) \ = \  \\
 \left[dI_p(\xi_p+\tilde{\zeta}), \ B\tilde{\zeta}\right]
\frac{pI_p(\xi_p)+\left[dI_p(\xi_p
),  \ A\xi_p\right]}
{pI_p(\xi_p+\tilde{\zeta})+\left[dI_p(\xi_p
+\tilde{\zeta}),  \ A\{\xi_p+\tilde{\zeta}\}\right]}- \left[dI_p(\xi_p), \ B\tilde{\zeta}\right] \  ,
\end{multline} 
then (\ref{D4}), (\ref{AT4}) are equivalent. 
\begin{lem}
Assume $h(\cdot)$ satisfies (\ref{M1}). Then the functionals $\del_1(\cdot), \ \del_2(\cdot)$ defined by (\ref{AU4}), (\ref{AV4})  are Lipschitz continuous close to $\zeta(\cdot)\equiv 0$ in the $m=1$ norm (\ref{E3}).  In particular,  there exist constants $C,\nu>0$ such that
\begin{eqnarray}  \label{AW4}
|\del_1(\tilde{\zeta}_1)-\del_1(\tilde{\zeta}_2)| \ &\le& \  C\|\tilde{\zeta}_1-\tilde{\zeta}_2\|_{1,\infty} \ , \\
|\del_2(\tilde{\zeta}_1)-\del_2(\tilde{\zeta}_2)| \ &\le& \  
C\{\|\tilde{\zeta}_1\|_{1,\infty}+\|\tilde{\zeta}_2\|_{1,\infty}\}\|\tilde{\zeta}_1-\tilde{\zeta}_2\|_{1,\infty} \ , \nonumber
\end{eqnarray}
provided  $\|\tilde{\zeta}_j\|_{1,\infty}<\nu, \ j=1,2$.
\end{lem}
\begin{proof}
The function $I_p(\zeta(\cdot))$ defined by (\ref{L2})  is infinitely differentiable with respect to $\zeta(\cdot)$, and has the property  there exist constants $C_1,\nu_1>0$  such that
\begin{multline} \label{AX4}
\left|I_p(\xi_p+\tilde{\zeta}_1)-I_p(\xi_p+\tilde{\zeta}_2) \right| +\|dI_p(\xi_p+\tilde{\zeta}_1)-dI_p(\xi_p+\tilde{\zeta}_2) \|_{L^1(\ve_0,\infty)} \\
 \le \  C_1\|\tilde{\zeta}_1-\tilde{\zeta}_2\|_\infty \quad {\rm for \ } \|\tilde{\zeta}_j\|_\infty <\nu_1, \ j=1,2.
\end{multline}
The result follows from (\ref{AX4}) and the inequality \\
$\sup_{y>\ve_0} |B\zeta(y)|\le \left[1/p+\sup_{y\ge\ve_0}h(y)/y\right]\|\zeta(\cdot)\|_{1,\infty}$.
\end{proof}
Let $\del:[0,\infty)\ra\R$ be a continuous function and consider the linear PDE
\be \label{AY4}
\frac{d\tilde{\xi}(t)}{dt} +[B+\del(t)A]\tilde{\xi}(t) -[dI_p(\xi_p),B\tilde{\xi}(t)]\frac{A\xi_p}{pI_p(\xi_p)+[dI_p(\xi_p), \ A\xi_p]}
 \ = \ 0 \ .
\ee
The results of Proposition 4.1 extend to solutions of (\ref{AY4}) provided $\|\del(\cdot)\|_\infty$ is sufficiently small. Parallel to Lemma 3.3 of \cite{cd} we have the following:
\begin{lem}
Assume $h(\cdot)$ satisfies the conditions of Lemma 4.1. Assume further that $\del:[0,\infty)\ra\R$ is a continuous function and $\|\del(\cdot)\|_\infty\le1/p$. Then the linear evolution equation (\ref{AY4}) with initial data $\tilde{\xi}_0:[\ve_0,\infty)\ra\R$ satisfying $\|\tilde{\xi}_0(\cdot)\|_{1,\infty}<\infty$ has a unique solution globally in time, $\tilde{\xi}(y,t;\del(\cdot)), \ y\ge \ve_0,t\ge0,$ which has $\|\tilde{\xi}(\cdot,t;\del(\cdot))\|_{1,\infty}<\infty$ for all $t\ge 0$.  For any $q>p$ there exists $C_q,\ve_q>0$ such that if $\|\del(\cdot)\|_\infty<\ve_q$ then
\be \label{AZ4}
\|\tilde{\xi}(\cdot,t;\del(\cdot)\|_{1,\infty} \ \le C_qe^{-t/q}
\|\tilde{\xi}_0(\cdot)\|_{1,\infty} \quad {\rm when \ \ } t\ge0 \ .
\ee
\end{lem}
We apply Lemma 4.3 as in \cite{cd} to obtain bounds on solutions to the non-linear PDE (\ref{AT4}).  
\begin{theorem}
Assume $h(\cdot)$ satisfies the conditions of Lemma 4.1 and $\ga:[0,\infty)\ra\R$ is a continuous function satisfying $\lim_{t\ra\infty}\ga(t)=0$. Let $\del_1(\cdot), \ \del_2(\cdot)$ be real valued functionals of $C^1$ functions $\tilde{\zeta}:[\ve_0,\infty)\ra\R$, which satisfy $\del_1(0)=\del_2(0)=0$ and the local Lipschitz conditions (\ref{AW4}). Then there exists $\ve>0$ such that the solution  $\tilde{\xi}(\cdot,t), \ t\ge 0,$ to the nonlinear evolution equation (\ref{AT4}) with initial data $\tilde{\xi}_0:[\ve_0,\infty)\ra\R$ satisfying $\|\tilde{\xi}_0(\cdot)\|_{1,\infty}+\|\ga(\cdot)\|_\infty<\ve$  has the property 
$\lim_{t\ra\infty}\|\tilde{\xi}(\cdot,t)\|_{1,\infty}=0$. If $|\ga(t)|\le Ce^{-\nu t}, \ t\ge 0,$ for some constants $C,\nu>0$ with $\nu<1/p$ then $\ve=\ve_\nu$ can be chosen depending on $\nu$ so that
$\|\tilde{\xi}(\cdot,t)\|_{1,\infty}\le C_\nu e^{-\nu t}, \ t\ge 0,$ for some constant $C_\nu$. 
\end{theorem}
\begin{proof}
We define the Green's function for  (\ref{AY4}) as $G(t,s;\del(\cdot)), \ 0\le s\le t,$ so $\tilde{\xi}(t)=G(t,s;\del(\cdot))\tilde{\xi}_0, \ t\ge s,$ is the solution to (\ref{AY4}) with initial data $\tilde{\xi}(s)=\tilde{\xi}_0$.  The solution to (\ref{AT4}) then satisfies the identity
\begin{multline} \label{BA4}
\tilde{\xi}(t) \ = \ G(t,0;\ga(\cdot)+\del_1(\cdot))\tilde{\xi}_0+ \\
\int_0^t\left[\frac{\del_2(\tilde{\xi}(s))}{pI(\xi_p)+[dI(\xi_p),A\xi_p]}-\ga(s)\right ]G(t,s;\ga(\cdot)+\del_1(\cdot))A\xi_p \ .
\end{multline}
We  estimate $\|\tilde{\xi}(t)\|_{1,\infty}$ just as in the proof of Theorem 3.1 of \cite{cd}. 
\end{proof}

\vspace{.1in}

\section{A Differential Delay Equation}
We shall formulate the problem of solving the initial value problem for (\ref{A3}), with $\rho(t)=\rho(\xi(\cdot,t),\eta(t))$ given by (\ref{T2}), as an initial value problem for a differential delay equation (DDE). To do this we define a function $I:[0,\infty)\ra\R^+$ by integration of the function $\rho(\cdot)$. Thus solving the initial value problem for the equation
\be \label{A5}
\frac{1}{p}\frac{d}{dt}\log I(t) \ = \ \rho(t)-\frac{1}{p} \ , \quad t>0, \ I(0)>0,
\ee
we obtain a formula for $I(\cdot)$ in terms of $\rho(\cdot)$ as
\be \label{B5}
\exp\left[\int_s^t \rho(s') \ ds'\right]  \ = \ e^{(t-s)/p}\left(\frac{I(t)}{I(s)}\right)^{1/p} \  .
\ee
We rewrite the characteristic equation (\ref{C3}) using the function $I(\cdot)$ by setting
\be \label{C5}
y(s) \ = \ \frac{z(s)}{v_t(s)} \ , \quad v_t(s) \ = \ \left(\frac{I(s)}{I(t)}\right)^{1/p} \ , \ 0\le s\le t.
\ee
From (\ref{A5}), (\ref{C5}) we see that (\ref{C3}) is equivalent to
\be \label{D5}
\frac{dz(s)}{ds} \ = \ -\frac{z(s)}{p}-v_t(s)h\left(\frac{z(s)}{v_t(s)}\right) \ ,  \quad s<t, \ \ z(t)=y \ .
\ee
The DDE is now obtained by rewriting (\ref{A5}) with $\rho(t)=\rho(\xi(\cdot,t),\eta(t))$ in the form
\be \label{E5}
\frac{1}{p}\frac{d}{dt}\log I(t) \ = \ \rho_p(\xi(\cdot,t))-\frac{1}{p}+\ga(t) \ , 
\ee
where $\ga(t)=\rho(\xi(\cdot,t),\eta(t))-\rho_p(\xi(\cdot,t)), \ t\ge 0,$ is assumed to be a known continuous function satisfying $\lim_{t\ra\infty}\ga(t)=0$.  As with (\ref{D4}), equation (\ref{E5}) can be seen to be equivalent to
\be \label{F5}
\frac{1}{p}\frac{d}{dt}\log I(t) \ = \ -\frac{\left[dI_p(\xi(\cdot,t)), \ B\{\xi(\cdot,t)-\xi_p(\cdot)\}\right]}{pI_p(\xi(\cdot,t)) 
+\left[dI_p(\xi(\cdot,t),  \ A\xi(\cdot,t)\right]}+\ga(t) \ . 
\ee
We see from the representations (\ref{D3}) for $\xi(\cdot,t)$ and (\ref{V3}) for $D\xi(\cdot,t)$ that the first term on the RHS of (\ref{F5}) is a function of $v_t(\cdot)$, whence (\ref{F5}) is a DDE. 

We define a function $F(t,y,v_t(\cdot))$ by
\begin{multline} \label{G5}
F(t,y,v_t(\cdot)) \ = \ \frac{1}{p}\int_0^t h\left(\frac{z(s)}{v_t(s)}\right) e^{-(t-s)/p}v_t(s) \ ds \\
-\left[ h(y)+\frac{y}{p}\right]\left\{\exp\left[\int_0^t  h'\left(\frac{z(s)}{v_t(s)}\right) \ ds\right]-1\right\} \ .
\end{multline}
We wish to estimate $F(t,y,1(\cdot))-h(y)$. In order to do this we denote by $z_p(s), \ s\le t,$ the solution to (\ref{D5}) when $v_t(\cdot)\equiv 1$.  From (\ref{U3}) we see that $F(t,y,1(\cdot))=B\tilde{F}(t,y)$ where
\be \label{H5}
\tilde{F}(t,y) \ = \ \int_0^t ds \  h\left(  z_p(s)  \right)  e^{-(t-s)/p} \ .
\ee
We compare $\tilde{F}(t,y)$ to the formula (\ref{P1}) for $\xi_p(y)$ by making the change of variable $y'=z_p(s)$.  Then (\ref{H5}) is the same as
\be \label{I5}
\tilde{F}(t,y) \ = \ \int_y^{z_p(0)} dy' \  \frac{ph(y')}{ph(y')+y'}  \ e^{-(t-s)/p} \ .
\ee
Next from (\ref{O1}) and (\ref{D5}) with $v_t(\cdot)\equiv 1$ we have that
\be \label{J5}
e^{-(t-s)/p} \ = \ \frac{\tilde{h}(y)}{\tilde{h}(y')} \ .
\ee
We conclude from (\ref{P1}), (\ref{I5}), (\ref{J5}) that
\be \label{K5}
\tilde{F}(t,y)-\xi_p(y) \ = \ -\tilde{h}(y)\int_{z_p(0)}^\infty \frac{ph(y')}{\left(ph(y')+y'\right)\tilde{h}(y')} \ dy' \  .
\ee
Applying $B$ to (\ref{K5}) we obtain the formula
\begin{multline} \label{L5}
F(t,y,1(\cdot))-h(y) \ = \ -\frac{e^{t/p}[ph(y)+y]\tilde{h}(y)h(z_p(0))}{\left[ph(z_p(0))+z_p(0)\right]\tilde{h}(z_p(0))} \exp\left[\int_0^t   h'(z_p(s)) \ ds\right] \\
= \ -\frac{[ph(y)+y]h(z_p(0))}{\left[ph(z_p(0))+z_p(0)\right]} \exp\left[\int_0^t   h'(z_p(s)) \ ds\right] \ .
\end{multline}
Since $z_p(0)\simeq e^{t/p}$ at large $t$, it is clear from (\ref{L5}) that $F(t,y,1(\cdot))-h(y)$ decays like $e^{-t/p}$ as $t\ra\infty$. We conclude from (\ref{D3}), (\ref{V3}) (\ref{G5}), (\ref{L5})  that
\be \label{M5}
B\xi(y,t)-B\xi_p(y) \ = \ F(t,y,v_t(\cdot))-F(t,y,1(\cdot)) + G(t,y,v_t(\cdot)) \ ,
\ee
where $G$ is given by the formula 
\begin{multline} \label{N5}
G(t,y,v_t(\cdot)) \ = \  \frac{1}{p}e^{-t/p}v_t(0)\xi\left( \frac{z(0)}{v_t(0)},0\right) \\
- \left(h(y)+\frac{y}{p}\right)\exp\left[\int_0^t   h'\left(z(s)/v_t(s)\right) \ ds\right]D\xi\left(\frac{z(0)}{v_t(0)},0\right) \\
-\frac{[ph(y)+y]h(z_p(0))}{\left[ph(z_p(0))+z_p(0)\right]} \exp\left[\int_0^t   h'(z_p(s)) \ ds\right] \ .
\end{multline}
From (\ref{M5})  we may rewrite the DDE equation (\ref{F5}) as
\be \label{O5}
\frac{1}{p}\frac{d}{dt}\log I(t) + f(t,v_t(\cdot)) \ = \ g(t,v_t(\cdot))+\ga(t) \ ,
\ee
where the functions $f,g$ are given by the formulae
\begin{eqnarray} \label{P5}
f(t,v_t(\cdot)) \ &=& \ \frac{\left[dI_p(\xi(\cdot,t)), \ F(t,\cdot,v_t(\cdot))-F(t,\cdot,1(\cdot))\right]}{pI_p(\xi(\cdot,t))+\left[dI_p(\xi(\cdot,t)),  \ A\xi(\cdot,t)\right]} \ , \\
g(t,v_t(\cdot)) \ &=& \ -\frac{\left[dI_p(\xi(\cdot,t)), G(t,\cdot,v_t(\cdot)) \ \right]}{pI_p(\xi(\cdot,t))+\left[dI_p(\xi(\cdot,t)),  \ A\xi(\cdot,t)\right]} \ . \nonumber
\end{eqnarray}

In the proof of global asymptotic stability for the DDE (\ref{O5}) the key property which needs to be established is monotonicity of the function $f(t,v_t(\cdot))$ in the following sense: 
$f(t,v_t(\cdot))\ge 0$ on the set $0<v_t(\cdot)\le 1$, and 
$f(t,v_t(\cdot))\le 0$ on the set $v_t(\cdot)\ge 1$. This property of 
$f(t,v_t(\cdot))$  holds to first order in $v_t(\cdot)-1$  provided the gradient $dF(t,y, v_t(\cdot);\tau), \ 0<\tau<t,$ of $F(t,y,v_t(\cdot))$ with respect to $v_t(\cdot)$ at $v_t(\cdot)\equiv 1$ is non-negative for $y\ge \ve_0$. It holds to all orders if we can show that
\begin{eqnarray} \label{Q5}
\sup_{0<v_t(\cdot)<1} F(t,y,v_t(\cdot)) \ &=& \ F(t,y,1(\cdot)) \quad {\rm for \ } y\ge \ve_0 \ , \\
\inf_{1<v_t(\cdot)<\infty} F(t,y,v_t(\cdot)) \ &=& \ F(t,y,1(\cdot)) \quad {\rm for \ } y\ge \ve_0 \  . \label{R5}
\end{eqnarray}
Using the results of $\S6$ we have the following:
\begin{proposition}
Assume the function $h(\cdot)$ satisfies the assumptions of Lemma 4.1, in particular that (\ref{I4}) holds. Then $dF(t,y, 1(\cdot);\cdot)$ is a non-negative function for all $t>0,  \ y\ge \ve_0$. If in addition the second inequality of (\ref{I4}) holds for all $p>0$ then (\ref{Q5}) is true. If the function
\be \label{S5}
y\ra\frac{y^2h''(y)}{h(y)-yh'(y)}  \quad {\rm decreases  \ for \ }y>\ve_0
\ee
then (\ref{R5}) is true. 
\end{proposition}
\begin{proof}
We see that the first integral in the formula (\ref{G5}) for $F(t,y,v_t(\cdot))$ is the same as (\ref{D6}) with $g(\cdot)\equiv h(\cdot)/p$. The condition (\ref{M6}) for non-negativity of the gradient is trivial in this case. To show non-negativity of the gradient of the second integral in (\ref{G5})  at $v_t(\cdot)\equiv 1$, one considers the functional (\ref{AC6}) with $g(\cdot)\equiv -h'(\cdot)$. In this case the second inequality of (\ref{I4}) and (\ref{AG6}) are equivalent.  Hence $dF(t,y, 1(\cdot);\cdot)$ is non-negative  if (\ref{I4}) holds.

The identity (\ref{Q5}) follows from propositions 6.1,6.2 and the associated remarks. The identity (\ref{R5}) is a consequence of Proposition 6.3 and associated remark.  Note that the function $h(\cdot)$ has domain $[\ve_0,\infty)$ and we wish to allow $h'(y)$ to diverge as $y\ra\ve_0$. We can still  apply propositions 6.1-6.3 in this case by  approximating $h(\cdot)$ with functions which have domain $(0,\infty)$, and also preserve the monotonicity and convexity properties required for the propositions. 
\end{proof}

\vspace{.1in}

\section{Some Optimal Control Problems}
Let $y>0, T\in\R$, and consider the dynamics
\be \label{A6}
\frac{dx(s)}{ds} \ = \ -\frac{1}{p}x(s)-v(s)h\left(\frac{x(s)}{v(s)}\right) \ , \ \ s<T, \quad x(T)=y \ ,
\ee
with terminal condition $x(T)=y$ and controller $v(\cdot)$. The function $h(\cdot)$ is assumed to be positive and decreasing. Let $g:(0,\infty)\ra\R^+$ be a positive decreasing function and for $y>0,t<T$ define the function
\be \label{B6}
q(x,y,t,T) \ = \ \max_{0<v(\cdot)\le 1}\left[\int_t^T ds \ g\left(\frac{x(s)}{v(s)}\right)e^{-(T-s)/p}v(s) \ \Bigg| \ x(t)=x\right] \ .
\ee
The reachable set for the control problem (\ref{B6}) is the set of $(x,t)$ which satisfy
\be \label{C6}
e^{(T-t)/p}y \ < \ x \ < \ x_p(t) \ , \quad t<T \ .
\ee
where $x_p(\cdot)$ is the solution to (\ref{A6}) when $v(\cdot)\equiv 1$.  Letting
\be \label{D6}
q(y,v(\cdot),t,T) \ = \ \int_t^T ds \ g\left(\frac{x(s)}{v(s)}\right)e^{-(T-s)/p}v(s) \ ,
\ee
we have that the gradient $dq$ of $q$ with respect to $v(\cdot)$ is given by
\begin{multline} \label{E6}
e^{(T-\tau)/p}dq(y,v(\cdot),t,T;\tau) \ = \ -\frac{x(\tau)}{v(\tau)}g'\left(\frac{x(\tau)}{v(\tau)}\right)
+g\left(\frac{x(\tau)}{v(\tau)}\right) \\
+ \int_t^\tau ds\  g'\left(\frac{x(s)}{v(s)}\right)e^{-(\tau-s)/p}dx(s)(v(\cdot);\tau) \ ,
\end{multline}
where $dx(s)(v(\cdot);\tau), \ s<\tau<T$, is the gradient of $x(s)$ with respect to $v(\cdot)$.  We obtain a formula for $dx(s)(v(\cdot);\cdot)$ by observing that for any function $\phi:(-\infty,T)$ the inner product
\be \label{F6}
u_\phi(s) \ = \ [dx(s)(v(\cdot)),\phi] \ = \ \int_{-\infty}^T dx(s)(v(\cdot);\tau)\phi(\tau) \ d\tau \ ,
\ee
is a solution to the terminal value problem,
\begin{multline} \label{G6}
\frac{du_\phi(s)}{ds} \ = \ -\frac{u_\phi(s)}{p}-\phi(s)h\left(\frac{x(s)}{v(s)}\right) \\
-v(s)h'\left(\frac{x(s)}{v(s)}\right)\frac{u_\phi(s)v(s)-x(s)\phi(s)}{v(s)^2} \ ,  \ \ s<T, \quad u_\phi(T)=0 \ .
\end{multline}
Integrating (\ref{G6}) we obtain the formula
\begin{multline} \label{H6}
u_\phi(s) \ = \ \int_s^T d\tau \exp\left[(\tau-s)/p+\int_s^\tau h'(x(s')/v(s')) \ ds'\right] \times \\
\left\{h\left(\frac{x(\tau)}{v(\tau)}\right)-\frac{x(\tau)}{v(\tau)}h'\left(\frac{x(\tau)}{v(\tau)}\right)\right\}\phi(\tau) \ .
\end{multline}
Evidently $dx(s)(v(\cdot);\tau)$ is the coefficient of $\phi(\tau)$ in the integral on the RHS of (\ref{H6}). We conclude then from (\ref{E6}), (\ref{H6}) that
\begin{multline} \label{I6}
e^{(T-\tau)/p}dq(y,v(\cdot),t,T;\tau) \ = \ -\frac{x(\tau)}{v(\tau)}g'\left(\frac{x(\tau)}{v(\tau)}\right)
+g\left(\frac{x(\tau)}{v(\tau)}\right) \\
+ \int_t^\tau ds\  g'\left(\frac{x(s)}{v(s)}\right)\exp\left[\int_s^\tau h'(x(s')/v(s')) \ ds'\right] 
\left\{h\left(\frac{x(\tau)}{v(\tau)}\right)-\frac{x(\tau)}{v(\tau)}h'\left(\frac{x(\tau)}{v(\tau)}\right)\right\} \ .
\end{multline}
The first two terms on the RHS of (\ref{I6}) are positive but the third term is negative.  Setting $v(\cdot)\equiv 1$ in (\ref{I6}) we have that
\begin{multline} \label{J6}
e^{(T-\tau)/p}dq(y,1(\cdot),t,T;\tau) \ = \ -x_p(\tau)g'\left(x_p(\tau)\right)
+g\left(x_p(\tau)\right) \\
+ \int_t^\tau ds\  g'\left(x_p(s)\right)\exp\left[\int_s^\tau h'(x_p(s')) \ ds'\right] 
\left\{h\left(x_p(\tau)\right)-x_p(\tau)h'\left(x_p(\tau)\right)\right\} \ .
\end{multline}
We have now on integration by parts and using the ODE (\ref{A6}) with $v(\cdot)\equiv 1$ which $x_p(\cdot)$ satisfies that
\begin{multline} \label{K6}
 \int_t^\tau ds\  g'\left(x_p(s)\right)\exp\left[\int_s^\tau h'(x_p(s')) \ ds'\right] \ = \\
  \frac{p}{x_p(t)+ph(x_p(t))}\exp\left[\int_t^\tau h'(x_p(s')) \ ds'\right]g(x_p(t)) - 
   \frac{p}{x_p(\tau)+ph(x_p(\tau))}g(x_p(\tau))    \\
+  \int_t^\tau  \frac{ds}{x_p(s)+ph(x_p(s))} \  g\left(x_p(s)\right)\exp\left[\int_s^\tau h'(x_p(s')) \ ds'\right]  \ .
\end{multline}
It follows that $dq(y,1(\cdot),t,T;\cdot)$ is non-negative provided
\be \label{L6}
-xg'(x)+g(x)-\frac{g(x)}{x/p+h(x)}\left[h(x)-xh'(x)\right] \ \ge \ 0 \quad {\rm for \ } x>0 \ .
\ee
Evidently (\ref{L6}) holds for all $p>0$ provided
\be \label{M6}
g'(x)h(x)-g(x)h'(x) \ \le \ \ 0 \quad {\rm for \ } x>0 \ .
\ee

The Hamilton-Jacobi (HJ) equation associated with (\ref{B6}) is given by
\begin{multline} \label{N6}
\frac{\pa q(x,y,t,T)}{\pa t}-\frac{x}{p} \frac{\pa q(x,y,t,T)}{\pa x} \\
+\sup_{0<v<1}\left[ -vh\left(\frac{x}{v}\right)\frac{\pa q(x,y,t,T)}{\pa x}   +e^{-(T-t)/p}g\left(\frac{x}{v}\right)v\right] \ = \ 0 \ .
\end{multline} 
We shall obtain the solution to the variational problem (\ref{B6}) by producing a $C^1$ solution to the HJ equation (\ref{N6}).  The solution is obtained  by using bang-bang control settings. 
Thus for $(x,t)$ in the reachable set (\ref{C6}) we define $\tau_{x,t}$ as the time at which the trajectory for (\ref{A6}) with $v(\cdot)\equiv 0$ and $x(t)=x$ reaches the curve $x_p(\cdot)$.  Hence $\tau_{x,t}$  satisfies the identity
\be \label{O6}
e^{-(\tau_{x,t}-t)/p}x \ = \ x_p(\tau_{x,t}) \ .
\ee
Then we set
\be \label{P6}
q(x,y,t,T) \ = \ \int_{\tau_{x,t}}^T ds  \ g(x_p(s))e^{-(T-s)/p} \ ds \ .
 \ee
 \begin{proposition}
 Assume $g(\cdot), \ h(\cdot)$ are $C^1$ non-negative decreasing functions such that (\ref{M6}) holds.  For any $y>0, \ t<T,$ let $(x,t)$ satisfy (\ref{C6}) and $\tau_{x,t}$ be defined by (\ref{O6}). Then $t<\tau_{x,t}<T$ and the function $q(x,y,t,T)$ of (\ref{B6}) is given by the formula (\ref{P6}). 
 \end{proposition}
 \begin{proof}
 Differentiating (\ref{O6}) with respect to $x$ we have that
 \be \label{Q6}
 \frac{\pa \tau_{x,t}}{\pa x} \ = \ -\frac{x_p(\tau_{x,t})}{xh(x_p(\tau_{x,t}))} \ , \quad
  \frac{\pa \tau_{x,t}}{\pa t} \ = \ -\frac{x_p(\tau_{x,t})}{ph(x_p(\tau_{x,t}))} \ .
 \ee
We have from (\ref{P6}) that
\be \label{R6}
\frac{\pa q(x,y,t,T)}{\pa x} \ = \ -\exp\left[-(T-\tau_{x,t})/p\right]g(x_p(\tau_{x,t})) \frac{\pa \tau_{x,t}}{\pa x} \ , 
\ee
and similarly that
\be \label{S6}
\frac{\pa q(x,y,t,T)}{\pa t} \ = \ -\exp\left[-(T-\tau_{x,t})/p\right]g(x_p(\tau_{x,t})) \frac{\pa \tau_{x,t}}{\pa t} \ .
\ee
It follows from (\ref{Q6})-(\ref{S6}) that $q$ is a solution to the PDE
\be \label{T6}
\frac{\pa q(x,y,t,T)}{\pa t}-\frac{x}{p} \frac{\pa q(x,y,t,T)}{\pa x}  \ = \ 0 \ .
\ee
Since $q$  is a solution to (\ref{T6}) we need only show that
\be \label{U6}
e^{-(T-t)/p}g\left(\frac{x}{v}\right) \ \le \ h\left(\frac{x}{v}\right)\frac{\pa q(x,y,t,T)}{\pa x}  \quad {\rm for \ } 0<v<1, \ x>0  ,
\ee
in order to prove that $q$ is a solution to the HJ equation (\ref{N6}). Observe now that since (\ref{M6}) holds  the function $x\ra g(x)/h(x), \ x>0,$ is decreasing. Hence  to prove (\ref{U6}) it is sufficient to show that 
\be \label{V6}
u(x,t) \ = \ \frac{\pa q(x,y,t,T)}{\pa x} -e^{-(T-t)/p}\frac{g(x)}{h(x)} \ \ge \ 0 \quad x>0 \ .
\ee
Observe next from (\ref{Q6}), (\ref{R6}) that $\pa q(x,y,t,T)/\pa x$ approaches $e^{-(T-t)/p}g(x_p(t))/h(x_p(t))$ as $x\ra x_p(t)$. It follows from (\ref{V6}) that $u(x,t)=0$ if $x=x_p(t)$.  We have now on differentiating (\ref{T6}) with respect to $x$ that
\be \label{W6}
 \frac{\pa u(x,t)}{\pa t}-\frac{x}{p} \frac{\pa u(x,t)}{\pa x}-\frac{1}{p}u(x,t) \ = \frac{x}{p}e^{-(T-t)/p}\frac{d}{dx}\frac{g(x)}{h(x)} \ \le \ 0 \ .
\ee
It follows by the method of characteristics that $u(x,t)\ge 0$ for all $(x,t)$ satisfying (\ref{C6}). Hence the function $q$ of (\ref{P6}) is a solution to the HJ equation (\ref{N6}). Since $q$ is a $C^1$ solution to the HJ equation for $(x,t)$ in the reachable set we see that 
the solution to the variational problem (\ref{B6}) is given by (\ref{P6}). 
\end{proof}
\begin{rem}
Since the function (\ref{P6}) satisfies $\pa q(x,y,t,T)/\pa x\ge 0$, it follows that the solution of the variational problem  $\max_{0<v(\cdot)<1} q(y,v(\cdot),t,T)$, with $q$ as in (\ref{D6}), is given by $v(\cdot)\equiv 1$. 
\end{rem}
Next we consider the variational problem analogous to (\ref{B6}) given by
\be \label{X6}
q(x,y,t,T) \ = \ \min_{1\le v(\cdot)<\infty}\left[\int_t^T ds \ g\left(\frac{x(s)}{v(s)}\right)e^{-(T-s)/p}v(s) \ \Bigg| \ x(t)=x\right] \ .
\ee
The reachable set for the control problem (\ref{X6}) is the set of $(x,t)$ which satisfy
\be \label{Y6}
x_p(t)<x<\infty \ , \quad t<T \ .
\ee
The HJ equation corresponding to (\ref{X6}) is given by
\begin{multline} \label{Z6}
\frac{\pa q(x,y,t,T)}{\pa t}-\frac{x}{p} \frac{\pa q(x,y,t,T)}{\pa x} \\
+\inf_{1<v<\infty}\left[ -vh\left(\frac{x}{v}\right)\frac{\pa q(x,y,t,T)}{\pa x}   +e^{-(T-t)/p}g\left(\frac{x}{v}\right)v\right] \ = \ 0 \ .
\end{multline} 
The minimization problem (\ref{X6}) is trivial in the case $g(\cdot)\equiv C_0h(\cdot)$ for some constant $C_0$,  since then the function $q(y,v(\cdot),t,T)$ of (\ref{D6}) is independent of $v(\cdot)$. In fact if $g(\cdot)\equiv h(\cdot)$ we have from (\ref{D6}) that
\be \label{AA6}
q(y,v(\cdot),t,T) \ = \ -\int_t^T ds \ e^{-(T-s)/p}\left[\frac{dx(s)}{ds}+\frac{x(s)}{p}\right] \ .
\ee
Evaluating the integral on the RHS of (\ref{AA6}), we conclude that $q(x,y,t,T)=e^{-(T-t)/p}x-y$ for $(x,t)$ in the reachable set (\ref{Y6}). Note that since $\pa q(x,y,t,T)/\pa x=e^{-(T-t)/p}$, the infimum in (\ref{Z6}) is now simply zero. 

Next we consider the optimization problem 
\be \label{AB6}
q(x,y,t,T) \ = \ \max_{0<v(\cdot)\le 1}\left[\int_t^T ds \ g\left(\frac{x(s)}{v(s)}\right) \ \Bigg| \ x(t)=x\right] \ ,
\ee
where $x(\cdot)$ has the dynamics (\ref{A6}) and $g(\cdot)$ is assumed positive decreasing with $\lim_{x\ra\infty}g(x)=0$. The reachable set for the control problem (\ref{AB6}) is given by (\ref{C6}). 
Letting
\be \label{AC6}
q(y,v(\cdot),t,T) \ = \ \int_t^T g\left(\frac{x(s)}{v(s)}\right) \ ds \ ,
\ee
we have that $dq$ is given by the formula
\begin{multline} \label{AD6}
dq(y,v(\cdot),t,T;\tau) \ = \\
 -\frac{x(\tau)}{v(\tau)^2}g'\left(\frac{x(\tau)}{v(\tau)}\right) 
+ \int_t^\tau ds\  g'\left(\frac{x(s)}{v(s)}\right) \frac{dx(s)(v(\cdot);\tau)}{v(s)}  \ .
\end{multline}
Observe now that similarly to (\ref{K6}) we have
\begin{multline} \label{AE6}
\int_t^\tau ds\  g'\left(x_p(s)\right)
\exp\left[(\tau-s)/p+\int_s^\tau h'(x_p(s')) \ ds'\right]  \\
= \ \frac{g\left(x_p(t)\right)}{x_p(t)/p+h(x_p(t))}
\exp\left[(\tau-t)/p+\int_t^\tau h'(x_p(s')) \ ds'\right] \\
-\frac{g\left(x_p(\tau)\right)}{x_p(\tau)/p+h(x_p(\tau))}  \ .
\end{multline}
From (\ref{H6}), (\ref{AE6}) we then see on setting $v(\cdot)\equiv 1$ in  (\ref{AD6}) that 
\begin{multline} \label{AF6}
dq(y,1(\cdot),t,T;\tau) \ \ge \\
 -x_p(\tau)g'\left(x_p(\tau)\right) 
-\frac{g\left(x_p(\tau)\right)}{x_p(\tau)/p+h(x_p(\tau))}\left\{h\left(x_p(\tau)\right)-x_p(\tau)h'\left(x_p(\tau)\right)\right\} \   \ .
\end{multline}
It follows from (\ref{AF6}) that $dq(y,1(\cdot),t,T;\cdot)$ is positive provided
\be \label{AG6}
\left[\frac{x}{p}+h(x)\right]xg'(x)+g(x)[h(x)-xh'(x)] \ \le  \ 0 \quad {\rm for\  } x>0 \ .
\ee

The HJ equation for  the variational problem (\ref{AB6}) is given by
\begin{multline} \label{AH6}
\frac{\pa q(x,y,t,T)}{\pa t}-\frac{x}{p} \frac{\pa q(x,y,t,T)}{\pa x} \\
+\sup_{0<v<1}\left[ -vh\left(\frac{x}{v}\right)\frac{\pa q(x,y,t,T)}{\pa x}   +g\left(\frac{x}{v}\right)\right] \ = \ 0 \ .
\end{multline} 
We seek a solution to (\ref{AH6})  by using bang-bang control. Thus we consider similarly to (\ref{P6}) the function
\be \label{AI6}
q(x,y,t,T) \ = \ \int_{\tau_{x,t}}^T ds  \ g(x_p(s)) \ ds \ .
 \ee
\begin{proposition}
 Assume $g(\cdot), \ h(\cdot)$ are $C^1$ non-negative decreasing functions  such that (\ref{AG6}) holds for all $p>0$.  For any $y>0, \ t<T,$ let $(x,t)$ satisfy (\ref{C6}) and $\tau_{x,t}$ be defined by (\ref{O6}). Then $t<\tau_{x,t}<T$ and the function $q(x,y,t,T)$ of (\ref{AB6}) is given by the formula (\ref{AI6}). 
 \end{proposition}
 \begin{proof} 
 From (\ref{Q6}) we see that the function (\ref{AI6}) satisfies (\ref{T6}). 
Hence $q$  is a solution to (\ref{AH6}) provided
\be \label{AJ6}
g\left(\frac{x}{v}\right) \ \le \ vh\left(\frac{x}{v}\right)\frac{\pa q(x,y,t,T)}{\pa x}  \quad {\rm for \ } 0<v<1, \ x>0 \ .
\ee
Letting $p\ra\infty$ in (\ref{AG6}) we see that the function $x\ra xg(x)/h(x)$ is decreasing, (whence $\lim_{x\ra\infty}g(x)=0$). The inequality (\ref{AJ6}) holds therefore if
\be \label{AK6}
u(x,t) \ = \ \frac{\pa q(x,y,t,T)}{\pa x} -\frac{g(x)}{h(x)} \ \ge \ 0 \quad {\rm for \ } x>0 \ .
\ee
Observe next from (\ref{Q6}), on differentiating (\ref{AI6}) with respect to $x$, that $\pa q(x,y,t,T)/\pa x$ approaches $g(x_p(t))/h(x_p(t))$ as $x\ra x_p(t)$. It follows from (\ref{AK6}) that $u(x,t)=0$ if $x=x_p(t)$.  We see now on differentiating (\ref{T6}) with respect to $x$ that
\be \label{AL6}
 \frac{\pa u(x,t)}{\pa t}-\frac{x}{p} \frac{\pa u(x,t)}{\pa x}-\frac{1}{p}u(x,t) \ = \frac{1}{p}\frac{d}{dx}\left[\frac{xg(x)}{h(x)}\right] \ \le \ 0 \ .
\ee
It follows by the method of characteristics that $u(x,t)\ge 0$ for all $(x,t)$ satisfying (\ref{C6}). Hence the function $q$ of (\ref{AI6}) is a solution to the HJ equation (\ref{AH6}). Since $q$ is also a $C^1$ solution to (\ref{AH6}),   the solution to the variational problem (\ref{AB6}) is  given by (\ref{AI6}).
\end{proof}
\begin{rem}
Since the function (\ref{AI6}) satisfies $\pa q(x,y,t,T)/\pa x\ge 0$, it follows that the solution of the variational problem  $\max_{0<v(\cdot)<1} q(y,v(\cdot),t,T)$, with $q$ as in (\ref{AC6}), is given by $v(\cdot)\equiv 1$. 
\end{rem}

Next  we  consider the optimization problem
\be \label{AM6}
q(x,y,t,T) \ = \ \min_{1<v(\cdot)<\infty}\left[\int_t^T ds \ g\left(\frac{x(s)}{v(s)}\right) \ \Bigg| \ x(t)=x\right] \ ,
\ee
where $x(\cdot)$ has the dynamics (\ref{A6}) and $(x,t)$ belongs to the reachable set (\ref{Y6}) of the control system.  The HJ equation for (\ref{AM6}) is given by
\begin{multline} \label{AN6}
\frac{\pa q(x,y,t,T)}{\pa t}-\frac{x}{p} \frac{\pa q(x,y,t,T)}{\pa x} \\
+\inf_{1<v<\infty}\left[ -vh\left(\frac{x}{v}\right)\frac{\pa q(x,y,t,T)}{\pa x}   +g\left(\frac{x}{v}\right)\right] \ = \ 0 \ .
\end{multline} 
Letting $G(x,\xi,v),  \ \tilde{G}(x,\xi)$ be defined by
\be \label{AO6}
G(x,\xi,v) \ = \ -vh\left(\frac{x}{v}\right)\xi   +g\left(\frac{x}{v}\right) \ , \quad \tilde{G}(x,\xi)=\inf_{1<v<\infty} G(x,\xi,v) \ ,
\ee
we see that the function $\xi\ra \tilde{G}(x,\xi)$ is concave. We can make a change of variable $v\ra w=vh(x/v)$, so
\be \label{AP6}
\frac{dw}{dv} \ = \ h\left(\frac{x}{v}\right)-\frac{x}{v}h'\left(\frac{x}{v}\right) \ > \ 0 \ .
\ee 
Hence the function $w\ra G(x,\xi,v(w))$ is convex provided
\be \label{AQ6}
\frac{d}{dw} g\left(\frac{x}{v}\right) \ = \ \frac{dv}{dw} \left[-\frac{x}{v^2}g'\left(\frac{x}{v}\right)\right]
\ee 
increases as a function of $w$. From (\ref{AP6}) we see that this is equivalent to the function
\be \label{AR6}
x\ra -\frac{x^2g'(x)}{h(x)-xh'(x)}  \ = \ m(x)\quad {\rm decreases  } .
\ee
Observe that if (\ref{AG6}) holds for $p=\infty$ then
\be \label{AS6}
-\frac{x^2g'(x)}{h(x)-xh'(x)} \ \ge \  \frac{xg(x)}{h(x)} \ , \quad x>0 \ .
\ee
By (\ref{AG6}) with $p=\infty$ the function on the RHS of (\ref{AS6}) is decreasing, so  (\ref{AR6}) is an extra condition that the function on the LHS of (\ref{AS6}) also decreases. 

We assume now that (\ref{AR6}) holds. Then the minimum of $G(x,\xi,v)$ on the interval $1<v<\infty$ is attained at  $v=1$ if $\xi\le -xg'(x)/[h(x)-xh'(x)]$. If $\xi>-xg'(x)/[h(x)-xh'(x)]$  the minimizer of $\min_{v\ge 1} G(x,\xi,v)$ is the solution to the equation
\be \label{AT6}
-\left[    h\left(\frac{x}{v}\right)-\frac{x}{v}h'\left(\frac{x}{v}\right)       \right]\xi-\frac{x}{v^2}g'\left(\frac{x}{v}\right) \ =  \ 0 \ ,  \ \  {\rm whence}  \ m\left(\frac{x}{v}\right) \ = \ x\xi \ = \zeta \ .
\ee
A solution to (\ref{AT6}) exists for all $\zeta>m(x)$ provided $\lim_{z\ra 0}m(z)=\infty$.  From (\ref{AT6}) it follows that the minimizing $v=v_{\rm min}(x,\xi)=x/m^{-1}(\zeta)$, where $m^{-1}(\cdot)$ is the inverse function for $m$. The corresponding HJ equation has therefore the form
\be \label{AU6}
\frac{\pa q(x,y,t,T)}{\pa t}+ H\left(x\frac{\pa q(x,y,t,T)}{\pa x}\right) \ = \ 0 \ ,
\ee
where
\be \label{AV6}
H(\zeta) \ = \ -\frac{\zeta}{p}-\frac{\zeta h(m^{-1}(\zeta))}{m^{-1}(\zeta)}+g\left(m^{-1}(\zeta)\right) \ .
\ee
Note that $\zeta=x\pa q(x,y,t,T)/\pa x$ is constant along characteristics for the HJ equation (\ref{AU6}), whence  it follows from (\ref{AT6}) that $x(\cdot)/v(\cdot)$ is also constant along characteristics.

The considerations of the previous paragraph lead us to propose a solution to (\ref{AN6}). For $s<t<T$ let $x_p(s,t)$ be the solution to the terminal value problem
\be \label{AW6}
\frac{dx_p(s,t)}{ds} \ = \ -\left[\frac{1}{p}+\frac{h(x_p(t))}{x_p(t)}\right]x_p(s,t) \ ,  \ s<t<T, \quad x_p(t,t) \ = \ x_p(t) \   .
\ee
Setting $x(s)=x_p(s,t)-x_p(s)$, we see from (\ref{A6}) with $v(\cdot)\equiv 1$ and (\ref{AW6}) that
\begin{multline} \label{AX6}
\frac{dx(s)}{ds} \ = \ -\left[\frac{1}{p}+\frac{h(x_p(t))}{x_p(t)}\right]x(s) \\
-x_p(s)\left[\frac{h(x_p(t))}{x_p(t)}-\frac{h(x_p(s))}{x_p(s)}\right] \ ,  \ s<t<T, \quad x(t) \ = \ 0 \   .
\end{multline}
Observe that the function $x\ra h(x)/x$ is decreasing and also  the function $s\ra x_p(s)$.  It follows then from (\ref{AX6}) that $x(s)>0$ for $s<t$.  Hence the trajectory $x_p(s,t), \ s<t,$ lies in the reachable set (\ref{Y6}) for the variational problem (\ref{AM6}).  We can show similarly that the trajectories $x_p(\cdot,t), \ t<T$, do not intersect. Thus for $t_1<t_2<T$ let $x(s)=x_p(s,t_2)-x_p(s,t_1), \ s<t_1$. We have already seen that $x(t_1)>0$, and from (\ref{AW6}) we also have that
\be \label{AY6}
\frac{dx(s)}{ds} \ = \ -\left[\frac{1}{p}+\frac{h(x_p(t_1))}{x_p(t_1)}\right]x(s)-x_p(s,t_2)\left[\frac{h(x_p(t_2))}{x_p(t_2)}-\frac{h(x_p(t_1))}{x_p(t_1)}\right] \ ,  \ s<t_1 \ .
\ee
Since $x_p(t_1)>x_p(t_2)$ we conclude from (\ref{AY6}) that $x_p(s,t_2)>x_p(s,t_1),\ s<t_1$.  

Since the trajectories $x_p(\cdot,t), \ t<T,$ do not entirely cover the reachable set we complement them with a set of trajectories with terminal point $y$ at time $T$. Thus for $s<T,  0<\la<y$ we define $y_p(s,\la)$ as the solution to
\be \label{AZ6}
\frac{dy_p(s,\la)}{ds} \ = \ -\left[\frac{1}{p}+\frac{h(\la)}{\la}\right]y_p(s,\la) \ ,  \ s<T, \quad y_p(T,\la) \ = \ y \   .
\ee
If $t<T$ and $x_p(t)<x<x_p(t,T)$ then there exists unique $\tau=\tau_{x,t}$ such that $t<\tau<T$ and $x_p(t,\tau)=x$.  If $x>x_p(t,T)$ then there exists unique $\la=\la_{x,t}$ such that $0<\la<y$ and $y_p(t,\la)=x$.  We define now a function $q(x,y,t,T)$ for $t<T$ and $(x,t)$ satisfying  (\ref{Y6}) by
\begin{multline} \label{BA6}
q(x,y,t,T) \ = \ (\tau_{x,t}-t)g(x_p(\tau_{x,t}))+\int_{\tau_{x,t}}^T g(x_p(s)) \ ds \quad {\rm if \ } x_p(t)<x<x_p(t,T) \ , \\
q(x,y,t,T) \ = \ (T-t)g(\la_{x,t}) \quad {\rm if \ }  x>x_p(t,T) \ .
\end{multline}
\begin{proposition}
Assume $g(\cdot), \ h(\cdot)$ are $C^1$ non-negative decreasing, and  also that (\ref{AR6}) holds. For any $y>0, \ t<T,$ let $(x,t)$ satisfy (\ref{Y6}). Then  the function $q(x,y,t,T)$ of (\ref{AM6}) is given by the formula (\ref{BA6}).
\end{proposition}
\begin{proof}
We first consider the case $x>x_p(t,T)$. The partial derivatives of $\la_{x,t}$ can be computed by using the formula
\be \label{BB6}
\exp\left[\left(\frac{1}{p}+\frac{h(\la_{x,t})}{\la_{x,t}}\right)(T-t)\right]y \ = \ x \  .
\ee
Thus we have that
\begin{multline} \label{BC6}
\frac{\pa \la_{x,t}}{\pa x} \ = \ -\frac{\la_{x,t}^2}{(T-t)x[h(\la_{x,t})-\la_{x,t}h'(\la_{x,t})]} \ , \\
 \frac{\pa \la_{x,t}}{\pa t} \ = \ -\frac{\la_{x,t}^2}{(T-t)[h(\la_{x,t})-\la_{x,t}h'(\la_{x,t})]}\left(\frac{1}{p}+\frac{h(\la_{x,t})}{\la_{x,t}}\right) \ .
\end{multline}
It follows from (\ref{BA6}), (\ref{BC6}) that
\begin{multline} \label{BD6}
x\frac{\pa q(x,y,t,T)}{\pa x} \ = \ -\frac{g'(\la_{x,t})\la_{x,t}^2}{[h(\la_{x,t})-\la_{x,t}h'(\la_{x,t})]} \ , \\
\frac{\pa q(x,y,t,T)}{\pa t} \ = \ -g(\la_{x,t})- \frac{g'(\la_{x,t})\la_{x,t}^2}{[h(\la_{x,t})-\la_{x,t}h'(\la_{x,t})]}\left(\frac{1}{p}+\frac{h(\la_{x,t})}{\la_{x,t}}\right) \ . 
\end{multline}
Hence  $q$ is a solution to the PDE
\begin{multline} \label{BE6}
\frac{\pa q(x,y,t,T)}{\pa t}-\frac{x}{p} \frac{\pa q(x,y,t,T)}{\pa x} -v(x,t)h\left(\frac{x}{v(x,t)}\right) \frac{\pa q(x,y,t,T)}{\pa x} +g\left(\frac{x}{v(x,t)}\right) \ = \ 0 \ , \\
{\rm where \ \ } \frac{x}{v(x,t)} \ = \ \la_{x,t} \ .
\end{multline}
Note that $v(x,t)>1$ since $\la_{x,t}<y<x$.  We also have that
\be \label{BF6}
\frac{\pa}{\pa v}\left[-vh\left(\frac{x}{v}\right)\frac{\pa q(x,y,t,T)}{\pa x}+g\left(\frac{x}{v}\right) \right] \ = \ 0 \quad {\rm at \ } v=v(x,t) \ .
\ee
Hence, in view of (\ref{AR6}), we conclude that $q(x,y,t,T)$ satisfies the HJ equation (\ref{AN6}) in the region $\{(x,t): \ t<T, \ x>x_p(t,T)\}$. 

Next we consider the region $\{(x,t): \ t<T, \ x_p(t)<x<x_p(t,T)\}$.  In that case we have
\be \label{BG6}
\exp\left[\left\{\frac{1}{p}+\frac{h(x_p(\tau_{x,t}))}{x_p(\tau_{x,t})}\right\}(\tau_{x,t}-t)\right]x_p(\tau_{x,t}) \ = \ x \ .
\ee
Differentiating (\ref{BG6}) with respect to $x$ gives
\be \label{BH6}
\frac{\pa \tau_{x,t}}{\pa x} \ = \ \frac{x_p(\tau_{x,t})^2}{[x_p(\tau_{x,t})/p+h(x_p(\tau_{x,t}))](\tau_{x,t}-t)x[h(x_p(\tau_{x,t}))-x_p(\tau_{x,t})h'(x_p(\tau_{x,t}))]} \ .
\ee
Similarly we have that
\be \label{BI6}
\frac{\pa \tau_{x,t}}{\pa t} \ = \ \frac{x_p(\tau_{x,t})}{(\tau_{x,t}-t)[h(x_p(\tau_{x,t}))-x_p(\tau_{x,t})h'(x_p(\tau_{x,t}))]} \ .
\ee
From (\ref{BA6}), (\ref{BH6}) we have that
\begin{multline} \label{BJ6}
x\frac{\pa q(x,y,t,T)}{\pa x} \ = \ -(\tau_{x,t}-t)g'(x_p(\tau_{x,t}))\left[\frac{x_p(\tau_{x,t})}{p}+h(x_p(\tau_{x,t}))\right]x\frac{\pa \tau_{x,t}}{\pa x} \\
= \ -\frac{x_p(\tau_{x,t})^2g'(x_p(\tau_{x,t}))}{ [h(x_p(\tau_{x,t}))-x_p(\tau_{x,t})h'(x_p(\tau_{x,t}))]} \ ,
\end{multline}
and also from (\ref{BI6}) that 
\begin{multline} \label{BK6}
\frac{\pa q(x,y,t,T)}{\pa t} \ = \ -g(x_p(\tau_{x,t}))-(\tau_{x,t}-t)g'(x_p(\tau_{x,t}))\left[\frac{x_p(\tau_{x,t})}{p}+h(x_p(\tau_{x,t}))\right]\frac{\pa \tau_{x,t}}{\pa t}  \\
= \ -g(x_p(\tau_{x,t}))-\frac{x_p(\tau_{x,t})g'(x_p(\tau_{x,t}))}{ [h(x_p(\tau_{x,t}))-x_p(\tau_{x,t})h'(x_p(\tau_{x,t}))]}\left[\frac{x_p(\tau_{x,t})}{p}+h(x_p(\tau_{x,t}))\right] \ .
\end{multline}
It follows from (\ref{BJ6}), (\ref{BK6}) that  $q$ is a solution to the PDE
\begin{multline} \label{BL6}
\frac{\pa q(x,y,t,T)}{\pa t}-\frac{x}{p} \frac{\pa q(x,y,t,T)}{\pa x} -v(x,t)h\left(\frac{x}{v(x,t)}\right)\frac{\pa q(x,y,t,T)}{\pa x} +g\left(\frac{x}{v(x,t)}\right) \ = \ 0 \ , \\
{\rm where \ \ } \frac{x}{v(x,t)} \ = \ x_p(\tau_{x,t}) \ .
\end{multline}
Note that since $x>x_p(\tau_{x,t})$ we have $v(x,t)>1$ in (\ref{BL6}).  Furthermore, the identity (\ref{BF6}) also holds. We therefore conclude that $q$ is a solution  to the HJ equation (\ref{AN6}).  Since $q$ is a $C^1$ solution to the HJ equation for $(x,t)$ in the reachable set (\ref{Y6}) it follows that
the solution to the variational problem (\ref{AM6}) is given by (\ref{BA6}). 
\end{proof}
\begin{rem}
Since the function (\ref{BA6}) satisfies $\pa q(x,y,t,T)/\pa x\ge 0$, it follows  that the solution of the variational problem  $\min_{1<v(\cdot)<\infty} q(y,v(\cdot),t,T)$ is given by $v(\cdot)\equiv 1$. 
\end{rem}
We wish to relate the condition (\ref{AG6}) with $p=\infty$, which insures a local extremum at $v(\cdot)\equiv1$, to the condition (\ref{AR6}).   We have  already observed that  (\ref{AG6}) with $p=\infty$ is equivalent to the function on the RHS of (\ref{AS6}) decreasing. Our goal is to show that (\ref{AR6}) is a convexity condition on this function.  To see this we set $z(x)=x/h(x)$, whence (\ref{AG6}) with $p=\infty$ implies that
\be \label{BM6}
\frac{d}{dz}\left[zg(x)\right] \ = \ g(x)+z\frac{dx}{dz}g'(x) \ \le  \ 0 \ .
\ee
We also have that the function $m(\cdot)$ of (\ref{AR6}) is given by
\be \label{BN6}
m(x) \ = \  g'(x)\Big/\frac{d}{dx}\left[\frac{h(x)}{x}\right] \ =  \ -z^2 \frac{dx}{dz}g'(x) \ .
\ee
Hence the condition $m(\cdot)$ decreasing is equivalent to
\be \label{BO6}
\frac{d}{dz}\left[z^2 \frac{dx}{dz}g'(x)\right] \ \ge \ 0 \ .
\ee
Observe from (\ref{BM6}) that (\ref{BO6}) is the same as 
\be \label{BP6}
\frac{d^2}{dz^2}\left[zg(x)\right] \ = 
 \left[2\frac{dx}{dz}+z\frac{d^2x}{dz^2}\right]g'(x)+ z\left(\frac{dx}{dz}\right)^2g''(x) \ \ge \ 0 \ .
\ee
Hence (\ref{AG6}) is equivalent to the function $z\ra zg(x(z))$ decreasing, while (\ref{AR6}) is equivalent to convexity of the function $z\ra zg(x(z))$. Note that convexity of a function implies that it is 
decreasing, provided the function is bounded at infinity.

\vspace{.1in}

\section{Global Asymptotic Stability}
In this section we prove a global asymptotic stability result for solutions of (\ref{A3}) with $\rho(t)=\rho(\xi(\cdot,t),\eta(t))$ given by (\ref{T2}), which extends the local asymptotic stability result Theorem 4.1. As in \cite{cd}, the key to proving this is to establish global asymptotic stability for the DDE (\ref{O5}), (\ref{P5}) using the monotonicity properties of the function $f(t,v_t(\cdot))$ implied by Proposition 5.1. Adapting the argument of Proposition 8.1 of \cite{cd}, we obtain the following:
\begin{proposition}
Let $\xi(\cdot,t), \ t>0,$ be the solution of  (\ref{A3}) with $\rho(t)=\rho(\xi(\cdot,t),\eta(t))$, which is considered in Proposition 3.1. In addition to the assumptions (\ref{R4}), (\ref{T4}) on $h(\cdot)$ required for Lemma 3.2, assume $h(\cdot)$ satisfies (\ref{I4}) for all $p>0$ and also that (\ref{S5}) holds. Let  $\eta(\cdot)$ satisfy (\ref{Z3}) and the inequality $\eta(t)\le Ce^{-\del t}, \ t>0,$ for some constants $C,\del>0$. Then if the function $x\ra\beta(x,0)$  is H\"{o}lder continuous at $x=1$, there exists $I_\infty>0$ such that the function $I(\cdot)$ defined by (\ref{A5}) satisfies
$\lim_{t\ra\infty}I(t)=I_\infty$. 
\end{proposition}
\begin{proof}
It follows from (\ref{AQ2}) and (\ref{V*3}) of Proposition 3.1 that
\be \label{B7}
\inf_{t>0} \left\{pI_p(\xi(\cdot,t))+\left[dI_p(\xi(\cdot,t)),  \ A\xi(\cdot,t)\right] \right\} \  > \ 0 \ .
\ee
We note that the function $I(\cdot)$ defined by (\ref{A5}) has the property
\be \label{C7}
c_0I(0) \ \le \ I(t) \ \le \ C_0 I(0), \ \ t>0, \quad {\rm for \ some \ constants \ } \ C_0,c_0>0 \ .
\ee
This is a consequence of the H\"{o}lder assumption on $\beta(\cdot,0)$ upon using Lemma 2.1, Lemma 2.2, the identity (\ref{AF3}), and our assumptions on the function $\eta(\cdot)$. 
It further follows from  the inequality in (\ref{AA3}), (\ref{D5}) and (\ref{C7}) that the function $G$ of (\ref{N5}) satisfies an inequality $|G(t,y,v_t(\cdot))|\le Ce^{-t/p}, \ y\ge \ve_0, \ t\ge 0,$ for some constant $C$.  We conclude then from (\ref{B7}) that the function g of (\ref{P5}) satisfies an inequality $|g(t,v_t(\cdot))|\le Ce^{-t/p}$ for some constant $C$.  We observe also from (\ref{BF2}) of Lemma 2.4 that the function $\ga(\cdot)$ on the RHS of (\ref{O5}) satisfies $|\ga(t)|\le C_1\e^{-\del_1 t}, \ t>0,$ for some constants $C_1,\del_1>0$. 

To prove convergence of $I(t)$ as $t\ra\infty$, we first assume that for any $\ve,\tau>0,\  \tau'>\tau$ and $\ve<1/2$, there exists $T_{\ve,\tau,\tau'}>\tau'$ such that $|I(t)^{1/p}/I(s)^{1/p}-1|<\ve$ for $t,s\in[T_{\ve,\tau,\tau'}-\tau, T_{\ve,\tau,\tau'}]$.  For $t>T_{\ve,\tau,\tau'}$ we set $I_{\rm max}(t)=\sup_{T_{\ve,\tau,\tau'}<s<t}I(s)$ and consider $T>T_{\ve,\tau,\tau'}$ such that $I(T)=I_{\rm max}(T)$.  Integrating (\ref{O5}) and using the exponential decay of the RHS of (\ref{O5}) we have that
\begin{multline} \label{D7}
\frac{1}{p}\log\left[\frac{I(T)}{I(T_{\ve,\tau,\tau'})}\right] + \\
\int_{(T_{\ve,\tau,\tau'},T)-\{T_{\ve,\tau,\tau'}<t<T:I_{\rm max}(t)>I(t)\}} f(t,v_t(\cdot)) \ dt \ \le C_2e^{-\del_2 T_{\ve,\tau,\tau'}} \ ,
\end{multline}
for some constants $C_2,\del_2>0$. 
We can estimate the second term on the LHS of (\ref{D7}) by using Proposition 5.1.  First we write the function $F$ of (\ref{G5}) as $F=F_1+F_2$, where 
\begin{multline} \label{E7}
F_1(t,y,v_t(\cdot)) \ = \ \frac{1}{p}\int_{T_{\ve,\tau,\tau'}-\tau}^t  h\left(   \frac{z(s)}{v_t(s)}  \right) e^{-(t-s)/p}v_t(s) \ ds \\
- \left(h(y)+\frac{y}{p}\right) \left\{\exp\left[\int_{T_{\ve,\tau,\tau'}-\tau}^t h'\left(    \frac{z(s)}{v_t(s)}  \right) \ ds\right]-1\right\} \ .
\end{multline}
Note that from (\ref{D5}) the first term on the RHS of (\ref{E7}) has the simplification
\begin{multline} \label{E*7}
 \frac{1}{p}\int_{T_{\ve,\tau,\tau'}-\tau}^t  h\left(   \frac{z(s)}{v_t(s)}  \right) e^{-(t-s)/p}v_t(s) \ ds 
  \\ = \ \frac{1}{p}\left[\exp\left\{-(t+\tau-T_{\ve,\tau,\tau'})/p\right\}z(T_{\ve,\tau,\tau'}-\tau)-y\right] \ .
\end{multline}
We see from (\ref{D5}), (\ref{C7}) and the assumptions on $h(\cdot)$ that \\
$|F_2(t,y,v_t(\cdot))|\le C_3e^{-\tau/p}e^{-(t-T_{\ve,\tau,\tau'})/p}, \ y\ge\ve_0,  \ t\ge T_{\ve,\tau,\tau'}, $ for some constant $C_3$. Next we define the function $\tilde{v}_t(\cdot)$ as
\begin {eqnarray} \label{F7}
\tilde{v}_t(s) \ &=& \ v_t(T_{\ve,\tau,\tau'}), \quad T_{\ve,\tau,\tau'}-\tau<s<T_{\ve,\tau,\tau'} \ , \\
\tilde{v}_t(s) \ &=& \ v_t(s), \quad T_{\ve,\tau,\tau'}<s<t \ .  \nonumber 
\end{eqnarray}
We define also the function $\tilde{z}(s), \  \quad T_{\ve,\tau,\tau'}-\tau<s<t,$ as the solution to (\ref{D5}) with $\tilde{v}_t(\cdot)$ in place of $v_t(\cdot)$. Evidently $z(s)=\tilde{z}(s)$ for $s\in[T_{\ve,\tau,\tau'},t]$.  For $s\in [T_{\ve,\tau,\tau'}-\tau, T_{\ve,\tau,\tau'}]$ we note that $\tilde{z}(s)/\tilde{v}_t(s)=\tilde{z}(s)/v_t(T_{\ve,\tau,\tau'})\ge z(T_{\ve,\tau,\tau'})/v_t(T_{\ve,\tau,\tau'})\ge\ve_0$. Hence the equation (\ref{D5}) is well-defined for $\tilde{v}_t(\cdot), \ \tilde{z}(\cdot)$.  We define  the function $\tilde{F}_1$ by (\ref{E7}) with $\tilde{v}_t(\cdot), \ \tilde{z}(\cdot)$ in place of $v_t(\cdot), \ z(\cdot)$. We see from (\ref{E7}), (\ref{E*7}) that
 \begin{multline} \label{G7}
\left| F_1(t,y,v_t(\cdot))- \tilde{F}_1(t,y,\tilde{v}_t(\cdot))\right| \ \le \\
\frac{1}{p}\exp\left\{-(t+\tau-T_{\ve,\tau,\tau'})/p\right\}\left|z(T_{\ve,\tau,\tau'}-\tau)
-\tilde{z}(T_{\ve,\tau,\tau'}-\tau)\right|
\\
- \left(h(y)+\frac{y}{p}\right) \int_{T_{\ve,\tau,\tau'}-\tau}^{T_{\ve,\tau,\tau'}} ds  \   \left| h'\left(   \frac{z(s)}{v_t(s)}  \right)-  h'\left(   \frac{\tilde{z}(s)}{\tilde{v}_t(s)}  \right)\right|  \ .
\end{multline}

Since we have by assumption that the fluctuation of $I(\cdot)$ in the interval $[T_{\ve,\tau,\tau'}-\tau, T_{\ve,\tau,\tau'}]$ is small, we should be able to estimate the RHS of (\ref{G7}) by a small constant. However since $h'(y)$ may diverge as $y\ra\ve_0$ we cannot simply apply Taylor's theorem to estimate the RHS of (\ref{G7}).  We can directly apply Taylor's theorem if $t\ge T_{\ve,\tau,\tau'}, \ y\ge\ve_0+\nu_0$  or if $t\ge T_{\ve,\tau,\tau'}+\nu_0, \ y\ge \ve_0,$ for any $\nu_0>0$. To do this we first observe by integrating (\ref{D5}) that
\be \label{H7}
z(s) \ = \ e^{(T_{\ve,\tau,\tau'}-s)/p}z\left(T_{\ve,\tau,\tau'}\right)+\int_s^{T_{\ve,\tau,\tau'}} ds' \ e^{(s'-s)/p}v_t(s')h\left(\frac{z(s')}{v_t(s')}\right) \ , \quad s<T_{\ve,\tau,\tau'} \ ,
\ee
with a similar formula for $\tilde{z}(s)$. Since $z\left(T_{\ve,\tau,\tau'}\right)=\tilde{z}\left(T_{\ve,\tau,\tau'}\right)$ we have from (\ref{H7}) that
\be \label{I7}
\frac{1}{(1+\ve)}\frac{h_\infty}{h(\ve_0)} \ \le \ \frac{\tilde{z}(s)}{z(s)} \ \le \  (1+\ve)\frac{h(\ve_0)}{h_\infty} \ , \quad T_{\ve,\tau,\tau'}-\tau< s<T_{\ve,\tau,\tau'} \ .
\ee
Taylor's theorem implies then upon using (\ref{I7}) and our assumptions on derivatives of the function $h(\cdot)$,  there is a constant $C_4$ such that
 \begin{multline} \label{J7}
|F_1(t,y,v_t(\cdot))- \tilde{F}_1(t,y,\tilde{v}_t(\cdot))| \ \le \\
C_4\tau e^{-(t-T_{\ve,\tau,\tau'})/p}\sup_{T_{\ve,\tau,\tau'}-\tau<s<T_{\ve,\tau,\tau'}}\left[ \frac{|v_t(s)-\tilde{v}_t(s)|}{v_t(s)}+ 
\frac{v_t(s)}{z(s)}\left| \frac{z(s)}{v_t(s)} -  \frac{\tilde{z}(s)}{\tilde{v}_t(s)} \right|\right] \ .
\end{multline}
From our assumptions on $I(\cdot)$ in the interval  $[T_{\ve,\tau,\tau'}-\tau,T_{\ve,\tau,\tau'}]$, we see that the first term in the supremum on the RHS of (\ref{J7}) is bounded  above by $\ve$.  In order to estimate $|1-\tilde{z}(s)/z(s)|$ for $s\in [T_{\ve,\tau,\tau'}-\tau,T_{\ve,\tau,\tau'}]$, we observe that the function $w(s)=z(s)-\tilde{z}(s)$ is a solution to the terminal value problem
\be \label{K7}
\frac{dw(s)}{ds} \ = \ - a(s)w(s)-b(s)\left[v_t(s)-\tilde{v}_t(s)\right] \ ,  \quad  s<T_{\ve,\tau,\tau'} \ , \ \ w(T_{\ve,\tau,\tau'})=0 \ ,
\ee
where the functions $a(\cdot),b(\cdot)$ are given by the formulae
\begin{multline} \label{L7}
a(s)\ = \ \frac{1}{p}+  \int_0^1 d\mu \  h'\left(\mu\frac{z(s)}{v_t(s)}+(1-\mu)\frac{\tilde{z}(s)}{\tilde{v}_t(s)}\right) \ , \\
b(s) \ = \ h\left(\frac{\tilde{z}(s)}{\tilde{v}_t(s)}\right)-\frac{\tilde{z}(s)}{\tilde{v}_t(s)}\int_0^1 d\mu \  h'\left(\mu\frac{z(s)}{v_t(s)}+(1-\mu)\frac{\tilde{z}(s)}{\tilde{v}_t(s)}\right) \ .
\end{multline}
Integrating (\ref{K7}) we have that
\be \label{M7}
w(s) \ = \ \int_s^{T_{\ve,\tau,\tau'}} ds' \ \exp\left[\int_s^{s'} a(s'') \ ds'' \right]b(s')\left[v_t(s')-\tilde{v}_t(s')\right] \ .
\ee
Observe now that $a(\cdot)\le 1/p$, and in view of (\ref{I7}) that $b(\cdot)$ is bounded in the interval 
$[T_{\ve,\tau,\tau'}-\tau,T_{\ve,\tau,\tau'}]$. We conclude from (\ref{H7}), (\ref{M7}) that $|w(s)/z(s)|\le C_4\ve$ for $s\in [T_{\ve,\tau,\tau'}-\tau,T_{\ve,\tau,\tau'}]$, where $C_4$ is independent of $\ve,\tau,\tau'$. We have shown that if $t\ge T_{\ve,\tau,\tau'}, \ y\ge\ve_0+\nu_0$  or if $t\ge T_{\ve,\tau,\tau'}+\nu_0, \ y\ge \ve_0,$ then 
\be \label{N7}
|F_1(t,y,v_t(\cdot))- \tilde{F}_1(t,y,\tilde{v}_t(\cdot))| \ \le \\
C_5\ve \tau e^{-(t-T_{\ve,\tau,\tau'})/p}\
\ee
for some constant $C_5$ independent of $\ve,\tau,\tau'$. 

We estimate the expression on the RHS of (\ref{G7}) when  $T_{\ve,\tau,\tau'}\le t\le T_{\ve,\tau,\tau'}+\nu_0, \ \ve_0\le y\le \ve_0+\nu_0$. Since we can apply the argument of the previous paragraph to integration on the RHS of (\ref{G7}) over the interval $[T_{\ve,\tau,\tau'}-\tau,T_{\ve,\tau,\tau'}-\nu_0]$, we restrict ourselves to the integral over the interval $[T_{\ve,\tau,\tau'}-\nu_0,T_{\ve,\tau,\tau'}]$. Observe from (\ref{C3}), (\ref{C5}), (\ref{D5}) that 
\begin{multline} \label{O7}
{\rm setting} \ y_\mu(s) \ = \ \mu\frac{z(s)}{v_t(s)}+(1-\mu)\frac{\tilde{z}(s)}{\tilde{v}_t(s)} \ , s<T_{\ve,\tau,\tau'}, \quad {\rm then} \\
\frac{dy_\mu(s)}{ds} \ = \ -\mu\left[\rho(s)y_1(s)+h\left(y_1(s)\right)\right]
-(1-\mu)\left[ \frac{y_0(s)}{p}+h\left(y_0(s)\right)\right] \ .
\end{multline}
Upon choosing $\nu_0>0$ sufficiently small, we see from (\ref{O7}) and the inequality  $\inf \rho(\cdot)>-h(\ve_0)/\ve_0$ there exists $c_0>0$ such that $y'_\mu(s)\le -c_0$ for all $0<\mu<1$ and $s\in [T_{\ve,\tau,\tau'}-\nu_0,T_{\ve,\tau,\tau'}]$. It follows that the function $b(\cdot)$ of (\ref{L7}) satisfies the inequality
\be \label{P7}
\int_{T_{\ve,\tau,\tau'}-\nu_0}^{T_{\ve,\tau,\tau'}} |b(s)| \ ds \ \le \ C_0 \ ,
\ee
where the constant $C_0$ is inversely proportional to $c_0$. We conclude from (\ref{H7}), (\ref{P7}) that the function $w(\cdot)$ of (\ref{M7}) satisfies the inequality $|w(s)/z(s)|\le C_5\ve, \  s\in [T_{\ve,\tau,\tau'}-\nu_0,T_{\ve,\tau,\tau'}]$,  for some constant $C_5$. It follows there are constants $C_6,C'_6$ such that
\begin{multline} \label{Q*7}
\left|  \frac{z(s)}{v_t(s)} -   \frac{\tilde{z}(s)}{\tilde{v}_t(s)}\right|  \ \le \ C_6\ve \quad {\rm for \ } T_{\ve,\tau,\tau'}-\nu_0<s<T_{\ve,\tau,\tau'} \ , \\
 \min\left[\frac{z(s)}{v_t(s)}, \    \frac{\tilde{z}(s)}{\tilde{v}_t(s)}\right] \ \ge \ \ve_0+2C_6\ve 
 \quad {\rm for \ } T_{\ve,\tau,\tau'}-\nu_0<s<T_{\ve,\tau,\tau'}-C'_6\ve \ .
\end{multline}
We estimate now
\be \label{Q7}
\int_{T_{\ve,\tau,\tau'}-\nu_0}^{T_{\ve,\tau,\tau'}} ds  \   \left| h'\left(   \frac{z(s)}{v_t(s)}  \right)-  h'\left(   \frac{\tilde{z}(s)}{\tilde{v}_t(s)}  \right)\right|  \ \le \int_{T_{\ve,\tau,\tau'}-C'_6\ve}^{T_{\ve,\tau,\tau'}}+
\int_{T_{\ve,\tau,\tau'}-\nu_0}^{T_{\ve,\tau,\tau'}-C'_6\ve} \ .
\ee
Since $h'(\cdot)$ is integrable the first integral on the right of (\ref{Q7}) converges to $0$ as $\ve\ra0$. Using (\ref{Q*7}) we see that the second integral on the RHS of (\ref{Q7}) also converges to $0$ as $\ve\ra0$. We have therefore shown that the  integral on the RHS of (\ref{G7}) converges to $0$ as $\ve\ra0$.   Evidently the first term on the RHS of (\ref{G7}) also converges to $0$.  

To estimate the second term on the LHS of (\ref{O5}) we write $f(t,v_t(\cdot))=f_1(t,v_t(\cdot))+f_2(t,v_t(\cdot))$, corresponding to the decomposition $F=F_1+F_2$. 
From our bound on $F_2$ we see there is a constant $C_7$ such that
\be \label{R7}
\int_{T_{\ve,\tau,\tau'}}^\infty |f_2(t,v_t(\cdot))| \ dt \ \le \  C_7e^{-\tau/p} \ .
\ee
Letting $\tilde{f}_1(t,\tilde{v}_t(\cdot))$ be the function (\ref{P5}) corresponding to $\tilde{F}_1$ in place of $F$, we have from (\ref{N7}) and the argument of the previous paragraph that
\be \label{S7}
\int_{T_{\ve,\tau,\tau'}}^\infty |f_1(t,v_t(\cdot))-\tilde{f}_1(t,\tilde{v}_t(\cdot))| \ dt \ \le \  C_8\tau\ve+C_9(\ve) \ ,
\ee
for a constant $C_8$ independent of $\ve$, and a constant $C_9(\ve)$ which has the property $\lim_{\ve\ra0}C_9(\ve)=0$. Observe next that by Proposition 5.1 one has $\tilde{F}_1(t,y,\tilde{v}_t(\cdot))\le \tilde{F}_1(t,y,1(\cdot))$ for $t>T_{\ve,\tau,\tau'}$ such that $I(t)=I_{\rm max}(t)$, whence $\tilde{f}_1(t,\tilde{v}_t(\cdot))\ge 0$. We note also from  (\ref{BF2}) of Lemma 2.4 that the function $\ga(\cdot)$ on the RHS of (\ref{O5})  satisfies an inequality $|\ga(t)|\le Ce^{-\del_1t}, \ t\ge 0, $ for some constants $C,\del_1>0$. We conclude then from (\ref{O5}), (\ref{B7}), (\ref{R7}), (\ref{S7}) and the bounds we have on the functions $(t,y)\ra G(t,y,v_t(\cdot))$  and $t\ra\ga(t)$ that
\be \label{T7}
\frac{1}{p}\log\left[\frac{I(T)}{I(T_{\ve,\tau,\tau'})}\right]  \ \le \  C_{10}\left[e^{-\tau'/p}+e^{-\del_1\tau'}\right]+C_7e^{-\tau/p}+C_8\tau\ve+C_9(\ve) \ .
\ee
 Since the constants $C_7, C_8,C_9,C_{10}$ in (\ref{T7}) are independent of $\ve,\tau,\tau'$, we conclude that for any $\del>0$ there exists $T_\del>0$ such that $\sup_{t>T_\del}[I(t)/I(T_\del)-1]<\del$. Since we can make an exactly analogous argument with the function $I_{\rm min}(t)=\inf_{T_{\ve,\ga,\tau}<s<t}I(s)$, we conclude that $\lim_{t\ra\infty}I(t)=I_\infty>0$ exists. 
 
 Alternatively there exists $\ve_0,\tau_0>0, \ \tau_1>\tau_0$ such that \\ $\sup_{s,t\in[T-\tau_0,T]}|I(t)^{1/p}/I(s)^{1/p}-1|\ge \ve_0$ for all $T\ge \tau_1$. Letting $I^+_\infty=\limsup_{t\ra\infty}I(t)$, there exists  for any $\del>0, \ N=1,2,..,$ a  time $T_{\del,N}>\max[\tau_1,N]$  such that $I(T_{\del,N})\ge I_\infty^+-\del$ and $I(t)\le I_\infty^++\del$ for $T_{\del,N}-N\le t\le T_{\del,N}$.  Since the oscillation of $I(\cdot)$ in the interval $[T_{\del,N}-\tau_0,T_{\del,N}]$ exceeds $\ve_0$,  there exists $\tau_{\del,N}\in[T_{\del,N}-\tau_0,T_{\del,N}]$ such that $I(\tau_{\del,N})^{1/p}\le (I_\infty^++\del)^{1/p}/(1+\ve_0)$.  We proceed similarly to before by writing the function $F$ of (\ref{G5}) as $F=F_1+F_2$, where $F_1$ is given by (\ref{E7}), but with the interval of integration now $[T_{\del,N}-N,t]$ in place of $[T_{\ve,\tau,\tau'}-\tau,t]$. As previously, one has the bound $|F_2(t,y,v_t(\cdot))|\le C_3e^{(\tau_0-N)/p}, \ y\ge\ve_0,  \ t\in[T_{\del,N}-\tau_0,T_{\del,N}]$. Evidently $F_1(t,y,v_t(\cdot))$ depends only on the values of $I(s)$ for $s\in [T_{\del,N}-N,t]$. We define $\tilde{v}_t(s)=I(s)^{1/p}/(I^+_\infty+\del)^{1/p}$ for  $s\in [T_{\del,N}-N,t]$, 
 and $\tilde{z}(s), \ s\in [T_{\del,N}-N,t]$ as the solution to (\ref{D5}) with $\tilde{v}_t(\cdot)$ replacing $v_t(\cdot)$. Since $\tilde{v}_t(\cdot)\le1$ we have that  $\tilde{z}(s)/\tilde{v}_t(s)\ge \tilde{z}(s)\ge \tilde{z}(t)\ge\ve_0$ for $s\in [T_{\del,N}-N,t]$. Hence (\ref{D5}) is well-defined for $\tilde{v}_t(\cdot), \ \tilde{z}(\cdot)$.  Similarly to (\ref{I7}), it is easy to see there are positive constants independent of $\del,N,t$ such that
 \be \label{U7}
 c_{12} \ \le \ \frac{\tilde{z}(s)}{z(s)} \ \le  C_{12} \  \quad  {\rm for \ }  \ T_{\del,N}-N<s<t \ .
 \ee
 Setting $w(s)=z(s)-\tilde{z}(s)$ we see using a representation analogous to (\ref{M7}), that for some constant $C_{13}$ one has $|w(s)/z(s)|\le C_{13} J(t),$ where $J(t)=\log\left[(I^+_\infty+\del)^{1/p}/I(t)^{1/p}\right]\ge 0$.

 Let $\tilde{F}_1(t,y,\tilde{v}_t(\cdot))$ be defined in the same way as $F_1(t,y,v_t(\cdot))$, but with $\tilde{v}_t(\cdot), \ \tilde{z}(\cdot)$ replacing $v_t(\cdot), \ z(\cdot)$. The difference $|F_1(t,y,v_t(\cdot))-\tilde{F}_1(t,y,\tilde{v}_t(\cdot))|$ is bounded by the RHS of (\ref{G7}), with the interval of integration now $[T_{\del,N}-N,t]$ in place of $[T_{\ve,\tau,\tau'}-\tau,T_{\ve,\tau,\tau'}]$.  Instead of (\ref{J7}) we have if $y\ge\ve_0+\nu_0$   the estimate
 \begin{multline} \label{V7}
|F_1(t,y,v_t(\cdot))- \tilde{F}_1(t,y,\tilde{v}_t(\cdot))| \ \le \\
C_{14}\int_{T_{\del,N}-N}^t ds \ e^{-(t-s)/p}\left[ \frac{|v_t(s)-\tilde{v}_t(s)|}{v_t(s)}+ 
\frac{v_t(s)}{z(s)}\left| \frac{z(s)}{v_t(s)} -  \frac{\tilde{z}(s)}{\tilde{v}_t(s)} \right|\right] \ ,
\end{multline}
for some constant $C_{14}$. It follows from (\ref{V7})  there is a constant  $C_{15}$ such that $|F_1(t,y,v_t(\cdot))- \tilde{F}_1(t,y,\tilde{v}_t(\cdot))| \le C_{15}J(t)$ for $t\in [T_{\del,N}-\tau_0,T_{\del,N}]$, provided $y\ge\ve_0+\nu_0$. For the case $\ve_0\le y\le \ve_0+\nu_0$ we need only estimate the RHS of (\ref{G7}) with the interval of integration $[t-\nu_0,t]$, since the previous argument applies to the integral over the interval $[T_{\del,N}-N,t-\nu_0]$.  

We shall show that
\be \label{W7}
\int_{\ve_0}^{\ve_0+\nu_0} dy\int_{t-\nu_0}^t ds  \   \left| h'\left(   \frac{z(s,y)}{v_t(s)}  \right)-  h'\left(   \frac{\tilde{z}(s,y)}{\tilde{v}_t(s)}  \right)\right|  \ \le \ C_{15}J(t)   
\ee
for some constant $C_{15}$. First observe that since the LHS of (\ref{W7}) is bounded, we need only consider the situation when $J(t)<<1$.  We proceed similarly to the method used in (\ref{Q7}).  Thus we first observe there are constants $C_{16},C'_{16}$ such that
\begin{multline} \label{X7}
\left|  \frac{z(s,y)}{v_t(s)} -   \frac{\tilde{z}(s,y)}{\tilde{v}_t(s)}\right|  \ \le \ C_{16}J(t) \quad {\rm for \ } 
t-\nu_0<s<t \ ,  \ \ve_0<y<\ve_0+\nu_0, \\
 \min\left[\frac{z(s,y)}{v_t(s)}, \    \frac{\tilde{z}(s,y)}{\tilde{v}_t(s)}\right] \ \ge \ y+2C_{16}J(t) 
 \quad {\rm for \ } t-\nu_0<s<t-C'_{16}J(t) \ .
\end{multline} 
We have now that
\begin{multline} \label{Y7}
\int_{t-C'_{16}J(t)}^t  \left| h'\left(   \frac{z(s,y)}{v_t(s)}  \right)\right|  \ ds \ \le \ 
C_{17}\int^{y+C_{18}J(t)}_y |h'(y')|  \ dy' \\
= \ C_{17}[h(y)-h(y+C_{18}J(t))] \quad {\rm for \ some \ constants} \ C_{17},C_{18} \ .
\end{multline}
It follows from (\ref{Y7}) that
\be \label{Z7}
\int_{\ve_0}^{\ve_0+\nu_0} dy\int_{t-C'_{16}J(t)}^t  ds \ \left| h'\left(   \frac{z(s,y)}{v_t(s)}  \right)\right|  \ 
\le \ C_{17}C_{18}h(\ve_0)J(t) \ .
\ee
Since we can obtain a similar estimate to (\ref{Z7}) when $z(s,y)/v_t(s)$ is replaced by $\tilde{z}(s,y)/\tilde{v}_t(s)$, we need only estimate the integral in (\ref{W7}) for $t-\nu_0<s<t-C'_{16}J(t)$.  To do this we note from the convexity of the function $h(\cdot)$  and (\ref{X7}) the inequality
\begin{multline} \label{AA7}
\left| h'\left(   \frac{z(s,y)}{v_t(s)}  \right)-  h'\left(   \frac{\tilde{z}(s,y)}{\tilde{v}_t(s)}  \right)\right| \ \le
\ h'(y'+C_{16}J(t))-h'(y'-C_{16}J(t)) \ , \\
{\rm where \ } y'=\frac{z(s,y)}{v_t(s)} \ \ {\rm and \ } t-\nu_0<s<t-C'_{16}J(t) \ .
\end{multline}
Integrating (\ref{AA7}) we obtain the inequality
\begin{multline} \label{AB7}
\int_{t-\nu_0}^{ t-C'_{16}J(t)} ds \ \left| h'\left(   \frac{z(s,y)}{v_t(s)}  \right)-  h'\left(   \frac{\tilde{z}(s,y)}{\tilde{v}_t(s)}  \right)\right| \\
 \le \ C_{19}[h(y+C_{16}J(t))-h(y+3C_{16}J(t))] \quad {\rm for \  some \  constant \ } C_{19} \ .
\end{multline}
 Integrating (\ref{AB7}) with respect to $y$ then yields the inequality 
\be \label{AC7}
\int_{\ve_0}^{\ve_0+\nu_0} dy\int_{t-\nu_0}^{t-C'_{16}J(t)} ds  \   \left| h'\left(   \frac{z(s,y)}{v_t(s)}  \right)-  h'\left(   \frac{\tilde{z}(s,y)}{\tilde{v}_t(s)}  \right)\right|  \ \le \ 2C_{19}C_{16}h(\ve_0)J(t)  \ .  
\ee
Now (\ref{W7}) follows from (\ref{Z7}), (\ref{AC7}). 

We estimate the second term on the LHS of (\ref{O5}) by  writing $f(t,v_t(\cdot))=f_1(t,v_t(\cdot))+f_2(t,v_t(\cdot))$, corresponding to the decomposition $F=F_1+F_2$. 
From our bound on $F_2$ we see there is a constant $C_{20}$ such that $|f_2(t,v_t(\cdot))|\le C_{20}e^{-N/p}$ for $t\in [T_{\del,N}-\tau_0,T_{\del,N}]$.
Letting $\tilde{f}_1(t,\tilde{v}_t(\cdot))$ be the function (\ref{P5}) corresponding to $\tilde{F}_1$ in place of $F$, we also have from the previous paragraph  that
$ |f_1(t,v_t(\cdot))-\tilde{f}_1(t,\tilde{v}_t(\cdot))|\le C_{21}J(t)$  for some constant $C_{21} $ if $t\in [T_{\del,N}-\tau_0,T_{\del,N}]$. Furthermore,  Proposition 5.1 implies that $\tilde{F}_1(t,y,\tilde{v}_t(\cdot))\le \tilde{F}_1(t,y,1(\cdot))$ for $t\in [T_{\del,N}-\tau_0,T_{\del,N}]$, whence $\tilde{f}_1(t,\tilde{v}_t(\cdot))\ge 0$ if $t\in [T_{\del,N}-\tau_0,T_{\del,N}]$. It follows now from (\ref{O5})   that
\be \label{AD7}
\frac{dJ(t)}{dt}+C_{22}J(t) \ \ge \ -C_{20}e^{-N/p}-C_{23}[e^{-t/p}+e^{-\del_1 t}] \ , \quad t\in[T_{\del,N}-\tau_0,T_{\del,N}] \ ,
\ee
for some positive constants $C_{22},C_{23}$.  In deriving (\ref{AD7}) we have used the fact that the function $t\ra J(t)$ is non-negative.  Integrating (\ref{AD7}) over the interval $[\tau_{\del,N},T_{\del,N}]$, we obtain the inequality
\be \label{AE7}
J(T_{\del,N} )\ \ge e^{-C_{22}\tau_0}J(\tau_{\del,N})-C_{20}\tau_0e^{-N/p}-C_{23}\tau_0[ e^{-(T_{\del,N}-\tau_0)/p}+e^{-\del_1(T_{\del,N}-\tau_0)}] \ .
\ee
Observe  that $J(T_{\del,N})\le 2\del/p(I^+_\infty-\del)$ and $J(\tau_{\del,N})\ge\log(1+\ve_0)$.  Since $T_{\del,N}\ge N$, the inequality (\ref{AE7}) yields a contradiction if $\del>0$ is sufficiently small and $N$ sufficiently large. 
\end{proof}
\begin{lem}
Assume $\rho(\cdot), \ h(\cdot)$ satisfy the assumptions of Lemma 3.2,  and let $\xi(\cdot,t), \ t>0$ be the solution to (\ref{A3}) with initial condition $\xi(\cdot,0)$ satisfying $\|\xi(\cdot,0)\|_{1,\infty}<\infty$. Assume further that
\be \label{AF7}
\lim_{t\ra\infty} \sup_{t-\tau<s<t}\left|\int_s^t\left[\rho(s')-\frac{1}{p}\right] \ ds'\right| \ = \ 0 \quad {\rm for \ all \ } \tau>0 \ .
\ee
Then $\lim_{t\ra\infty}\|\xi(\cdot,t)-\xi_p(\cdot)\|_{1,\infty}=0$. 
\end{lem}
\begin{proof}
Let $\xi_0(\cdot,\cdot)$ be defined by
\be \label{AG7}
\xi_0(y,t) \ = \ \int_0^tds \ h(y(s))\exp\left[-\int_s^t\rho(s') \ ds'\right] \ , \quad y\ge \ve_0,t>0 \ ,
\ee
where $y(\cdot)$ is the solution to (\ref{C3}).  We see from (\ref{N3}), (\ref{Q4}) upon using the representations (\ref{D3}), (\ref{V3})  for $\xi(\cdot,t), \ D\xi(\cdot,t)$ that $\|\xi(\cdot,t)-\xi_0(\cdot,t)\|_{1,\infty}<C_0e^{-\del_0 t}, \ t>0$ for some constant $C_0$. For any $\tau,\nu>0$ with $0<\nu\le 1$, let $T_{\tau,\nu}>\tau$ have the property that
\be \label{AH7}
\left|\int_s^t\left[\rho(s')-\frac{1}{p}\right] \ ds'\right| \ \le \ \nu \quad {\rm for  \ }t-\tau<s<t, \ t>T_{\tau,\nu} \ .
\ee
Letting $\xi_{0,\tau}(\cdot,t)$ be defined as in (\ref{AG7}), but with the interval of integration $[t-\tau,t]$ instead of $[0,t]$, then we have that $\|\xi_0(\cdot,t)-\xi_{0,\tau}(\cdot,t)\|_{1,\infty}\le Ce^{-\del_0\tau}$ for some constant $C$ independent of $\tau>0$. 

Let $\tilde{\xi}_{0,\tau}(\cdot,t)$ be the function $\xi_{0,\tau}(\cdot,t)$ in the case $\rho(\cdot)\equiv 1/p$.  We wish  to estimate $\|\xi_{0,\tau}(\cdot,t)-\tilde{\xi}_{0,\tau}(\cdot,t)\|_{1,\infty}$.  In order to do this we first need an estimate on the difference $y(s)-y_p(s), \ t-\tau<s<t,$ where $y_p(\cdot)$ is the solution to (\ref{C3}) with $\rho(\cdot)\equiv 1/p$.  Observe from (\ref{N3}) that $u(s)=e^{-(t-s)/p}[y(s)-y_p(s)]$ satisfies the integral equation
\be \label{AI7}
u(s)+\int_s^t K(s')u(s') \ ds' \ = \ g(s) \ , \ s<t \ , 
\ee
where
\begin{multline} \label{AJ7}
K(s) \ = \    -\int_0^1 h'(\la y(s)+(1-\la)y_p(s)) \ ds \ \ge \ 0  \ , \\
g(s) \ = \ \left\{\exp\left[\int_s^t\rho(s')-\frac{1}{p} \ ds'\right]-1\right\}y + \\
\int_s^t ds'  \ e^{-(t-s')/p}\left\{\exp\left[\int_s^{s'}\rho(s'')-\frac{1}{p} \ ds''\right]-1\right\} h(y(s')) \ .
\end{multline}
Observing that $g(t)=0$, we see on differentiating (\ref{AI7}) that $u(s), \ s<t,$ is the solution to the terminal value problem
\be \label{AK7}
\frac{du(s)}{ds}-K(s)u(s) \ = \ g'(s) \ , \quad s<t, \ \ u(t)=0 \ .
\ee
The integral representation for the solution to (\ref{AK7}) is given by
\begin{multline} \label{AL7}
u(s) \ = \ -\int_s^t\exp\left[-\int_{s}^{s'} K(s'') \ ds''\right]g'(s') \ ds' \\
=  \ g(s)-\int_s^t\exp\left[-\int_{s}^{s'} K(s'') \ ds''\right]K(s')g(s') \ ds'  \ .
\end{multline}
Since $K(\cdot)$ is non-negative we conclude from (\ref{AL7}) that
\be \label{AM7}
|u(s)| \ \le \  2\sup_{s\le s'\le t}|g(s')| \ , \quad s<t \  .
\ee
It follows from (\ref{AH7}), (\ref{AJ7}) that
\be \label{AN7}
|g(s)| \ \le [e^{\nu}-1]y+p[e^{2\nu}-1]h(\ve_0) \ , \quad t-\tau<s<t, \ t>T_{\tau,\nu} \ .
\ee
We conclude from (\ref{N3}), (\ref{AM7}), (\ref{AN7}) there is a constant $C$ such that
\be \label{AO7}
|y(s)-y_p(s)| \  \le \ C\nu y_p(s) \quad 0<\nu\le 1, \ \ t-\tau<s<t, \ t>T_{\tau,\nu} \ .
\ee 
It follows easily from (\ref{R4}), (\ref{AO7}) that $\|\xi_{0,\tau}(\cdot,t)-\tilde{\xi}_{0,\tau}(\cdot,t)\|_\infty\le C\nu$ for some constant $C$. To bound the derivative  $D\tilde{\xi}_{0,\tau}(\cdot,t)-D\xi_{0,\tau}(\cdot,t)$ we observe as in (\ref{V3}) that
\be \label{AP7}
D\xi_{0,\tau}(y,t) \ = \ \exp\left[\int_{t-\tau}^th'(y(s)) \ ds\right]-1 \ ,
\ee
with a similar representation for $D\tilde{\xi}_{0,\tau}(\cdot,t)$. Hence using the negativity of $h'(\cdot)$ we have from (\ref{AP7}) and Taylor's theorem that
\be \label{AQ7}
|D\xi_{0,\tau}(y,t)-D\tilde{\xi}_{0,\tau}(y,t)| \ \le \ \left|\int_{t-\tau}^t[h'(y(s))-h'(y_p(s))] \ ds \ \right| \ .
\ee
We can estimate the RHS of (\ref{AQ7}) by applying Taylor's theorem and using (\ref{N3}), (\ref{T4}) to conclude that for any $\nu_0>0$,
\be \label{AR7}
\sup_{y\ge \ve_0+\nu_0} y|D\xi_{0,\tau}(y,t)-D\tilde{\xi}_{0,\tau}(y,t)|  \ \le C(\nu_0)\nu \ , \quad t>T_{\tau,\nu} \ ,
\ee
where the constant $C(\nu_0)$ may depend on $\nu_0>0$. We could extend the estimate of (\ref{AR7}) to the supremum over $y\ge\ve_0$ if we were to replace the interval of integration $[t-\tau,t]$ on the RHS of (\ref{AQ7}) by $[t-\tau,t-\nu_0]$ for any $\nu_0>0$.  Hence if we can show that
\be \label{AS7}
\lim_{t\ra\infty} \sup_{\ve_0<y<\ve_0+\nu_0}\left|\int_{t-\nu_0}^t[h'(y(s))-h'(y_p(s))] \ ds \ \right| \ = \ 0 \ ,
\ee
it follows that $\lim_{t\ra\infty}\|\xi_{0,\tau}(\cdot,t)-\tilde{\xi}_{0,\tau}(\cdot,t)\|_{1,\infty}=0$. The limit (\ref{AS7}) is a consequence of the dominated convergence theorem. 

To complete the proof we observe there is a constant $C$ such that 
\be \label{AT7}
\|\tilde{\xi}_{0,\tau}(\cdot,t)-\xi_p(\cdot)\|_{1,\infty}\le Ce^{-\tau/p} \quad {\rm for \ } 0<\tau<t \ .
\ee
This follows from the identity
\be \label{AU7}
\xi_p(y) \ = \ e^{-t/p}\xi_p(y_p(0))+\int_0^t e^{-(t-s)/p}h(y_p(s)) \ ds  \ .
\ee
 \end{proof}
Under the assumptions of Proposition 7.1, we see from  (\ref{B5}) that the condition (\ref{AF7}) of Lemma 7.1 holds. Thus we obtain a global asymptotic stability theorem in the case when the function $x\ra\beta(x,0)$ is H\"{o}lder continuous at $x=1$. To prove global asymptotic stability under just  a continuity assumption on the function $x\ra\beta(x,0)$  at $x=1$, we need to proceed somewhat differently. 

Recall that Proposition 8.1 of \cite{cd} is a non-linear generalization of Proposition 6.2 of \cite{cd}. This is a stability result for solutions to the linear differential delay equation (DDE)
\be \label{AV7}
\frac{dI(t)}{dt}+\int_0^tk(t,s)[I(t)-I(s)] \ ds \ = \ f(t) \ , \quad t>0 \ ,
\ee
where $k(\cdot,\cdot)$ is non-negative and $f\in L^1(\R^+)$.  We first prove a result for  solutions to (\ref{AV7}) when  $f\in L^\infty(\R^+)$, and then generalize it to the non-linear case.
\begin{lem}
Assume the function $k(\cdot,\cdot)$ of (\ref{AV7}) is non-negative, and the function $b(t)=\int_0^tk(t,s) \ ds, \ t\ge 0,$ is bounded. Assume further  there exists $\tau>0$ such that $k(t,s)=0$ for $t-s>\tau$. Then there exists a constant $C$ such that the solution to (\ref{AV7}) satisfies the inequality $\|I'(\cdot)\|_\infty\le C\|f(\cdot)\|_\infty$. 
\end{lem}
\begin{proof}
Assuming $0\le T_1<T_2$, we integrate (\ref{AV7}) to obtain the formula
\begin{multline} \label{AW7}
I(T_2) \ = \ I(T_1)\exp\left[-\int_{T_1}^{T_2} b(t) \ dt\right] \\
+\int_{T_1}^{T_2}\exp\left[-\int_{t}^{T_2} b(s) \ ds\right] \left\{f(t)+\int_0^tk(t,s)I(s) \ ds\right\} \ dt \ .
\end{multline}
We rewrite (\ref{AW7}) as
\be \label{AX7}
I(T_2)-I(T_1) \ = \ \int_{T_1}^{T_2}\exp\left[-\int_{t}^{T_2} b(s) \ ds\right] f(t)  \ dt
+\del(T_1,T_2)E[I(\mathcal{T})-I(T_1)] \ ,
\ee
where
\be \label{AY7}
\del(T_1,T_2)  \ = \ 1-\exp\left[-\int_{T_1}^{T_2} b(t) \ dt\right] \ ,
\ee
and $\mathcal{T}$ is a random variable with distribution in the interval $[T_1-\ga,T_2]$. 

For $n=1,2,\dots,$ we may use (\ref{AX7}) to estimate the oscillation of $I(\cdot)$ on the interval $[n\tau,(n+1)\tau]$ in terms of the oscillation of $I(\cdot)$ on the interval $[(n-1)\tau,n\tau]$ and  $\sup|f(\cdot)|$ on  $[n\tau,(n+1)\tau]$. We set $T_1=n\tau$ and  choose $T_2\in[n\tau,(n+1)\tau]$ such that
$ I(T_2)=\sup_{n\tau<t<(n+1)\tau}I(t)$. Then (\ref{AX7}) yields the inequality
\begin{multline} \label{BA7}
\sup_{n\tau<t<(n+1)\tau}[I(t)-I(n\tau)] \ \le \ \tau\sup_{n\tau<t<(n+1)\tau}|f(t)| \\
+\del\al\left\{\sup_{n\tau<t<(n+1)\tau}[I(t)-I(n\tau)] \right\}  
+\del(1-\al)\left\{\sup_{(n-1)\tau<t<n\tau}[I(t)-I(n\tau)] \right\} \ ,
\end{multline}
where $\del,\al$ are given by
\be \label{BB7}
\del \ = \ 1-\exp\left[-\tau\sup b(\cdot)\right] \ , \quad \al \ = \ P(\mathcal{T}>n\tau) \ .
\ee
It follows from (\ref{BA7}) that
\begin{multline} \label{BC7}
\sup_{n\tau<t<(n+1)\tau}[I(t)-I(n\tau)] \ \le \ \frac{\tau}{1-\del\al}\sup_{n\tau<t<(n+1)\tau}|f(t)| \\
+\frac{\del(1-\al)}{1-\del\al}\left\{\sup_{(n-1)\tau<t<n\tau}[I(t)-I(n\tau)] \right\} \ .
\end{multline}
We may apply a similar argument using (\ref{AX7}) with $T_1=n\tau$ and $T_2\in[n\tau,(n+1)\tau]$ such that $ I(T_2)=\inf_{n\tau<t<(n+1)\tau}I(t)$.  This yields the inequality
\begin{multline} \label{BD7}
\sup_{n\tau<t<(n+1)\tau}[I(n\tau)-I(t)] \ \le \ \frac{\tau}{1-\del\al}\sup_{n\tau<t<(n+1)\tau}|f(t)| \\
+\frac{\del(1-\al)}{1-\del\al}\left\{\sup_{(n-1)\tau<t<n\tau}[I(n\tau)-I(t)] \right\} \ .
\end{multline}
Adding (\ref{BC7}) and (\ref{BD7}) we obtain the estimate
\begin{multline} \label{BE7}
\sup_{n\tau<s,t<(n+1)\tau}|I(t)-I(s)| \ \le \ \frac{2\tau}{1-\del}\sup_{n\tau<t<(n+1)\tau}|f(t)| \\
+\del\left\{\sup_{(n-1)\tau<s,t<n\tau}|I(t)-I(s)| \right\} \ .
\end{multline}
We conclude from (\ref{BE7}) that
\begin{multline} \label{BF7}
\sup_{n\tau<s,t<(n+1)\tau}|I(t)-I(s)| \ \le \  \frac{2\tau(1-\del^n)}{(1-\del)^2}\|f(\cdot)\|_\infty \\
+ \del^n \sup_{0<s,t<\tau}|I(t)-I(s)|  \ , \quad {\rm for \ } n=1,2,...
\end{multline}
Following the same argument as before, we have from (\ref{AX7}) that
\be \label{BG7}
\sup_{0<s,t<\tau}|I(t)-I(s)|   \ \le \  \frac{2\tau}{(1-\del)^2}\|f(\cdot)\|_\infty  \ .
\ee

We conclude from  (\ref{BF7}), (\ref{BG7}) that
\be \label{BH7}
\sup_{n\tau<s,t<(n+1)\tau}|I(t)-I(s)|   \ \le \  \frac{2\tau}{(1-\del)^2}\|f(\cdot)\|_\infty  \quad {\rm for \ } n=0,1,2,...
\ee
It follows now from (\ref{AV7}), (\ref{BH7}) that
\be \label{BI7}
\|I'(\cdot)\|_\infty \ \le \ \left[ \frac{4\tau\sup b(\cdot)}{(1-\del)^2}+1\right]\|f(\cdot)\|_\infty \ .
\ee
\end{proof}
\begin{rem}
Lemma 7.2 implies a result for Volterra integral equations. Thus consider the integral equation
\be \label{BJ7}
u(t)+\int_0^tK(t,s)u(s) \ ds \ = \ f(t) \ , \quad t>0 \ ,
\ee
with continuous kernel $K(t,s), \ 0\le s\le t<\infty$. 
Assume the functions $s\ra K(t,s), \ 0\le s\le t$, are increasing for all $t>0$,  there exists $\tau>0$ such that $K(t,s)=0$ for $t-s>\tau$,  and that $\sup_{t>0} K(t,t)<\infty$. Then there is a constant $C$ such that the solution to (\ref{BJ7}) satisfies $\|u(\cdot)\|_\infty\le C \|f(\cdot)\|_\infty$.  

One should compare this result to the analogous result of Gripenberg (Theorem 9.1 of Chapter 9 of \cite{grip}), given as Proposition 6.1 of \cite{cd}. The monotonicity assumption on $K(\cdot,\cdot)$ in this case is that the functions $t\ra K(t,s)$ are decreasing on $[s,\infty)$ for all $s\ge 0$. 
\end{rem}
\begin{proposition}
Let $h(\cdot), \xi(\cdot,\cdot)$ satisfy the assumptions of Proposition 7.1,  $\eta(\cdot)$ satisfy (\ref{Z3}), and  the function $x\ra\beta(x,0)$ be continuous at $x=1$. Then (\ref{AF7}) holds.  
\end{proposition}
\begin{proof}
We first observe that (\ref{B7}) holds, but not (\ref{C7}) in general.  We replace (\ref{C7}) by the inequality
\be \label{BK7}
c_\al e^{-\al(t-s)} \ \le \ \left[ \frac{I(t)}{I(s)}\right]^{1/p} \ \le C_\al e^{\al(t-s)}  \quad 0\le s\le t<\infty \ ,
\ee
which is valid for any $\al>0$, where $C_\al,c_\al$ are  positive constants depending on $\al$.  This follows from (\ref{Z3}), (\ref{AG3}) and (\ref{BQ2}) of Lemma 2.5 upon integrating (\ref{AF3}) over the interval $[s,t]$. It follows from (\ref{BK7}) that the function $G$ of (\ref{N5}) satisfies an inequality
$|G(t,y,v_t(\cdot))|\le C_\al e^{-t(1/p-\al)}, \ y\ge\ve_0, \ t\ge 0$, and hence that   $|g(t,v_t(\cdot))|\le C_\al e^{-t(1/p-\al)}, \ t\ge 0$, where $\al>0$ can be arbitrarily small, and the constant $C_\al$ depends on $\al>0$.  We also have from (\ref{BE2}) of Lemma 2.4 that the function $\ga(\cdot)$ on the RHS of (\ref{O5}) is bounded and $\lim_{t\ra\infty}\ga(t)=0$. 

We proceed now as in the proof of Proposition 7.1 by writing the function $F$ of (\ref{G5}) as $F=F_1+F_2$, where $F_1$ is given by (\ref{E7}), but with the interval of integration $[\max\{t-\tau,0\},t]$ in place of $[T_{\ve,\tau,\tau'}-\tau,t]$. From (\ref{BK7}) we see that $|F_2(t,y,v_t(\cdot))|\le C_\al e^{-\tau(1/p-\al)}, \ y\ge\ve_0, \ t\ge \tau$, where $\al>0$ can be arbitrarily small.  Evidently $F_1(t,y,v_t(\cdot))$ depends only on the values of $I(s)$ for $s\in [t-\tau,t]$.  Let $I^*>0$ be a constant and define $\tilde{v}_t(s)=I(s)^{1/p}/(I^*)^{1/p}$ for $\max\{t-\tau,0\}<s<t$, and $\tilde{z}(s), \ \ s\in[\max\{t-\tau,0\},t]$ as the solution to (\ref{D5}) with $\tilde{v}_t(\cdot)$ replacing $v_t(\cdot)$.  Let
$\tilde{F}_1(t,y,\tilde{v}_t(\cdot))$ be defined in the same way as $F_1(t,y,v_t(\cdot))$, but with $\tilde{v}_t(\cdot), \ \tilde{z}(\cdot)$ replacing $v_t(\cdot), \ z(\cdot)$. We write the second term on the LHS of (\ref{O5}) as $f(t,v_t(\cdot))=f_1(t,v_t(\cdot))+f_2(t,v_t(\cdot))$, corresponding to the decomposition $F=F_1+F_2$.  We also denote by $\tilde{f}_1(t,\tilde{v}_t(\cdot))$ the function (\ref{P5}) corresponding to $\tilde{F}_1$ in place of $F$. On setting $J(t)=\log\left[(I^*)^{1/p}/I(t)^{1/p}\right]$, we see that (\ref{O5}) is equivalent to the equation
\begin{multline} \label{BL7}
\frac{dJ(t)}{dt}+[\tilde{f}_1(t,\tilde{v}_t(\cdot))-f_1(t,v_t(\cdot))] 
 = \ \tilde{f}_1(t,\tilde{v}_t(\cdot))+ \Ga(t) \ , \\
 {\rm where \ } \quad \Ga(t) \ = \ f_2(t,v_t(\cdot))-g(t,v_t(\cdot))-\ga(t) \ .
\end{multline}

Assume now that for some given $T_1\ge\tau$, the constant $I^*$ has been chosen so that $J(\cdot)$ is non-negative in the interval $[T_1-\tau,T_1]$.  Let $T_2>T_1$ be such that $J(s)\ge 0$ also for $T_1\le s\le T_2$. Then $\tilde{f}_1(t,\tilde{v}_t(\cdot))\ge 0$ for $T_1\le s\le T_2$. Note from (\ref{BK7}) that
\be \label{BM7}
\frac{1}{C_\al}e^{-\al\tau} \ \le \ v_t(s) \ \le \ \frac{1}{c_\al}e^{\al\tau} \ , \quad \max\{t-\tau,0\}\le s\le t \ .
\ee
We also have that $\tilde{v}_t(\cdot)\le 1$, so $\tilde{z}(s)/\tilde{v}_t(s)>\ve_0$ for $s<t$. Hence we may argue as in the proof of Proposition 7.1 that 
\be \label{BO7}
|f_1(t,v_t(\cdot))-\tilde{f}_1(t,\tilde{v}_t(\cdot))|\le C(\tau)J(t) \ , \quad T_1<t<T_2 \ ,
\ee
where the constant $C(\tau)$ depends on $\tau$ since we need to use the inequality (\ref{BM7}). 
Integrating (\ref{BL7}), using (\ref{BO7}) and the non-negativity of  $\tilde{f}_1$, we obtain the inequality
\be \label{BP7}
J(T_2) \ \ge \ \exp\left[-C(\tau)\{T_2-T_1\}\right]J(T_1)-\int_{T_1}^{T_2}
\exp\left[-C(\tau)\{T_2-t\}\right]|\Ga(t)| \ dt \ .
\ee

We can use (\ref{BP7}) to estimate the oscillation of $I(\cdot)$ in the interval $[n\tau,(n+1)\tau]$ in terms of the oscillation of $I(\cdot)$ in the interval $[(n-1)\tau,n\tau], \ n=1,2,...$.  We set $T_1=n\tau$ and define $I^*$ in such a way that the RHS of (\ref{BP7}) is non-negative for $T_1\le t\le T_2\le (n+1)\tau$. Thus we require $I^*$ to satisfy the inequality
\be \label{BQ7}
\log\left[\left(\frac{I^*}{I(n\tau)}\right)^{1/p}\right] \ \ge \ C(\tau)^{-1}\left[e^{\tau C(\tau)}-1\right]\sup_{n\tau<t<(n+1)\tau} |\Ga(t)| \ .
\ee
We also require $I^*\ge \sup_{(n-1)\tau<t<n\tau}I(t)$ so that $J(t)\ge 0$ for $(n-1)\tau<t<n\tau$.  Hence we define $I^*$ as
\be \label{BR7}
\log\left[\left(\frac{I^*}{\sup_{(n-1)\tau<t<n\tau}I(t)}\right)^{1/p}\right] \ = \ C(\tau)^{-1}\left[e^{\tau C(\tau)}-1\right]\sup_{n\tau<t<(n+1)\tau} |\Ga(t)| \ .
\ee
In that case (\ref{BP7}) holds for all $T_2$ such that $n\tau<T_2<(n+1)\tau$, whence we obtain the inequality
\begin{multline} \label{BS7}
\log\left[\left(\frac{I^*}{I(t)}\right)^{1/p}\right] \ \ge \ e^{-\tau C(\tau)}\log\left[\left(\frac{I^*}{I(n\tau)}\right)^{1/p}\right] \\
 -C(\tau)^{-1}[1-e^{-\tau C(\tau}]\sup_{n\tau<t<(n+1)\tau} |\Ga(t)| \quad {\rm for \ } n\tau<t<(n+1)\tau \ .
\end{multline}

Analogously to the previous paragraphs, we let $I_*$ be a constant and set $J(t)= \log\left[I(t)^{1/p}/I_*^{1/p}\right]$. Defining $\tilde{v}_t(s)=I(s)^{1/p}/I_*^{1/p}$, we see as in (\ref{BL7}) that $J(t)$ is a solution to the equation
\begin{multline} \label{BT7}
\frac{dJ(t)}{dt}+[f_1(t,v_t(\cdot))-\tilde{f}_1(t,\tilde{v}_t(\cdot))] 
 = \ -\tilde{f}_1(t,\tilde{v}_t(\cdot))- \Ga(t) \ , \\
 {\rm where \ } \quad \Ga(t) \ = \ f_2(t,v_t(\cdot))-g(t,v_t(\cdot))-\ga(t) \ .
\end{multline}
Assume now that $I_*$ has been chosen so that $J(\cdot)$ is non-negative in the interval $[T_1-\tau,T_1]$.  If $T_2>T_1$ is such that $J(t)\ge 0$ for $T_1<t<T_2$, then  the function $\tilde{f}_1(t,\tilde{v}_t(\cdot))$ is  negative for $T_1<t<T_2$.  Hence if (\ref{BO7}) holds, we obtain upon integrating (\ref{BT7}) the inequality (\ref{BP7}) again.  We choose $T_1=n\tau$ and $I_*$ to be given by
\be \label{BU7}
\log\left[\left(\frac{\inf_{(n-1)\tau<t<n\tau}I(t)}{I_*}\right)^{1/p}\right] \ = \ C(\tau)^{-1}\left[e^{\tau C(\tau)}-1\right]\sup_{n\tau<t<(n+1)\tau} |\Ga(t)| \ .
\ee
In that case (\ref{BP7}) holds for all $T_2$ such that $n\tau<T_2<(n+1)\tau$, whence we obtain the inequality
\begin{multline} \label{BV7}
\log\left[\left(\frac{I(t)}{I_*}\right)^{1/p}\right] \ \ge \ e^{-\tau C(\tau)}\log\left[\left(\frac{I(n\tau)}{I_*}\right)^{1/p}\right] \\
 -C(\tau)^{-1}[1-e^{-\tau C(\tau}]\sup_{n\tau<t<(n+1)\tau} |\Ga(t)| \quad {\rm for \ } n\tau<t<(n+1)\tau \ .
\end{multline}
Adding (\ref{BS7}) to (\ref{BV7}) we have that
\begin{multline} \label{BW7}
\log\left[\left(\frac{I(s)}{I(t)}\right)^{1/p}\right] \ \ge \ \left[e^{-\tau C(\tau)}-1\right]\log\left[\left(\frac{I^*}{I_*}\right)^{1/p}\right] \\
 -2C(\tau)^{-1}[1-e^{-\tau C(\tau}]\sup_{n\tau<t<(n+1)\tau} |\Ga(t)| \quad {\rm for \ } n\tau<s,t<(n+1)\tau \ .
\end{multline}

Upon taking the infimum of the LHS of (\ref{BW7}) over $n\tau<s,t<(n+1)\tau$, and using the formulae (\ref{BR7}), (\ref{BU7}) for $I^*,I_*$, we obtain the estimate 
\begin{multline} \label{BX7}
\log\left[\left(\frac{\sup_{n\tau<t<(n+1)\tau}I(t)}{\inf_{n\tau<t<(n+1)\tau}I(t)}\right)^{1/p}\right]  \ \le \\
[1-e^{-\tau C(\tau}]\left\{
\log\left[\left(\frac{\sup_{(n-1)\tau<t<n\tau}I(t)}{\inf_{(n-1)\tau<t<n\tau}I(t)}\right)^{1/p}\right]+
2\frac{\exp[\tau C(\tau)]}{C(\tau)}\sup_{n\tau<t<(n+1)\tau}|\Ga(t)|\right\} \ .
\end{multline}
Arguing now as in the proof of Lemma 7.2 we see from (\ref{BX7}) that for any integers $n\ge  N\ge 1$ there is the inequality
\begin{multline} \label{BY7}
\log\left[\left(\frac{\sup_{n\tau<t<(n+1)\tau}I(t)}{\inf_{n\tau<t<(n+1)\tau}I(t)}\right)^{1/p}\right]  \ \le \\
  [1-e^{-\tau C(\tau}]^{n+1-N}
\log\left[\left(\frac{\sup_{(N-1)\tau<t<N\tau}I(t)}{\inf_{(N-1)\tau<t<N\tau}I(t)}\right)^{1/p}\right] 
+
2\frac{\exp[2\tau C(\tau)]}{C(\tau)}\sup_{t>N\tau}|\Ga(t)| \ .
\end{multline}
The property (\ref{AF7})  follows now from (\ref{BK7}), (\ref{BY7}) upon using the fact that  $\lim_{t\ra\infty}\Ga(t)=0$.

We are left to establish that (\ref{BO7}) holds with $\tilde{v}_t(s)=I(s)^{1/p}/I_*^{1/p}$, where $I^*$ is given by (\ref{BU7}).  To see this we first observe that $y(s)=\tilde{z}(s)/\tilde{v}_t(s), \ s<t,$ is a solution to (\ref{C3}) with terminal data $y(t)=\tilde{z}(t)/\tilde{v}_t(t)\le z(t)=y$, where $z(\cdot)$ is the solution to the terminal value problem (\ref{D5}).  If $y(t)\ge \ve_0$ we can argue as previously. However if $y(t)<\ve_0$ then the RHS of (\ref{C3}) is not necessarily defined. To circumvent this problem we simply define $\tilde{f}_1(t,\tilde{v}_t(\cdot))$ as given by  the function (\ref{P5}) corresponding to $\tilde{F}_1$ in place of $F$, but with the modification that $[dI_p(\xi(\cdot,t)),\tilde{F}_1(t,\cdot,\tilde{v}_t(\cdot))-\tilde{F}_1(t,\cdot,1(\cdot))]$ is replaced by 
\be \label{BZ7}
\int_{\ve_0\tilde{v}_t(t)}^\infty dy \ dI_p(\xi(\cdot,t);y)\{\tilde{F}_1(t,y,\tilde{v}_t(\cdot))-\tilde{F}_1(t,y,1(\cdot))\} \ .
\ee
To compare this modified $\tilde{f}_1(t,\tilde{v}_t(\cdot))$ to $f_1(t,v_t(\cdot))$ we write $f_1(t,v_t(\cdot))$  as a sum of two parts, the first part corresponding to integration with respect to $y$ in the region $\ve_0\tilde{v}_t(t)<y<\infty$.  We can then estimate the difference between it and 
$\tilde{f}_1(t,\tilde{v}_t(\cdot))$ similarly as before to obtain an upper bound which is a constant times $J(t)$. The second part of $f_1(t,v_t(\cdot))$ is an integral over the interval $\ve_0<y<\ve_0\tilde{v}_t(t)=\ve_0\exp[J(t)]$. Since $F_1(t,y,v_t(\cdot))$ is uniformly bounded, we conclude that this integral is also bounded by a constant times $J(t)$. 
\end{proof}
 \begin{theorem}
 Let $h(\cdot)$ satisfy the assumptions of Lemma 7.1 and  $\xi(\cdot,t), \ t>0,$ be the solution of  (\ref{A3}) with $\rho(t)=\rho(\xi(\cdot,t),\eta(t))$ constructed in Proposition 3.1.  If  $\eta(\cdot)$ satisfies (\ref{Z3})  then
  \be \label{CA7}
\lim_{t\ra\infty} \|\xi(\cdot,t)-\xi_p(\cdot)\|_{1,\infty} \ = \ 0, \quad  \lim_{t\ra\infty}\left[\rho(\xi(\cdot,t),\eta(t))-\frac{1}{p}\right] \ = \ 0  \ .
 \ee
 Suppose in addition the inequality $\eta(t)\le C_1e^{-\nu_1 t}, \ t>0,$ holds for some constants $C_1,\nu_1>0$, and the function $x\ra\beta(x,0)$  is H\"{o}lder continuous at $x=1$. Then there are constants $C_2>0$ and $\nu_2,\  0<\nu_2\le1/p,$ such that 
 \be \label{CB7}
 \|\xi(\cdot,t)-\xi_p(\cdot)\|_{1,\infty}\le C_2e^{-\nu_2 t}, \quad  \left|\rho(\xi(\cdot,t),\eta(t))-\frac{1}{p}\right| \ \le \ C_2e^{-\nu_2t} \ , \ \  t\ge 0  \ .
 \ee
  \end{theorem}
  \begin{proof}
  From (\ref{B5}) and Proposition 7.2 we see that the condition (\ref{AF7}) of Lemma 7.1 holds. The  convergence of $\xi(\cdot,t)$ to $\xi_p(\cdot)$ then follows from Lemma 7.1. The convergence of $\rho(\xi(\cdot,t),\eta(t))$  follows now from (\ref{B4}) upon writing
  $\rho(\xi(\cdot,t),\eta(t))=\rho_p(\xi(\cdot,t))+\ga(t)$, noting  we have already shown  that $\ga(t)$ converges  to $0$. The exponential convergence (\ref{CB7}) can already be obtained from Proposition 7.1 since the exponential decay of $\eta(\cdot)$ implies the exponential decay of  $\ga(\cdot)$.
  \end{proof}
  
 \vspace{.1in}
 \section{Asymptotic Stability for the LSW Model}
Just as we obtained the proof of Theorem 1.2 from Theorem 3.1, we show here that Theorem 7.1 enables us to generalize Theorem 1.3 beyond the case of quadratic $\phi(\cdot)$ and  $\psi(\cdot)$. 
\begin{theorem} Let $w(x,t), \ x,t\ge 0$, be the solution to (\ref{A1}), (\ref{B1}) with coefficients satisfying (\ref{D1}), (\ref{E1}), (\ref{I1}) and assume that the initial data $w(\cdot,0)$ has beta function $\beta(\cdot,0)$ satisfying (\ref{H1}) with $0<\beta_0<1$.  Assume also that the function $h:[\ve_0,\infty)\ra\R$ defined by (\ref{I2})  is  convex, that (\ref{I4}) holds for all  $p>0$ and also  (\ref{S5}). 
  Then setting $\kappa=[1/\beta_0-\phi'(1)-1]/|\psi'(1)|$ one has  
\be \label{A8}
\lim_{t\ra\infty} \kappa(t ) =   \kappa, \quad \lim_{t\ra\infty} \|\beta(\cdot,t)-\beta_\kappa(\cdot)\|_\infty=0,
\ee
where $\beta_\kappa(\cdot)$ is the beta function of the time independent solution $w_\kappa(\cdot)$ of (\ref{A1}).  If the function $x\ra\beta(x,0)$  is H\"{o}lder continuous at $x=1$ then the convergence in (\ref{A8}) is exponential:
\be \label{A*8}
|\kappa(t)-\kappa| \ \le \ Ce^{-\nu t} \ , \quad \|\beta(\cdot,t)-\beta_\kappa(\cdot)\|_\infty \ \le Ce^{-\nu t} \ ,  \ \ t\ge 0 ,
\ee
for some constants $C,\nu>0$.
\end{theorem}
\begin{proof}
We show that (\ref{T4}) holds for the function $h(\cdot)$ defined by (\ref{I2}). Given the further properties of $h(\cdot)$  established in the proof of Theorem 1.2  in $\S3$, we conclude that $h(\cdot)$ satisfies all the assumptions of Theorem 7.1. To prove (\ref{T4}) we use the formula
\be \label{B8}
h'(y) \ = \ \frac{\phi(x)}{\psi(x)}[\psi'(x)+\psi'(1)]-[\phi'(x)+\phi'(1)] \  ,
\ee
which is equivalent to (6.9) of \cite{cn}. 
One easily sees from (\ref{C2}), (\ref{B8}) that $\lim_{y\ra\infty}yh'(y)=0$. It follows from this,  upon using Taylor's theorem about $y=\infty$ or $x=1$, that $\lim_{y\ra\infty}y^2h''(y)=0$.
We can obtain more precise information on the behavior of $h''(y)$ as $y\ra\infty$ by  differentiating (\ref{B8}) with respect to $x$. Upon using (\ref{C2}) we then obtain the formula
\be \label{C8}
yh''(y) \ = \ \frac{\phi(x)\psi'(x)-\phi'(x)\psi(x)}{\psi'(1)\psi(x)}\left[\psi'(x)+\psi'(1)\right] +\frac{\phi''(x)\psi(x)-\phi(x)\psi''(x)}{\psi'(1)} \ .
\ee
We see from (\ref{C2}), (\ref{C8}) that 
\be \label{D8}
\lim_{y\ra\infty}y^3h''(y) \ = \ \frac{1}{3\al_0^2\psi'(1)}\left[\phi'''(1)\psi'(1)-\phi'(1)\psi'''(1)\right] \ .
\ee
Note that (\ref{D1}), (\ref{E1}), (\ref{I1})  imply the RHS of (\ref{D8}) is non-negative, whence the function $y\ra h(y)$ is convex for large $y$. 

In order to apply Theorem 7.1 we need to show that 
$\|\xi(\cdot,T_0)\|_{1,\infty}<\infty$ for all sufficiently large $T_0$. We have already shown this in the proof of Theorem 1.2 in $\S3$.  Applying now Theorem 7.1 and using the identity (\ref{AO3})  we have that $\lim_{t\ra\infty}\kappa(t)=\kappa$. To show that $\beta(\cdot,t)$ converges to $\beta_\kappa(\cdot)$ we argue as in the proof of Proposition 3.1 of \cite{cn}.  Thus it is sufficient to show that the function  $g(x,t)=w(x,t)/w_\kappa(x)$ satisfies an inequality
\be \label{E8}
\left|(1-x)\frac{\pa}{\pa x}\log g(x,t)\right| \   \le \ \del(t) \ , \quad 0\le x<1, \ t\ge 0 \ ,
\ee
where $\lim_{t\ra\infty}\del(t)=0$. In the case of $\beta(x,0)$ being H\"{o}lder continuous at $x=1$ then $|\del(t)|\le Ce^{-\nu t}, \ t\ge 0,$ for some positive constants $C,\nu$.  To prove (\ref{E8}) we observe from (\ref{A1}), (\ref{AI2}), upon using the identity $w(x,t)=e^tw(F(x,t),0)$, that
\be \label{F8}
\frac{\pa}{\pa x} \log g(x,t) \ = \ \frac{1}{\kappa\psi(x)-\phi(x)}-\frac{\pa F(x,t)}{\pa x}\frac{\beta(F(x,t),0)}{\int_{F(x,t)}^11-\beta(x',0) \ dx'} \ .
\ee
From (\ref{AK3}), (\ref{AL3}) we have that
\be \label{G8}
\frac{1}{1-F(x,t)}\frac{\pa F(x,t)}{\pa x} \ = \ \frac{f(x)\psi(F(x,t))u(t)}{[1-F(x,t)]\psi(x)f(F(x,t))}\left[1+D\xi(f(x)/\al_0,t)\right] \ .  
\ee
Since $\lim_{t\ra\infty}F(0,t)=1$ it follows from (\ref{C2}) that $\lim_{t\ra\infty} [1-F(x,t)]f(F(x,t))=1$, uniformly for $0\le x<1$.  We also have from (\ref{C2}), (\ref{F2}), (\ref{H2}) that
\begin{multline} \label{H8}
\lim_{t\ra\infty}\psi(F(x,t))u(t) \ = \ -\psi'(1)\lim_{t\ra\infty}\frac{1}{f(F(x,t))}u(t)  \\
= \ -\psi'(1)\lim_{t\ra\infty}\frac{1}{f(x)+\al_0\xi(f(x)/\al_0,t)} \ , \quad {\rm uniformly \ for \ } 0\le x<1 \ .
\end{multline}
We conclude from (\ref{H1}), (\ref{G8}), (\ref{H8}) and Theorem 7.1 that
\begin{multline} \label{I8}
\lim_{t\ra\infty} \psi(x)\frac{\pa F(x,t)}{\pa x}\frac{\beta(F(x,t),0)}{\int_{F(x,t)}^11-\beta(x',0) \ dx'}  \\
 = \ \frac{p|\psi'(1)|y[1+D\xi_p(y)]}{[y+\xi_p(y)]} \ {\rm where} \ y=f(x)/\al_0 \ , \quad {\rm uniformly \ for \ } 0\le x<1 \ .
\end{multline}
Observe now from (\ref{N1}) that
\be \label{J8}
\frac{y+\xi_p(y)}{p[1+D\xi_p(y)]} \ = \ h(y)+\frac{y}{p} \ .
\ee
It follows from (\ref{I2}), (\ref{J8}) that the expression on the RHS of (\ref{I8}) is the same as $1/[\kappa-\phi(x)/\psi(x)]$. Hence the inequality (\ref{E8})  with $\lim_{t\ra\infty}\del(t)=0$ holds.  It was already shown in the proof of Proposition 3.1 of \cite{cn} that (\ref{E8}) implies  $\lim_{t\ra\infty} \|\beta(\cdot,t)-\beta_\kappa(\cdot)\|_\infty=0$. We have established (\ref{A8}).  A similar argument gives (\ref{A*8}) in the case when the function $x\ra\beta(x,0)$ is H\"{o}lder continuous at $x=1$. 
\end{proof}
\begin{rem}
Theorem 3.1 of  \cite{nv1} establishes a necessary condition for local asymptotic stability of the LSW model  in the case when the functions $\phi(\cdot), \ \psi(\cdot)$ are given by  (\ref{C1}) with $\al=1/3$. 
The condition is given in terms of the function $S_0:[0,\infty)\ra\R$, where $w(x,0)=S_0(z)e^{-z}$ and
\begin{multline} \label{AO8}
z \ = \ \int_0^x \frac{dx'}{x'-x'^{1/3}+\kappa[1-x'^{1/3}]} \ = \\
|\psi'(1)| \int_0^x\frac{y \ dx'}{\psi(x')[h(y)+y/p]} \ , \quad {\rm with \ } y=y(x') \ .
\end{multline}
From (\ref{C2}), (\ref{AO8}) we have that
\be \label{AQ8}
\frac{dy}{dx} \ = \ -\frac{y\psi'(1)}{\psi(x)} \ ,  \quad \frac{dz}{dx} \ = \ = \frac{1}{\kappa\psi(x)-\phi(x)} \ =  \ \frac{|\psi'(1)|y}{\psi(x)[h(y)+y/p]} \ .
\ee
The necessary condition on $S_0(\cdot)$ is that
\be \label{AU8}
\lim_{z\ra\infty}\sup_{z\le y\le z+1}\frac{|S_0(z)-S_0(y)|}{S_0(z)} \ = \ 0 \ .
\ee
Evidently (\ref{AU8}) holds if $S_0(\cdot)$ is $C^1$ and
\be \label{AV8}
\lim_{z\ra\infty}\frac{S'_0(z)}{S_0(z)} \ = \ \lim_{z\ra\infty}\frac{d}{dz}\log S_0(z) \ = \  0 \  .
\ee
From (\ref{AI2}), (\ref{AQ8}) we have that
\be \label{AW8}
-\frac{d}{dz}\log w(x,0) \ = \ \frac{\psi(x)}{|\psi'(1)|}\left[\frac{h(y)}{y}+\frac{1}{p}\right]\beta(x,0)\Big/\int_x^1[1-\beta(x',0)] \ dx' \ .
\ee
Taking $z\ra\infty$ in (\ref{AW8}) and using (\ref{H1}) we have that
\be \label{AX8}
\lim_{z\ra\infty} \frac{d}{dz}\log w(x,0) \ = \ -\frac{\beta_0}{p(1-\beta_0)} \ = \ -1 \ .
\ee
Evidently (\ref{AX8}) implies (\ref{AV8}). 

Theorem 3.2 of \cite{nv1} proves local asymptotic stability. The condition on the initial data is that
\be \label{AU*8}
\sup_{z\ge 0}\sup_{z\le y\le z+1}\frac{|S_0(z)-S_0(y)|}{S_0(z)} \ < \ \ve \quad {\rm for \ sufficiently  \ small \ } \ve\  ,
\ee
and also that (\ref{AU8}) holds. Observe that if we set $\beta(\cdot,0)\equiv\beta_\kappa(\cdot)$ on the RHS of (\ref{AW8}) then the LHS is equal to $-1$. It follows there exists $\del>0$, depending on $\ve$, such that if $\|\beta(\cdot,0)-\beta_\kappa(\cdot)\|_\infty<\del$ then (\ref{AU*8}) holds. 
\end{rem}
Next we obtain conditions on the functions $\phi(\cdot), \psi(\cdot)$ which imply that (\ref{I4}) and (\ref{S5}) hold.
\begin{lem}
Assume that $\phi(\cdot) \ \psi(\cdot)$ satisfy (\ref{D1}), (\ref{E1}), (\ref{I1}) and $h(\cdot)$ is defined by (\ref{I2}). Then (\ref{I4}) holds for all $p>0$ if and only if the function
\be \label{K8}
x\ra \frac{\{\phi'(x)+\phi'(1)\}\psi(x)-\{\psi'(x)+\psi'(1)\}\phi(x)}{\psi'(1)\phi(x)-\phi'(1)\psi(x)} \quad {\rm is  \ decreasing.} 
\ee 
The inequality (\ref{S5}) holds if and only if the function
\be \label{L8}
x\ra\frac{\psi(x)[\phi(x)\psi''(x)-\phi''(x)\psi(x)]}{\phi'(x)\psi(x)-\psi'(x)\phi(x)} 
+[\psi'(x)+\psi'(1)] \quad {\rm decreases .}
\ee
\end{lem}
\begin{proof}
We first reformulate the condition (\ref{I4}) at $p=\infty$. Setting $z=\log y, \ y\ge \ve_0,$ and $G(z)=h(y)$, then we have that
\be \label{M8}
G'(z)=yh'(y), \quad G''(z)=y^2h''(y)+yh'(y) \ .
\ee
From (\ref{M8}) we see that (\ref{I4}) at $p=\infty$ is equivalent to the inequality
\be \label{N8}
G(z)G''(z) \ \ge \ G'(z)^2 \quad {\rm or \ } \frac{d^2}{dz^2}\log G(z) \ \ge \ 0 \ .
\ee
Since $z$ is an increasing function of $x$ the inequality (\ref{N8}) is equivalent to showing that
\be \label{O8}
\frac{d}{dx} \frac{G'(z)}{G(z)} \ \ge \ 0 \ .
\ee
From (\ref{I2}) we have that
\be \label{P8}
G(z) \ = \ \frac{f(x)}{\al_0\psi(x)}\left[\psi'(1)\phi(x)-\phi'(1)\psi(x)\right] \ .
\ee
Recalling that $y=f(x)/\al_0$ we have upon differentiating (\ref{P8}) using (\ref{C2}) that
\be \label{Q8}
G'(z) \ = \ \frac{f(x)}{\al_0\psi(x)}\left[\{\psi'(x)+\psi'(1)\}\phi(x)-\{\phi'(x)+\phi'(1)\}\psi(x)\right] \ .
\ee
It follows from (\ref{P8}), (\ref{Q8}) that (\ref{O8}) is equivalent to (\ref{K8}). 

To see (\ref{K8}) we observe from (\ref{M8}) that
\be \label{R8}
\frac{y^2h''(y)}{h(y)-yh'(y)} \ = \frac{G''(z)-G'(z)}{G(z)-G'(z)} \ .
\ee
It follows from (\ref{R8}) that (\ref{S5}) is equivalent to
\be \label{S8}
\frac{d}{dx}  \left[\frac{G''(z)-G'(z)}{G(z)-G'(z)}\right]  \ \le \ 0 \  .
\ee
From (\ref{P8}), (\ref{Q8}) we have that
\be \label{T8}
G(z)-G'(z) \ = \ \frac{f(x)}{\al_0\psi(x)}\left[\phi'(x)\psi(x)-\phi(x)\psi'(x)\right] \ .
\ee
Differentiating (\ref{T8}) using (\ref{C2}) again we have that
\begin{multline} \label{U8}
|\psi'(1)|\left[G''(z)-G'(z)\right] \ =  \ \frac{f(x)}{\al_0\psi(x)}\Big[\psi(x)\{\phi(x)\psi''(x)-\phi''(x)\psi(x)\} \\
+\{\phi'(x)\psi(x)-\psi'(x)\phi(x)\}[\psi'(x)+\psi'(1)] \Big]\ .
\end{multline}
One sees from (\ref{T8}), (\ref{U8}) that (\ref{S8}) and (\ref{L8}) are equivalent.
\end{proof}
Unlike (\ref{D1}), (\ref{E1}), (\ref{I1}) the conditions (\ref{K8}), (\ref{L8}) are not immediately checkable for the functions $\phi(\cdot), \ \psi(\cdot)$ of (\ref{C1}).   However they do hold for these functions provided $0<\al<1$. 
\begin{lem}
Let $\phi(\cdot), \ \psi(\cdot)$ be the functions (\ref{C1}) for some $\al$ with $0<\al<1$. Then (\ref{K8}), (\ref{L8}) hold.
\end{lem}
\begin{proof}
We have that
\be \label{V8}
\psi'(1)\phi(x)-\phi'(1)\psi(x) \ = \ (1-\al)-x^\al+\al x \ , 
\ee
and also
\begin{multline} \label{W8}
 [\phi'(x)+\phi'(1)]\psi(x)- [\psi'(x)+\psi'(1)]\phi(x)  \\
 = \  \al x^{\al-1}-(2-\al)+(2-\al)x^\al-\al x  \ .
\end{multline}
Hence (\ref{K8}) becomes the function
\be \label{X8}
x\ra\frac{\al x^{\al-1}-(2-\al)+(2-\al)x^\al-\al x}{(1-\al)-x^\al+\al x } 
\quad {\rm is \  decreasing}\ .
\ee
In order to prove (\ref{X8}) we need to show that the function
\begin{multline} \label{Y8}
g(x) \ = \ [\al x^{\al-1}-(2-\al)+(2-\al)x^\al-\al x] [-x^{\al-1}+1] \\
- [(1-\al)-x^\al+\al x] [-(1-\al)x^{\al-2}+(2-\al)x^{\al-1}-1] \quad {\rm is \ positive.}
\end{multline}
We have that
\be \label{Z8}
g(x) \ = \ (1-\al)^2x^{\al-2}-x^{2\al-2}+2\al(2-\al) x^{\al-1}-1+(1-\al)^2x^\al \ .
\ee
Note that $g(1)=0$, whence to prove (\ref{Y8}) it is sufficient to show that $g'(x)$ is negative. We have that $g'(x)=-(1-\al)x^{\al-1}g_1(x)$ where
\be \label{AA8}
g_1(x) \ = \ (1-\al)(2-\al)x^{-2}-2x^{\al-2}+2\al(2-\al)x^{-1} \ -\al(1-\al) \ .
\ee
Since $g_1(1)=0$  it is sufficient for us to show that  $g_1'(x)$ is negative. Setting $g_1'(x)=-2(2-\al)x^{-2}g_2(x)$, we need to show that the function
\be \label{AB8}
g_2(x) \ = \  (1-\al)x^{-1}-x^{\al-1}+\al \ 
\ee
is positive. 
Observe that $g_2(1)=0$, whence it is sufficient for us to show that  $g_2'(x)$ is negative. This is clear since $g_2'(x)=-(1-\al)x^{-2}[1-x^\al]$.

To prove (\ref{L8}) we note that
\begin{multline} \label{AC8}
\phi'(x)\psi(x)-\phi(x)\psi'(x) \ = \ \al x^{\al-1}-1+(1-\al)x^\al \ , \\
\phi(x)\psi''(x)-\phi''(x)\psi(x) \ = \ \al(1-\al)[x^{\al-2}-x^{\al-1}] \ .
\end{multline}
From (\ref{AC8}) we see that the function in (\ref{L8}) is given by
\be \label{AD8}
\al(1-\al)\frac{x^{\al-2}-x^{2\al-2}-x^{\al-1}+x^{2\al-1}}{\al x^{\al-1}-1+(1-\al)x^\al}-\al[x^{\al-1}+1] \ .
\ee
We need therefore to show that the function 
\be \label{AE8}
x\ra (1-\al)\frac{x^{-2}-x^{\al-2}-x^{-1}+x^{\al-1}}{\al x^{-1}-x^{-\al}+(1-\al)}-[x^{\al-1}+1]  \quad {\rm is \ decreasing} \ .
\ee
The derivative of the function (\ref{AE8}) is given by
\be \label{AF8}
(1-\al)\left[\frac{g(x)}{[\al x^{-1}-x^{-\al}+(1-\al)]^2}+x^{\al-2}\right]
\ee
where
\begin{multline} \label{AG8}
g(x) \ = \ -\al x^{-4}-4(1-\al)x^{-3}+(2-3\al)x^{-2}+\al(1-\al)x^{\al-4} \\
+(2-3\al+2\al^2)x^{\al-3}-(1-\al)^2x^{\al-2}+(2-\al)x^{-\al-3}-(1-\al) x^{-\al-2} \ .
\end{multline} 
We have that
\be \label{AH8}
g(x)+x^{\al-2}[\al x^{-1}-x^{-\al}+(1-\al)]^2 \ = \ -x^{-\al-3}g_1(x)  \ ,
\ee
where
\begin{multline} \label{AI8}
g_1(x) \ = \ \al x^{\al-1}+2(2-\al)x^{\al}+\al x^{\al+1}-\al x^{2\al-1} \\
-(2-\al)x^{2\al}-(2-\al)-\al x \ .
\end{multline} 
We need to show that $g_1(\cdot)$ is positive.  Since $g_1(1)=0$ we consider the derivative  $g_1'(x)=-\al g_2(x)$, where
\be \label{AJ8}
g_2(x) \ = \ (1-\al)x^{\al-2}-2(2-\al)x^{\al-1}-(\al+1) x^{\al}+(2\al-1) x^{2\al-2} \\
+2(2-\al)x^{2\al-1}+1 \ .
\ee
It is sufficient then to show that $g_2(\cdot)$ is positive. Since $g_2(1)=0$ we may consider the derivative $g'_2(x)=-x^{\al-2}g_3(x)$, where
\begin{multline} \label{AK8}
g_3(x) \ = \ (1-\al)(2-\al)x^{-1}-2(1-\al)(2-\al)+\al(\al+1) x \\
+2(1-\al)(2\al-1) x^{\al-1} 
-2(2-\al)(2\al-1)x^{\al} \ .
\end{multline}
Again it is sufficient to show that $g_3(\cdot)$ is positive. Since $g_3(1)=0$  we consider $g'_3(x)$ given by the formula
\begin{multline} \label{AL8}
g'_3(x) \ = \ -(1-\al)(2-\al)x^{-2}+\al(\al+1)  \\
-2(1-\al)^2(2\al-1) x^{\al-2} 
-2\al(2-\al)(2\al-1)x^{\al-1} \ .
\end{multline}
 Since $g'_3(1)=0$ we evaluate $g''_3(x)$, which is  given by the formula
\begin{multline} \label{AM8}
g''_3(x) \ = \ 2(1-\al)(2-\al)x^{-3}  \\
+2(2-\al)(1-\al)^2(2\al-1) x^{\al-3} 
+2\al(1-\al)(2-\al)(2\al-1)x^{\al-2} \\
= \ 2(1-\al)(2-\al)\left[x^{-3}  +(2\al-1)\{(1-\al)x^{\al-3}+\al x^{\al-2}\}    \right] \ .
\end{multline}
It is evident the RHS of (\ref{AM8}) is non-negative provided $0<\al<1$, whence we conclude that $g_3(\cdot)$ is positive in this case. 
\end{proof}
\begin{rem}
We wish to relate our computations for local asymptotic stability  to those of \cite{nv1}. They define  a function
\be \label{AN8}
H(z) \ = \ \frac{1-x^{1/3}}{x-x^{1/3}+\kappa(1-x^{1/3})}   \ = \ \frac{|\psi'(1)|y}{h(y)+y/p}\  ,
\ee
where $z$ is given by (\ref{AO8}).
We have that
\be \label{AP8}
H'(z) \ = \ |\psi'(1)|\frac{[h(y)-yh'(y)]}{[h(y)+y/p]^2}\frac{dy}{dz} \ .
\ee
It follows from (\ref{AQ8}) that
\be \label{AR8}
\frac{dy}{dz} \ = \ h(y)+y/p \ .
\ee
We then have from (\ref{AP8}), (\ref{AR8})  that
\be \label{AS8}
H'(z) \ = \ |\psi'(1)|\frac{[h(y)-yh'(y)]}{h(y)+y/p} \ .
\ee
Differentiating (\ref{AS8}) with respect to $y$ yields the formula 
\be \label{AT8}
\frac{d}{dy}H'(z) \ = \ -|\psi'(1)|\frac{[h(y)+y/p]yh''(y)+[h'(y)+1/p][h(y)-yh'(y)]}{[h(y)+y/p]^2} \ .
\ee
Since $y$ is an increasing function of $z$ we conclude from (\ref{AT8}) that $H''(\cdot)< 0$   if (\ref{I4}) holds.   The condition $H''(\cdot)< 0$  is required in the proof of local stability for the LSW model in  $\S3$ of \cite{nv1}.  Note that (\ref{I4}) is a slightly stronger condition than the condition $H''(\cdot)< 0$. This is not  surprising since Lemma 4.1 proves that the kernel of the Volterra integral equation decays exponentially at rate $1/p$.  For asymptotic stability all one needs is some exponential decay of the kernel.  
\end{rem}
\begin{proof}[Proof of Theorem 1.1]
We first recall the transformation which converts the system (\ref{U1}), (\ref{V1}) to the system (\ref{A1}), (\ref{B1}). We define the function $w:[0,\infty)\times\R^+\ra\R$ in terms of the solution $c(\cdot,\cdot)$ to (\ref{U1}), (\ref{V1})  by
\be \label{AY8}
w(x,t) \ = \ \int_x^\infty c(x',t) \ dx' \ ,\quad x,t\ge 0 \ .
\ee
Evidently $w(\cdot,\cdot)$ is the solution to the system
\begin{eqnarray}
\frac{\pa w(x,t)}{\pa t}&=&  \Big( 1- \big( x L^{-1}(t)\big)^{\alpha} \Big) \frac{\pa w(x,t)}{\pa x},  \quad x>0 \ , \label{AZ8}\\
\int_0^\infty  w(x,t) dx&=& 1. \label{BA8} 
\end{eqnarray}
Letting $[0,\Ga(t)]$ be the support of $w(\cdot,t), \ t\ge 0,$ we have from (\ref{AZ8}) that
\be \label{BB8}
\frac{d\Ga(t)}{dt} \ = \ \left(\frac{\Ga(t)}{L(t)}\right)^\al -1\ .
\ee
Note that $\Ga(t)>L(t)$ for all $t\ge 0$,  so the function $t\ra\Ga(t)$ is increasing. 
We normalize the support of $w(\cdot,t)$ to the interval $[0,1]$ by making the variable change
\be \label{BC8}
y \ = \ \frac{x}{\Ga(t)} \ , \quad w(x,t) \ = \ \frac{1}{\Ga(t)}\tilde{w}(y,s(t)) \ ,
\ee
where the function $t\ra s(t)$ is to be determined. 
It follows from (\ref{BA8}) that
\be \label{BD8}
\int_0^1\tilde{w}(y,s) \ dy \ = \ 1 \ .
\ee
From (\ref{BC8}) we have that
\be \label{BE8}
\frac{\pa w(x,t)}{\pa x} \ = \ \frac{1}{\Ga(t)^2}\frac{\pa \tilde{w}(y,s(t))}{\pa y} \ .
\ee
We also have that
\begin{multline} \label{BF8}
\frac{\pa w(x,t)}{\pa t} \ = \ \frac{1}{\Ga(t)}\frac{\pa \tilde{w}(y,s(t))}{\pa t} 
-\frac{1}{\Ga(t)^2}\frac{d\Ga(t)}{dt}\left[y\frac{\pa \tilde{w}(y,s(t))}{\pa y}+\tilde{w}(y,s(t))\right] \\
= \ \frac{\Ga'(t)}{\Ga(t)^2}\left[\frac{\Ga(t)}{\Ga'(t)}\frac{\pa \tilde{w}(y,s(t))}{\pa t} -
y\frac{\pa \tilde{w}(y,s(t))}{\pa y}-\tilde{w}(y,s(t))\right] \ .
\end{multline}
From (\ref{BE8}), (\ref{BF8}) we see that (\ref{AZ8}) becomes
\begin{multline} \label{BG8}
\frac{\Ga(t)}{\Ga'(t)}\frac{\pa \tilde{w}(y,s(t))}{\pa t} -
y\frac{\pa \tilde{w}(y,s(t))}{\pa y}-\tilde{w}(y,s(t)) \\
+\frac{(\Ga(t)y/L(t))^\al-1}{(\Ga(t)/L(t))^\al-1} \  \frac{\pa \tilde{w}(y,s(t))}{\pa y} \ = \ 0 \ .
\end{multline}
If we set 
\be \label{BH8}
\kappa(t) \ = \ \frac{1}{(\Ga(t)/L(t))^\al-1} \ ,
\ee
then (\ref{BG8}) becomes
\be \label{BI8}
\frac{\Ga(t)}{\Ga'(t)}\frac{\pa \tilde{w}(y,s(t))}{\pa t} +[(y^\al-y)-\kappa(t)(1-y^\al)] \frac{\pa \tilde{w}(y,s(t))}{\pa y} \ = \ \tilde{w}(y,s(t)) \ .
\ee
Now by a change of time variable $s(t)=\log\Ga(t)$ we can normalize the coefficient of the time derivative in (\ref{BI8}) to $1$.
Evidently the system (\ref{BD8}), (\ref{BI8}) with the time variable $s$ is the same as (\ref{A1}), (\ref{B1}) with $\phi(\cdot), \ \psi(\cdot)$ given by (\ref{C1}). 

For $s\ge 0$  let $\tilde{\beta}(\cdot,s),$ be the beta function corresponding to $\tilde{w}(\cdot,s)$, as given by the formula (\ref{G1}).  Then from (\ref{BC8}) we have that
\be \label{BJ8}
\beta_{X_t}(x) \ = \ \tilde{\beta}\left(\frac{x}{\Ga(t)},\log\Ga(t)\right) \ , \quad 0\le x<\Ga(t), \ \ t\ge 0 \ .
\ee
It follows from (\ref{X1}), (\ref{BJ8}) that
\be \label{BK8}
\frac{d}{dt}\langle X_t\rangle \ = \   \tilde{\beta}\left(0,\log\Ga(t)\right) \ ,  \quad t\ge 0 \ .
\ee
We conclude from  (\ref{A8}) of Theorem 8.1 that
\be \label{BL8}
\lim_{t\ra\infty} \frac{d}{dt}\langle X_t\rangle \ = \ \beta_{\mathcal{X}_\beta}(0) \ .
\ee
Next observe that
\be \label{BM8}
\frac{\langle X^\al_t\rangle}{\langle X_t\rangle^\al} \ = \ \al\int_0^1 y^{\al-1}\tilde{w}(y,\log\Ga(t)) \ dy \Big/ \tilde{w}(0,\log\Ga(t))^{(1-\al)} \ .
\ee
It follows from (\ref{AI2}), (\ref{A8}), (\ref{BD8}), (\ref{BM8}) that
\be \label{BN8}
\lim_{t\ra\infty} \frac{\langle X^\al_t\rangle}{\langle X_t\rangle^\al} \ = \ \frac{\langle \mathcal{X}^\al_\beta\rangle}{\langle \mathcal{X}_\beta\rangle^\al}  \ .
\ee
We also have from (\ref{A8}), (\ref{BH8}) that
\be \label{BO8}
\lim_{t\ra\infty} \frac{\Ga(t)^\al}{\langle X^\al_t\rangle} \ = \ 1+\frac{1}{\kappa} \ = \ \frac{\|\mathcal{X}_\beta\|_\infty^\al}{\langle \mathcal{X}_\beta^\al\rangle}   \ .
\ee
Evidently (\ref{AI2}), (\ref{A8}) imply that
\be \label{BP8}
\lim_{s\ra\infty}\frac{\tilde{w}(y,s)}{\tilde{w}(0,s)} \ = \  P\left(\frac{\mathcal{X}_\beta}{\|\mathcal{X}_\beta\|_\infty}>y\right) \ , \quad 0\le y<1 \ .
\ee
We also have that
\be \label{BQ8}
P\left(\frac{X_t}{\langle X_t\rangle}>x\right) \ = \ P\left(\frac{X_t}{\Ga(t)}>x\frac{\langle X_t\rangle}{\Ga(t)}\right) \ = \frac{\tilde{w}(x\langle X_t\rangle/\Ga(t),\log\Ga(t))}{\tilde{w}(0,\log\Ga(t))} \ .
\ee
We conclude from (\ref{BN8})-(\ref{BQ8}) that
\be \label{BR8}
\lim_{t\ra\infty} P\left(\frac{X_t}{\langle X_t\rangle}>x\right) \ = \  P(\mathcal{X}_\beta>x) \ , \quad 0\le x<\infty \ .
\ee
Evidently (\ref{AC1}) follows from (\ref{BL8}), (\ref{BR8}). 

We obtain the rate of convergence results in Theorem 1.1 in the case when $x\ra\beta_{X_0}(x)$ is H\"{o}lder continuous at $x=\|X_0\|_\infty$ by applying (\ref{A*8}). From (\ref{A*8}), (\ref{BK8}) we have that
\be \label{BS8}
\left|\frac{d}{dt}\langle X_t\rangle - \beta_{\mathcal{X}_\beta}(0)\right| \ \le \ C\exp[-\nu\log\Ga(t)] \ \le \frac{C'}{1+t^\nu} \ , \quad t\ge 0 \ .
\ee
for some constant $C'$, whence (\ref{AD1}) holds. Similarly we have from (\ref{AI2}), (\ref{A*8}) that for any $\del$ with $0<\del<1$,  there is a constant $C_\del>0$ such that
\begin{multline} \label{BT8}
 P\left(\frac{X_t}{\langle X_t\rangle}>x\right) \ \le \ P\left(\mathcal{X}_\beta>\frac{x}{1+C_\del/(1+ t^\nu)}\right) \ , \\
 {\rm for \ } 0\le x\le (1-\del)\|\mathcal{X}_\beta\|_\infty \ , \ t\ge 0 \ .
 \end{multline}
 Now (\ref{BS8}) and (\ref{BT8}) (with the corresponding lower bound in addition) proves (\ref{AE1}).

\end{proof}

\end{document}